\documentclass[a4paper,12pt]{article}
\usepackage[OT2,T1]{fontenc}
\DeclareSymbolFont{cyrletters}{OT2}{wncyr}{m}{n}
\DeclareMathSymbol{\Sha}{\mathalpha}{cyrletters}{"58}
\usepackage{latexsym}
\usepackage{amsfonts, amsthm}
\usepackage[all,cmtip]{xy}
\usepackage{rotating}
\usepackage{amsmath}
\usepackage{amssymb}
\usepackage{scalefnt}
\usepackage{amscd}
\usepackage{enumerate}
\usepackage[frenchb,english]{babel}
\usepackage[T1]{fontenc}
\usepackage{makeidx}
\usepackage[utf8]{inputenc}
\usepackage{fancyhdr}
\usepackage{titlesec}
\usepackage[font={footnotesize},aboveskip=0pt]{caption}
\usepackage{mathrsfs}
\usepackage{url}
\usepackage{upgreek}
\usepackage[active]{srcltx}
\def\leftnote#1{\vadjust{\setbox1=\vtop{\hsize 35mm\parindent=0pt\bf\baselineskip=9pt\rightskip=4mm plus 4mm#1}\hbox{\kern-25mm\smash{\box1}}}}
\titleformat{\subsubsection}[runin]{\normalfont\bfseries}{\thesubsubsection.}{3pt}{}
\newtheorem{ithm}{Theorem}

\newtheorem{theorem}{Theorem}[section]
\newtheorem{lemma}[theorem]{Lemma}
\newtheorem{corollary}[theorem]{Corollary}
\newtheorem{proposition}[theorem]{Proposition}

\theoremstyle{remark}
\newtheorem{remark}[theorem]{Remark}

\theoremstyle{definition}

\newcommand{\OO}{\mathcal O}
\newcommand{\SSS}{\mathcal S}
\newcommand{\RR}{\mathcal R}
\newcommand{\A}{\mathcal A}
\newcommand{\I}{\mathcal I}
\newcommand{\J}{\mathcal J}

\newcommand{\C}{\mathcal C}
\newcommand{\M}{\mathcal M}
\newcommand{\UU}{\mathcal U}
\newcommand{\ZZ}{\mathcal Z}
\newcommand{\HH}{\mathscr H}
\newcommand{\CC}{\mathscr C}
\newcommand{\oo}{\mathfrak o}
\newcommand{\pp}{\mathfrak p}
\newcommand{\qq}{\mathfrak q}
\newcommand{\LL}{\mathfrak L}
\newcommand{\RRR}{\mathfrak R}
\newcommand{\T}{\mathcal T}
\newcommand{\PP}{\mathcal P}
\newcommand{\Z}{\mathbb Z}
\newcommand{\Q}{\mathbb Q}
\newcommand{\F}{\mathbb F}
\newcommand{\R}{\mathbb R}

\newcommand{\Det}{\mathrm{Det}}
\newcommand{\Gal}{\mathrm{Gal}}
\newcommand{\Tr}{\mathrm{Tr}}
\newcommand{\Hom}{\mathrm{Hom}}
\newcommand{\Div}{\mathrm{Div}}
\newcommand{\Ram}{\mathrm{Ram}}
\newcommand{\Map}{\mathrm{Map}}
\newcommand{\Cl}{\mathrm{Cl}}
\newcommand{\val}{\mathrm{val}}

\begin{document}
\bibliographystyle{plain}
\title{\textsc{Gauss sums, Jacobi sums and cyclotomic units related to torsion Galois modules}}
\author{Luca Caputo and St\'ephane Vinatier}
\maketitle
\abstract{Let $G$ be a finite group and let $N/E$ be a tamely ramified $G$-Galois extension of number fields. We show how Stickelberger's factorization of Gauss sums can be used to determine the stable isomorphism class of various arithmetic $\Z[G]$-modules attached to $N/E$. If $\OO_N$ and $\OO_E$ denote the rings of integers of $N$ and $E$ respectively, we get in particular that $\OO_N\otimes_{\OO_E}\OO_N$ defines the trivial class in the class group $\Cl(\Z[G])$ and, if $N/E$ is also assumed to be locally abelian, that the square root of the inverse different (whenever it exists) defines the same class as $\OO_N$. These results are obtained through the study of the Fr\"ohlich representatives of the classes of some torsion modules, which are independently introduced in the setting of cyclotomic number fields. Gauss and Jacobi sums, together with the Hasse-Davenport formula, are involved in this study. These techniques are also applied to recover the stable self-duality of $\OO_N$ (as a $\Z[G]$-module). Finally, when $G$ is the binary tetrahedral group, we use our results in conjunction with Taylor's theorem to find a tame $G$-Galois extension whose square root of the inverse different has nontrivial class in $\Cl(\Z[G])$.
} 


\section{Introduction}

Let $G$ be a finite group and let $N/E$ be a Galois extension of number fields with Galois group $G$. Let $\OO_N$ and $\OO_E$ denote the rings of integers of $N$ and $E$, respectively; then $\OO_N$ is an $\OO_E[G]$-module, and in particular a $\Z[G]$-module, whose structure has long been studied. In the case where $N/E$ is a tame extension, $\OO_N$ is known to be locally free by Noether's theorem and the investigation of its Galois module structure culminated with M. Taylor's theorem \cite{tay2} expressing the class $(\OO_N)$ defined by $\OO_N$ in the class group of locally free $\Z[G]$-modules $\Cl(\Z[G])$ in terms of Artin root numbers (see Theorem \ref{pearl}). 

Other $\OO_E[G]$-modules which appear naturally in our context have also been studied, among which the inverse different $\C_{N/E}$ of the extension and, when it exists, its square root $\A_{N/E}$. The existence of $\A_{N/E}$ is of course equivalent to $\C_{N/E}$ being a square, a condition which can be tested using Hilbert's valuation formula \cite[IV, Proposition 4]{Serre}. In particular, when $N/E$ is tame, $\A_{N/E}$ exists if and only if the inertia group of every prime of $\OO_E$ in $N/E$ has odd order (which is the case for example if the degree $[N:E]$ is odd). In the tame case, both these modules are locally free $\Z[G]$-modules by a result of S. Ullom in \cite{Ullom-cohomology} (applying to any $G$-stable fractional ideal of $N$).

The modules $\OO_N$ and $\C_{N/E}$ are related by duality. For a fractional ideal $I$ of $\OO_N$, the dual of $I$ with respect to the trace $\Tr_{N/E}$ from $N$ to $E$ is the fractional ideal
$$I^\#=\{x\in N\,\mid\, \Tr_{N/E}(xI)\subseteq \OO_E\}\enspace.$$
Then $I^\#$ is $G$-isomorphic to the $\mathcal{O}_E$-dual of $I$, namely $I^\#\cong \mathrm{Hom}_{\OO_E}(I,\OO_E)$, and by definition one has $\C_{N/E}=\OO_N^\#$, namely $\OO_N$ and $\C_{N/E}$ are dual of each other. It can be shown that, for any fractional ideal $I$ of $\OO_N$, one has $I^\#=\C_{N/E} I^{-1}$, which implies that $\A_{N/E}$, when it exists, is the only self-dual fractional ideal of $\OO_E$. 

The duality relation between $\C_{N/E}$ and $\OO_N$ accounts for comparing their $\Z[G]$-module structures. In the tame case, it essentially amounts to comparing their classes $(\C_{N/E})$ and $(\OO_N)$ in $\Cl(\Z[G])$, and A. Fr\"ohlich conjectured that $\OO_N$ is stably self-dual, namely that
\begin{equation}\label{stably-self-dual}
(\OO_N)=(\C_{N/E})\enspace,
\end{equation}
equality which was later proved by M. Taylor \cite{tay1} under slightly stronger hypotheses and by S. Chase \cite{ChaseTors} in full generality. Taylor's proof uses Fr\"ohlich's Hom-description of $\Cl(\mathbb{Z}[G])$, while Chase examines the torsion module 
$$\T_{N/E}=\C_{N/E}/\OO_N\enspace.$$
It is worth noting that both proofs make crucial use of the stable freeness of the Swan modules of cyclic groups, a result due to R. Swan \cite{SwanPRFG}. 

The study of the Galois module structure of $\A_{N/E}$ was initiated by B. Erez, who proved in particular that, when $N/E$ is tamely ramified and of odd degree, the class $(\A_{N/E})$ it defines in $\Cl(\Z[G])$ is trivial, see \cite{Erez2}. His proof follows the same strategy as that of Taylor in \cite{tay2}, using in particular Fr\"ohlich's Hom-description of $\Cl(\mathbb{Z}[G])$. Since Taylor's theorem implies the triviality of $(\OO_N)$ when $N/E$ is of odd degree, we get:
\begin{equation}\label{equality}
\textrm{if $[N:E]$ is odd, then }(\OO_N)=(\A_{N/E})
\end{equation}
and both classes are in fact trivial.

In this paper, we consider a tame $G$-Galois extension $N/E$ such that the square root of the inverse different $\A_{N/E}$ exists. We introduce the torsion module
$$\SSS_{N/E}=\A_{N/E}/\OO_N\enspace,$$  
check it to be $G$-cohomologically trivial, hence to define a class in $\Cl(\Z[G])$, and show that this class is trivial, yielding a new proof of Equality \eqref{equality} without any assumption on the degree of the extension. Nevertheless we need to assume that $N/E$ is locally abelian, namely that the decomposition group of every prime ideal of $\OO_E$ is abelian (note that this condition is automatically satisfied at unramified primes). 
The proofs of Erez \cite{Erez2} and Taylor \cite{tay2} involve the study of Galois Gauss sums, which may be regarded as objects of analytic nature since they come from the functional equation of Artin $L$-functions. In our work instead only classical Gauss and Jacobi sums are involved, which might seem more satisfactory if one considers that (\ref{equality}) is a purely algebraic statement.

It turns out that the strategy of our proof also applies, without any restriction on the decomposition groups, to the torsion module $\T_{N/E}$ defined above as well as to another torsion module, denoted $\RR_{N/E}$. The module $\RR_{N/E}$, whose definition will be given in Section \ref{section:reduction-to-totally-ramified}, was introduced and studied by Chase in \cite{ChaseTors}. We will prove that both $\T_{N/E}$ and $\RR_{N/E}$ define the trivial class in $\Cl(\Z[G])$. This allows us on the one hand to recover Equality \eqref{stably-self-dual}, and on the other hand to deduce that the $\Z[G]$-module $\OO_N\otimes_{\OO_E}\OO_N$ defines the trivial class in $\Cl(\Z[G])$. This result looked new to us at first sight, but we show in Proposition \ref{ONn} that $(\OO_N\otimes_{\OO_E}\OO_N)=(\OO_N)^{[N:E]}$ in $\Cl(\Z[G])$, hence it is easily deduced from Taylor's theorem.

These results are stated in Theorem \ref{maincor} below. After a reduction to the case of a local totally ramified tame Galois extension with Galois group $\Delta$, we introduce new torsion Galois modules, which define classes $(R)$ and $(S)$ in $\Cl(\Z[\Delta])$. The study of these classes is the core of our work. In Theorem \ref{main} we give explicit expressions of representative morphisms of $(R)$ and $(S)$ in terms of Gauss and Jacobi sums, which through Fr\"ohlich's Hom-description of $\Cl(\Z[\Delta])$ yield the triviality of $(R)$ and $(S)$. Theorem \ref{surprise} will show a surprising consequence of the generalisation of Equality \eqref{equality} to even degree extensions.
\medskip

As suggested by various experts, it seems reasonable to expect that our methods also work in the context of sums of $k$-th roots of codifferents developed by Burns and Chinburg \cite{Burns-Chinburg}.
\medskip

We now explain our strategy more in detail. In Section \ref{section:reduction-to-totally-ramified} we follow Chase's approach: we consider the torsion modules $\T_{N/E}$, $\SSS_{N/E}$ and $\RR_{N/E}\,$, then we
reduce to the study, for every prime $\PP$ of $N$, of their $\Z[I_\PP]$-module structure, where $I_\PP$ is the inertia group at $\PP$ (note that $I_\PP$ is cyclic of order coprime to the residual characteristic of $\PP$, since $N/E$ is tame). In other words we show that it is sufficient to prove the triviality of the classes  in $\Cl(\Z[\Delta])$ of the local analogues $\T_{K/F}$, $\SSS_{K/F}$ and $\RR_{K/F}$, where $K/F$ is a cyclic totally ramified tame $p$-adic extension with Galois group $\Delta$ of order $e$ (prime to $p$). Up to this point, the only difference with Chase's study of $\T_{N/E}$ is that
we need $N/E$ to be locally abelian when dealing with $\SSS_{N/E}$. 

It is interesting to remark that $\T_{N/E}$ and $\SSS_{N/E}$ are particular instances of torsion modules arising from ideals, \textit{i.e.} modules of the form $\OO_N/\I$ or $\I^{-1}/\OO_N$, where $\I$ is a $G$-stable ideal of $\OO_N$. In fact we will perform the reduction step of Section \ref{section:reduction-to-totally-ramified} for general torsion modules arising from ideals (under the assumption that $N/E$ is locally abelian). Working in this more general situation requires no additional effort and allows us to easily recover the cases of $\T_{N/E}$ and $\SSS_{N/E}$ (while $\RR_{N/E}$ needs a somehow separate treatment). Moreover we will get almost for free a new proof of a result of Burns \cite[Theorem 1.1]{BurnsARC} on arithmetically realisable classes in tame locally abelian Galois extensions of number fields (see Subsection \ref{section:arcburns}). Nonetheless, for the sake of simplicity, in this introduction we shall stick with the modules $\T_{N/E}$, $\SSS_{N/E}$ and $\RR_{N/E}$.

The key ingredient in Chase's proof of the triviality of $\T_{N/E}$ is the link he establishes between the local torsion module $\T_{K/F}$ and the Swan module $\Sigma_\Delta(p)=p\Z[\Delta]+\Tr_\Delta\Z[\Delta]$ (here $\Tr_\Delta=\sum_{\delta\in\Delta}\delta\in\Z[\Delta]$). Consider the torsion module $T(p,\Z[\Delta])=\Z[\Delta]/\Sigma_\Delta(p)$ associated to the Swan module, then
$$\T_{K/F}\cong T(p,\Z[\Delta])^{f_F}$$ 
as $\Z[\Delta]$-modules, where $f_F$ is the inertia degree of $F/\Q_p$.

To deal with $\SSS_{K/F}$ and $\RR_{K/F}$, we need to enlarge the coefficient ring of the group algebra of $\Delta$. Let $\mu_e$ denote the group of $e$th roots of unity in a fixed algebraic closure $\overline\Q$ of $\Q$ and let $\oo$ denote the ring of integers of the cyclotomic field $\Q(\mu_e)$. We introduce new torsion $\Z[\Delta]$-modules $S_\chi(\pp,\oo[\Delta])$ and $R_\chi(\pp,\oo[\Delta])$, which depend on a character $\chi:\Delta\to\mu_e$ and a prime $\pp$ of $\oo$ not dividing $p$. They may be considered as analogues of $T(p,\Z[\Delta])$ in the sense that, for an appropriate choice of $\chi$ and $\pp$ (in particular $\pp\mid p$, the residual characteristic of $K/F$), we have
$$\SSS_{K/F}\cong S_\chi(\pp,\oo[\Delta])^{f_F/f},\quad\RR_{K/F}\cong R_\chi(\pp,\oo[\Delta])^{f_F/f},$$  
as $\Z[\Delta]$-modules, where $f$ is the inertia degree of $\pp$ in $\Q(\mu_e)/\Q$.
These constructions are described in Section \ref{section:back-to-global}. Furthermore it is easy to see that both $S_\chi(\pp,\oo[\Delta])$ and $R_\chi(\pp,\oo[\Delta])$ are cohomologically trivial $\Z[\Delta]$-modules, hence they define classes in $\Cl(\Z[\Delta])$ (see Section \ref{ctsec}).  

Now let $\chi:\Delta\to\mu_e$ be any injective character and $\pp$ be a prime of $\oo$ not dividing $e$. Set $R=R_\chi(\pp,\oo[\Delta])$ and $S=S_\chi(\pp,\oo[\Delta])$. The core of this paper is the study of the classes $(R)$ and $(S)$ in $\Cl(\Z[\Delta])$. In Section \ref{section:hom-des}, we find equivariant morphisms $r$ and $s$ from the group of virtual characters of $\Delta$ to the idèles of $\Q(\mu_e)$, representing $(R)$ and $(S)$ respectively in Fr\"ohlich's Hom-description of $\Cl(\Z[\Delta])$. In Section \ref{section:stickelberger}, we note that the contents of the values of $r$ and $s$ are principal ideals: this follows from Stickelberger's theorem, which also gives explicit generators of these ideals in terms of Gauss and Jacobi sums respectively. Dividing $r$ and $s$ by suitably modified generators $c_r$ and $c_s$ of their contents yields morphisms with unit idelic values. With the help of the Hasse-Davenport formula we express the resulting morphisms $rc_r^{-1}$ and $sc_s^{-1}$ as Fr\"ohlich's generalized Determinants of some unit idèles of $\Z[\Delta]$, showing that they lie in the denominator of the Hom-description, namely that $(R)$ and $(S)$ are trivial.

It is worth noting that applying these techniques to $T_\mathbb{Z}=T(p,\Z[\Delta])$ yields a proof of the triviality of $(T_\mathbb{Z})$ in $\Cl(\Z[\Delta])$. In this proof, as in the original proof by Swan, cyclotomic units play a central role, analogous to that of the Gauss and Jacobi sums above, as it will soon be apparent. 
\medskip

To state our main results, we let $G$ and $J$ denote the Gauss and Jacobi sums defined respectively by 
$$G=\sum_{x\in\oo/\pp}\left(\frac{x}{\pp}\right)^{-1}\xi^{\mathrm{Tr}(x)}\ ,\quad J=\sum_{x\in\oo/\pp}\left(\frac{x}{\pp}\right)^{-1}\left(\frac{1-x}{\pp}\right)^{-1}\enspace,$$
where  $\mathrm{Tr}:\oo/\pp\to \Z/p\Z$ denotes
the residue field trace homomorphism, $\xi$ is a fixed $p$th root of unity in $\overline\Q$ and $\left(\frac{}{\pp}\right)$ is the $e$th power residue symbol.
Then both $G^e$ and $J$ belong to $\oo$. Let $\delta$ be a fixed generator of $\Delta$. In Section \ref{section:stickelberger} we define $m_i,n_i\in\Z$ for $i=0,\,\ldots\,e-1$ such that  
$$G^e=\sum_{i=0}^{e-1}m_i\chi(\delta)^i\ ,\quad J=\sum_{i=0}^{e-1}n_i\chi(\delta)^i\enspace.$$
The corresponding expression for the cyclotomic unit $C$ is:
$$C=\frac{\chi(\delta)^p-1}{\chi(\delta)-1}=\sum_{i=0}^{p-1}\chi(\delta)^i\enspace.$$

We can now formulate the main result of this paper, in terms of Fr\"ohlich's Hom-description of the class group (see Section \ref{hdcg} for more details).

\begin{ithm}\label{main}
The classes $(T_\mathbb{Z})$, $(R)$ and $(S)$ are trivial in $\Cl(\Z[\Delta])$. More precisely they are represented in $\mathrm{Hom}_{\Omega_{\Q}}(R_{\Delta},J(\Q(\mu_e)))$ by the morphisms with $\qq$-components equal to $1$ at places $\qq$ of $\Q(\mu_e)$ such that $\qq\nmid e$, and to 
$$\Det(p^{-1}u_t)\ ,\quad\Det(u_r^{-1})\ ,\quad \Det(u_s^{-1})\enspace,$$
respectively, at prime ideals $\qq$ of $\oo$ such that $\qq\mid e$, where $u_t,u_r,u_s\in\Z[\Delta]$ are defined by
$$u_t=\sum_{i=0}^{p-1}\delta^i\ ,\quad u_r=\sum_{i=0}^{e-1}m_i\delta^i\ ,\quad u_s=\sum_{i=0}^{e-1}n_i\delta^i\enspace,$$
and satisfy $u_t,u_r,u_s\in\Z_q[\Delta]^\times$, for any rational prime $q$ such that $q\mid e$.
\end{ithm}
The triviality of $(T_{\Z})$ is essentially Swan's theorem \cite[Corollary 6.1]{SwanPRFG}. Our proof using Fr\"ohlich's Hom-description of the class group, given in Section \ref{hom-rep-TZ}, might be folklore, nevertheless the representative morphism we use to describe $(T_{\Z})$ is slightly different from the one given in \cite[(I.2.23)]{Frohlich-Alg_numb} (see Remark \ref{t-rep}). The proof for $S$ and $R$ will follow from the results of Sections \ref{section:hom-des} and \ref{section:stickelberger} and will enlight the analogy between the torsion modules $R$ and $S$ we introduce and $T_\Z$ (see Theorem \ref{desiderio} and its proof).


In his early work \cite{Erez} on the square root of the inverse different, Erez already establishes a link between this module and a Jacobi sum, in the case of a cyclic extension of $K/\Q$ of odd prime degree $l$. Using a result of Ullom \cite[Theorem 1]{Ullom-representations}, he shows that the class of $\A_{K/\Q}$ in $\Cl(\Z[\Gal(K/\Q)])$ corresponds, through Rim's isomorphism \cite[(42.16)]{CR2}, to the class in $\Cl(\Q(\mu_l))$ of an ideal of $\Q(\mu_l)$ which is defined by the action of some specific element of $\Z[\Gal(\Q(\mu_l)/\Q)]$ on prime ideals. This group algebra element, which is pretty close to what ours, $u_s$, would yield in the prime degree case (see \S\ref{further}), happens to be the exponent which appears in the factorisation of the Jacobi sum, yielding the triviality of the class under study.

The factorisation of Gauss and Jacobi sums through Stickelberger's theorem is also an essential step in our proof of the triviality of the classes of $R$ and $S$. In the introduction of his thesis \cite{Erez-thesis}, Erez gives other examples of results on the Galois module structure of rings of integers that are proved using Stickelberger elements (\cite{Martinet-diedral}, \cite{Cougnard}). In \cite{ChaseTors} Chase uses such elements to give the composition series of certain primary components of the $\OO_E[G]$-module $\RR_{N/E}$ when $N/E$ is cyclic or elementary abelian. Reversely, Galois module structure results can be used to build Stickelberger elements: the first statement in the theory of integral Galois modules is due to Hilbert (for the ring of integers of an abelian extension of $\Q$ with discriminant coprime to the degree), who used it to construct annihilating elements for the class group of some cyclotomic extensions. Finally, Stickelberger's elements also have a fundamental role in the work of L. McCulloh (see \cite{McCullohCNFEAGR}, \cite{McCullohGMSAE} and the references given in those papers).

As pointed out by Fröhlich in \cite[Note 6 to Chapter III]{Frohlich-Alg_numb}, Chase remarks that the character function associated to $\T_{N/E}$ (in the ideal-theoretic Hom-description) is the Artin conductor. Similarly, the character function associated to $\RR_{N/E}$ (in the idèle-theoretic Hom-description) is closely related to a resolvent map (see \cite[p. 210]{ChaseTors}). Fröhlich also wonders how his method, based on the comparison of Galois Gauss sums and resolvents, could be linked to Chase's approach. We hope that the explicit connection we establish between $\RR_{K/F}$ and a Gauss sum might be useful to answer Fröhlich's question.\smallskip

Using the reduction results of Section \ref{section:reduction-to-totally-ramified}, we deduce the following consequence.

\begin{ithm}\label{maincor}
Let $N/E$ be a $G$-Galois tamely ramified extension of number fields. Then the classes of $\T_{N/E}$ and $\RR_{N/E}$ are trivial in $\Cl(\Z[G])$. In particular we have  
$$(\OO_N)=(\C_{N/E})\quad\textrm{and}\quad (\OO_N\otimes_{\OO_E}\OO_N)=(\OO_N)^{[N:E]}=1\enspace.$$ 

If, further, $N/E$ is locally abelian, then the class of $\SSS_{N/E}$ is trivial in $\Cl(\Z[G])$. In particular we have 
$$(\OO_N)=(\A_{N/E})\enspace,$$
thus $\OO_N$, $\C_{N/E}$ and $\A_{N/E}$ define the same class in $\Cl(\Z[G])$. 
\end{ithm} 

As already mentioned, the triviality of $(\T_{N/E})$ in this context is due to Chase, see \cite[Corollary 1.12]{ChaseTors}. In the same paper he studies the torsion module $\RR_{N/E}$ as an $\OO_E[G]$-module, showing that it determines the local Galois module structure of $\OO_N$ \cite[Theorem 2.10]{ChaseTors} and describing its primary components \cite[Theorem 2.15]{ChaseTors}. We will show that $(\RR_{N/E})=1$, which is in fact equivalent to $(\OO_N\otimes_{\OO_E}\OO_N)=(\OO_N)^{[N:E]}=1$ (see \S\ref{ctglob}). The equality $(\OO_N)^{[N:E]}=1$ follows of course from Taylor's theorem (which implies that $(\OO_N)$ is of order dividing $2$ and is trivial in odd degree) but was also known before Taylor's proof of Fröhlich's conjecture (it follows from \cite[Corollary]{TaylorGMSIRA} using the last remark of \cite[Definition 1.6]{LamAEFG}). We make more comments on the last equality of Theorem \ref{maincor} in the presentation of Section \ref{analytic} below.
\medskip

We end this paper by focusing in Section \ref{analytic} on the class of the square root of the inverse different  (for the rest of this introduction we shall assume that the inverse different of every extension we consider is a square). We briefly recall some known results on this matter. In odd degree, Erez showed that $(\A_{N/E})\in \Cl(\Z[G])$ is defined (\textit{i.e.} $\A_{N/E}$ is $\Z[G]$-locally free) if and only if $N/E$ is a weakly ramified $G$-Galois extension (\textit{i.e} the second ramification group of every prime is trivial). It seems reasonable to conjecture that $(\A_{N/E})$ is trivial when $N/E$ is a weakly ramified Galois extension of odd degree (see also \cite[Conjecture]{Vinatier-Surla}). As already mentioned this conjecture is true in the tame case, thanks to a result of Erez. When $N/E$ is a weakly ramified $G$-Galois extension of odd degree, one also knows that 
\begin{itemize}
\item $\M\otimes_{\Z[G]}\A_{N/E}$ is free over $\M$, where $\M$ is a maximal order of $\Q[G]$ containing $\Z[G]$ (\cite[Theorem 2]{Erez2});
\item $(\A_{N/E})=1$ if, for any wildly ramified prime $\mathcal P$ of $\OO_N$, the decomposition group is abelian, the inertia group is cyclic and the localized extension $E_P/\Q_p$ is unramified, where $P=\mathcal P\cap E$ and $p\Z=\mathcal P\cap\Q$ (\cite[Theorem 1]{Pickettvinatier});
\item $(\A_{N/\Q})^e=1$ if $[N:\Q]$ is a power of a prime $p$ and $e$ is the ramification index of $p$ in $N/\Q$ (\cite[Th\'eor\`eme 1]{Vinatier-Surla}); when $p=3$ one even has $(\A_{N/\Q})^3=1$ by \cite[Theorem 1]{Vinatier-3}.
\end{itemize}
It is interesting to observe that so far 
the class of $\A_{N/E}$ had only been studied when $N/E$ has odd degree (except for a short local study, due to Burns and Erez, in the absolute, abelian and very wildly ramified case, see \cite[\S 3]{Erez-survey}). Specifically the question of whether it is trivial or not for \textit{every} tame Galois extension $N/E$ had not been considered. The reason for this restriction is that the second Adams operation, which is fundamental in Erez's approach in \cite{Erez2}, does not behave well with respect to induction for groups of even order. 

This is precisely the point where our result above brings new information.
When $N/E$ is locally abelian and tame, Theorem \ref{maincor} together with Taylor's theorem reduces the question to the study of the triviality of the image of the root number class in $\Cl(\Z[G])$ (see Corollary \ref{A=tW}). This immediately gives that $(\A_{N/E})=1$ if $N/E$ is abelian or has odd order (thus recovering Erez's result in the locally abelian case). Computing the appropriate root numbers, we show next that $(\A_{N/E})=1$ if $N/E$ is locally abelian and no real place of $E$ becomes complex in $N$. However the class of the square root of the inverse different is not trivial in general, in other words we have the following result.

\begin{ithm}\label{surprise}
There exists a tame Galois extension $N/\Q$ of even degree such that $\C_{N/\Q}$ is a square and the class of $\A_{N/\Q}$ is nontrivial in $\Cl(\Z[\mathrm{Gal}(N/\Q)])$.  
\end{ithm}

In fact, in Section \ref{explicitexample} we will explicitly describe a tame locally abelian $\tilde{A}_4$-Galois extension $N/\Q$, taken from \cite{BachocKwon}, such that $(\A_{N/\Q})\ne 1$ in $\Cl(\Z[\tilde{A}_4])$, where $\tilde{A}_4$ is the binary tetrahedral group (which has order $24$). This is, to our knowledge, the first example of a tame Galois extension of number fields whose square root of the inverse different (exists and) has nontrivial class. We also show that this example is minimal in the sense that $(\A_{N/\Q})=1$ if $N/\Q$ is a tame locally abelian $G$-Galois extension with $\#G\leq 24$ and $G\ne\tilde{A}_4$.

\section{Reduction to inertia subgroups}\label{section:reduction-to-totally-ramified}
We use the notation of the Introduction, in particular $N/E$ is a tame $G$-Galois extension. For a prime $\PP$ of $N$, we denote by $I_\PP$ (resp. $D_\PP$) the inertia subgroup (resp. decomposition subgroup) of $\PP$ in $N/E$. If $P$ is the prime ideal of $E$ below $\PP$, we will often identify the Galois group of $N_\PP/E_P$ with $D_\PP$, where $N_\PP$ (resp. $E_P$) is the completion of $N$ at $\PP$ (resp. of $E$ at $P$). In this section we show that the $G$-module structure of the torsion modules we are interested in can be recovered by the knowledge, for finitely many primes $\PP$ of $N$, of the $I_\PP$-module structure of a certain torsion $I_\PP$-module. More precisely, if $\Ram(N/E)$ denotes the set of primes of $\OO_E$ which ramify in $N/E$, then $\T_{N/E}$, $\SSS_{N/E}$ and $\RR_{N/E}$ can be written as a direct sum over $\Ram(N/E)$ of torsion $G$-modules induced from $I_{\PP}$-modules, where, for each $P\in\Ram(N/E)$, $\PP$ is any fixed prime above $P$. For $\RR_{N/E}$ and $\T_{N/E}$ such a decomposition follows directly from Chase's results. For instance, in the case of $\T_{N/E}$, the corresponding $I_{\PP}$-module is a suitable power of the quotient $T(p,\Z[I_{\PP}])=\Z[I_{\PP}]/\Sigma_{I_{\PP}}(p)$, where $\Sigma_{I_{\PP}}(p)=p\Z[I_{\PP}]+\Tr_{I_{\PP}}\Z[I_{\PP}]$ is the Swan module generated by the trace $\Tr_{I_{\PP}}=\sum_{g\in I_{\PP}}g\in\Z[I_{\PP}]$ and the residual characteristic $p$ of $\PP$. 

However for torsion modules arising from general $G$-stable ideals of $\OO_N$ (and in particular for $\SSS_{N/E}$), we need to introduce torsion $\oo_{e_P}[I_{\PP}]$-modules and also assume that $N/E$ is locally abelian. Here $e_P$ is the order of $I_{\PP}$ (which indeed depends only on the prime $P$ of $\OO_E$ lying below $\PP$), $\mu_{e_P}$ is the group of $e_P$th roots of unity in $\overline\Q$ and $\oo_{e_P}$ is the ring of integers of $\Q(\mu_{e_P})$. These $\oo_{e_P}[I_{\PP}]$-modules can be considered as analogues of $T(p,\Z[I_{\PP}])$ and they also give a decomposition for $\RR_{N/E}$, which is slightly different from that of Chase and will be needed in the proof of Theorem \ref{main}. To stress their similarity with $T(p,\Z[I_{\PP}])$, we shall denote by $R_{\chi_{\PP}}(\pp,\oo_{e_P}[I_{\PP}])$ and $S_{\chi_{\PP}}(\pp,\oo_{e_P}[I_{\PP}])$ those which correspond to $\RR_{N/E}$ and $\SSS_{N/E}$, respectively (see (\ref{rchi}) and (\ref{schi}) for a precise definition). Here $\chi_{\PP}$ is an injective character of $I_\PP$ and $\pp$ is a prime above $p$ in $\oo_{e_P}$. 

We now state the main result of this section whose proof will be given in \S\ref{reductionproof}.
 
\begin{theorem}\label{reduction}
For every $P\in \Ram(N/E)$, choose a prime $\PP$ of $N$ above $P$. Then, with the notation introduced above, there is an isomorphism of $\Z[G]$-modules
\begin{eqnarray*}
\T_{N/E}&\cong& \bigoplus_{P\in \Ram(N/E)}\Big(\Z[G]\otimes_{\Z[I_{\PP}]}T(p,\Z[I_{\PP}])\Big)^{\oplus [\OO_E/P\,:\,\F_p]}.
\end{eqnarray*}
Furthermore, for every choice of injective characters $\chi_\PP:I_\PP\to \overline{\Q}^\times$ for every prime $\PP$ as above, one can find primes $\pp\subset \oo_{e_P}$ and injections $\oo_{e_P}/\pp\to \OO_N/\PP$ such that there is an isomorphism of $\Z[G]$-modules 
\begin{eqnarray*}
\RR_{N/E}&\cong& \bigoplus_{P\in \Ram(N/E)}\Big(\Z[G]\otimes_{\Z[I_{\PP}]}R_{\chi_{_\PP}}(\pp,\oo_{e_{P}}[I_{\PP}])\Big)^{\oplus [G:D_{\PP}][\OO_N/\PP\,:\,\oo_{e_P}/\pp]}
\end{eqnarray*}
Assume moreover that $N/E$ is locally abelian. Then the injections $\oo_{e_P}/\pp\to \OO_N/\PP$ factor through $\OO_E/P\to \OO_N/\PP$ and there is an isomorphism of $\Z[G]$-modules
\begin{eqnarray*}
\SSS_{N/E}&\cong& \bigoplus_{P\in \Ram(N/E)}\Big(\Z[G]\otimes_{\Z[I_{\PP}]}S_{\chi_{_\PP}}(\pp,\oo_{e_{P}}[I_{\PP}])\Big)^{\oplus [\OO_{E}/P\,:\,\oo_{e_{P}}/\pp]}.
\end{eqnarray*}
\end{theorem} 

In Section \ref{ctsec} we will see how the above theorem, together with Theorem \ref{main}, can be used to prove Theorem \ref{maincor}. More precisely, we show that the torsion $G$-modules (resp. $I_\PP$-modules) appearing in Theorem \ref{reduction} are $G$-cohomologically trivial (resp. $I_\PP$-cohomologically trivial) and therefore define classes in $\Cl(\Z[G])$ (resp. $\Cl(\Z[I_\PP])$). Thus the isomorphisms of Theorem \ref{reduction} can be translated into equalities of classes in $\Cl(\Z[G])$, which in turn will give directly Theorem \ref{maincor}, assuming Theorem \ref{main}.
   
\subsection{The torsion module $\RR_{N/E}$}

The results of this subsection are due to Chase \cite{ChaseTors}. 
\subsubsection{}\label{defr} 
We shall first recall the definition of $\RR_{N/E}$ in a wider context. Let $\Gamma$ be a finite group and let $K/k$ be a $\Gamma$-Galois extension of either global or local fields. If $X$ and $Y$ are sets, we denote by $\Map(X,Y)$ the set of mappings from $X$ to $Y$. Consider the bijection 
\begin{equation}\label{isom-NG}
\psi_{K/k}:K\otimes_k K\to \Map(\Gamma, K)
\end{equation}
defined by $\psi_{K/k}(x\otimes y)(\gamma)=x\gamma(y)$ for $x,y\in K$ and $\gamma\in \Gamma$. Now $K\otimes_k K$ is a $K[\Gamma]$-module with $K$ acting on the left factor, $\Gamma$ on the right and $\Map(\Gamma,K)$ is a $K[\Gamma]$-module with $K$ acting pointwise and $\Gamma$ acting by $(\gamma u)(\gamma')=u(\gamma'\gamma)$ for all $\gamma,\gamma'\in \Gamma$, $u\in\Map(\Gamma,K)$.
These structures make $\psi_{K/k}$ an isomorphism of $K[\Gamma]$-modules. Restricting $\psi_{K/k}$ to the subring $\OO_K\otimes_{\OO_k}\OO_K\subset K\otimes_{k}K$ yields an $\OO_K[\Gamma]$-modules injection
$$\psi_{K/k}:\OO_K\otimes_{\OO_k}\OO_K\to \Map(\Gamma,\OO_K)$$
whose cokernel
$$\RR_{K/k}=\Map(\Gamma,\OO_K)/\psi_{K/k}(\OO_K\otimes_{\OO_k}\OO_K)$$
is a torsion $\Gamma$-module. We refer the reader to \cite[Sections 2, 3, 4]{ChaseTors} and \cite[Note 6 to Chapter III]{Frohlich-Alg_numb} for more details on $\RR_{K/k}$. 

We now come back to the notation of the beginning of this section, in particular $N/E$ is a tame $G$-Galois extension of number fields. Recall also that for every $P\in \Ram(N/E)$, we fix a prime $\PP$ of $N$ above $P$.


\begin{proposition}\label{redtolocr}
There is an isomorphism of $\OO_E[G]$-modules
\begin{align*}
\RR_{N/E} &\cong \bigoplus_{P\in \Ram(N/E)}\left(\Z[G]\otimes_{\Z[D_{\PP}]} \RR_{N_{\PP}/E_{P}}\right)^{\oplus [G:D_\PP]}.
\end{align*}
\end{proposition}
\begin{proof}
See \cite[Corollary 3.11]{ChaseTors}.
\end{proof}

\subsubsection{}\label{localr}
Proposition \ref{redtolocr} shows that we can focus on the local setting. Therefore in this subsection we shall put ourselves in the following situation (which will appear again at later stages of this paper). We fix a rational prime $p$ and a tamely ramified Galois extension $K/k$ of $p$-adic fields inside a fixed algebraic closure $\overline{\Q}_p$ of $\Q_p$. We denote by $\Gamma$ the Galois group of $K/k$. Let $\Delta\subseteq \Gamma$ be the inertia subgroup of $K/k$, which is cyclic of order denoted by $e$, and set $F=K^{\Delta}$. As usual, $\OO_K$, $\OO_F$ and $\OO_k$ denote the rings of integers of $K$, $F$ and $k$, respectively, and we shall denote by $\PP_K$, $\PP_F$ and $\PP_k$ the corresponding maximal ideals. 

The following result shows that we can in fact focus on totally and tamely ramified local extensions.
 \begin{proposition}\label{redtotramr}
There is an isomorphism of $\OO_K[\Gamma]$-modules
$$\RR_{K/k}\cong\Z[\Gamma]\otimes_{\Z[\Delta]}\RR_{K/F}.$$
\end{proposition}
\begin{proof}
See \cite[Corollary 3.8]{ChaseTors}.
\end{proof}

By standard theory $F$ contains the group of $e$th roots of unity $\mu_{e,p}\subseteq \overline\Q_p$, and we can choose a uniformizer $\pi_{\!_K}$ of $K$ such that $\pi_{\raisebox{2pt}{$\scriptscriptstyle{\!_{K}}$}}^e\in F$ (thus $\pi_{\raisebox{2pt}{$\scriptscriptstyle{\!_{K}}$}}^e$ is a uniformizer of $F$). Consider the map $\chi_{_{K/F}}:\Delta \to \mu_{e,p}$ defined by 
$$\chi_{_{K/F}}(\delta)=\frac{\delta(\pi_{\!_K})}{\pi_{\!_K}}\enspace.$$
Since $K/F$ is totally ramified, any unit $u\in\OO_K^\times$ such that $u^e\in F$ lies in $F$, hence $\chi_{_{K/F}}$ does not depend on the choice of a uniformizer $\pi_{\!_K}$ as above. It easily follows that $\chi_{_{K/F}}$ is a group homomorphism, hence an isomorphism comparing cardinals: $\#\Delta=e=\#\mu_{e,p}$ (see also \cite[Chapitre IV, Propositions 6(a) and 7]{Serre}). 

\begin{remark}\label{gammachi}
Note that $\Delta$ (resp. $\mu_{e,p}$) is a $\Gamma/\Delta$-module with the conjugation (resp. Galois) action. Then $\chi_{_{K/F}}$ is in fact an isomorphism of $\Gamma/\Delta$-modules. To prove this, it is enough to verify that, for every $\gamma\in\Gamma$ and $\delta\in\Delta$, we have $\chi_{_{K/F}}(\gamma\delta\gamma^{-1})=\gamma(\chi_{_{K/F}}(\delta))$. But indeed, if $\pi_{\!_K}$ is as above, we have
$$\chi_{_{K/F}}(\gamma\delta\gamma^{-1})=\frac{\gamma\delta\gamma^{-1}(\pi_{\!_K})}{\pi_{\!_K}}=\gamma\left(\frac{\delta\gamma^{-1}(\pi_{\!_K})}{\gamma^{-1}(\pi_{\!_K})}\right)=\gamma(\chi_{_{K/F}}(\delta))$$
since $\gamma^{-1}(\pi_{\!_K})$ is a uniformizer of $K$ whose $e$th power belongs to $F$. Hence $\chi_{_{K/F}}$ is a $\Gamma$-isomorphism and we deduce in particular that, if $\Gamma$ is abelian, then $\mu_{e,p}\subset k$. The reverse implication is also true: if $\mu_{e,p}\subset k$, then $\Gamma$ acts trivially on $\Delta$. This implies that $\Gamma$ is abelian, since $\Gamma=\langle\gamma,\Delta\rangle$ for any $\gamma\in \Gamma$ whose image in $\Gamma/\Delta$ generates $\Gamma/\Delta$ (which is a cyclic group).
\end{remark}

If $M$ is an $\OO_K$-module, we let $M(\chi_{_{K/F}}^i)$ denote the $\OO_K[\Delta]$-module which is $M$ as an $\OO_K$-module and has $\Delta$-action defined by $\delta\cdot m=\chi_{_{K/F}}^i(\delta)m$.

We now show that the $\Delta$-module $\RR_{K/F}$ can be decomposed in smaller pieces. 

\begin{proposition}\label{chlemr}
The action of $\OO_F$ on $\RR_{K/F}$ factors through $\OO_F/\PP_F$ and there is an isomorphism of $\OO_F/\PP_F[\Delta]$-modules
$$\RR_{K/F}\cong \bigoplus_{i=1}^{e-1}(\PP_K^i/\PP_K^{i+1})^{\oplus i}.$$
\end{proposition}
\begin{proof}
We start from Chase's decomposition \cite[Theorem 2.8]{ChaseTors} which, in our notation, is an isomorphism of $\OO_K[\Delta]$-modules:
$$\RR_{K/F}\cong \bigoplus_{i=1}^{e-1}(\OO_K/\PP_K^{i})(\chi_{_{K/F}}^i)\enspace.$$
Note that, for $0\leq i\leq e$, $\OO_K/\PP_K^{i}$ is an $\OO_F/\PP_F$-module, since  $\PP_F\OO_K=\PP_K^e\subseteq \PP_K^i$ and hence the action of $\OO_F$ on $\OO_K/\PP_K^{i}$ factors through $\PP_F$. From its definition $\RR_{K/F}$ is an $\OO_F$-module, hence an $\OO_F/\PP_F$-module by the above isomorphism, which is thus an isomorphism of $\OO_F/\PP_F[\Delta]$-modules. In what follows we shall be mainly concerned with $\OO_F/\PP_F[\Delta]$-module structures (although some of the assertions hold true in the category of $\OO_K[\Delta]$-modules). 

Observe that the $\OO_F/\PP_F[\Delta]$-module $(\OO_K/\PP_K^{i})(\chi_{_{K/F}}^i)$ has the filtration $\{(\PP_K^{j}/\PP_K^{i})(\chi_{_{K/F}}^i)\}_{j=0}^{i}$ whose corresponding subquotients are $(\PP_K^{j}/\PP_K^{j+1})(\chi_{_{K/F}}^i)$ for $j=0,\,\ldots,\,i-1$. It is clear that multiplication by the $j$th power of any uniformizer of $K$ induces an $\OO_F/\PP_F$-isomorphism $\OO_K/\PP_K\cong\PP_K^{j}/\PP_K^{j+1}$ and hence an $\OO_F/\PP_F[\Delta]$-isomorphism $(\OO_K/\PP_K)(\chi_{_{K/F}}^i)\cong(\PP_K^{j}/\PP_K^{j+1})(\chi_{_{K/F}}^i)$. Therefore using the se\-mi\-simplicity of $\OO_F/\PP_F[\Delta]$, we get 
$$(\OO_K/\PP_K^{i})(\chi_{_{K/F}}^i)\cong\bigoplus_{j=0}^{i-1}(\OO_K/\PP_K)(\chi_{_{K/F}}^i)=(\OO_K/\PP_K)(\chi_{_{K/F}}^i)^{\oplus i}.$$

Note that the Galois action of $\Delta$ on $\PP_K^i/\PP_K^{i+1}$ coincides with the action given by multiplication by $\chi_{_{K/F}}^i$ since
\begin{equation}\label{actionPimod}
\delta[\pi_{\!_K}^ix]=[\delta(\pi_{\!_K}^ix)]=[\chi_{_{K/F}}(\delta)^i\pi_{_K}^i\delta(x)]=[\chi_{_{K/F}}(\delta)^i\pi_{_K}^ix]
\end{equation}
where, for $y\in \PP_K^i$, we let $[y]$ denote the class of $y$ in $\PP_K^i/\PP_K^{i+1}$ (the last equality of (\ref{actionPimod}) follows from the fact that $\Delta$ acts trivially on $\OO_K/\PP_K=\OO_F/\PP_F$). Thus both $\PP_K^i/\PP_K^{i+1}$ and $(\OO_K/\PP_K)(\chi_{_{K/F}}^i)$ are $\OO_F/\PP_F$-vector spaces of dimension $1$ on which $\Delta$ acts by multiplication by $\chi_{_{K/F}}^i$. Therefore
$$(\OO_K/\PP_K^{i})(\chi_{_{K/F}}^i)\cong (\PP_K^i/\PP_K^{i+1})^{\oplus i}$$ 
as $\OO_F/\PP_F[\Delta]$-modules.
\end{proof}

\subsection{Torsion modules arising from ideals}

\subsubsection{}
We keep the notation of the beginning of this section. Let $\I\subset \OO_N$ be a $G$-invariant ideal. We will show that the $G$-module structure of $\OO_{N}/\I$ and $\I^{-1}/\OO_N$ is of local nature. We denote by $\Div(\I)$ the (finite) set of primes of $E$ dividing $\I$ and, for every $P\in \Div(\I)$ we fix a prime $\PP$ of $N$ above $P$. 

\begin{proposition}\label{redtoloc}
Let $\I\subset \OO_N$ be a $G$-invariant ideal. For every prime $P\in\Div(\I)$, let $n_P$ be the valuation of $\I$ at any prime of $\OO_N$ above $P$.
Then  there are isomorphisms of $\OO_E[G]$-modules
\begin{align*}
\OO_N/\I &\cong \bigoplus_{P\in \Div(\I)}\Z[G]\otimes_{\Z[D_{\PP}]} \left(\OO_{N_{\PP}}/\PP^{n_P}\OO_{N_{\PP}}\right)\\
\I^{-1}/\OO_N &\cong \bigoplus_{P\in \Div(\I)}\Z[G]\otimes_{\Z[D_{\PP}]} \left(\PP^{-n_P}\OO_{N_{\PP}}/\OO_{N_{\PP}}\right)\enspace.
\end{align*}
\end{proposition}

\begin{proof}
We begin by proving the first isomorphism. Since $\I$ is $G$-invariant, we can write
$$\I=\prod_{P\in \Div(\I)} \prod_{\PP\mid P}\PP^{n_P}.$$ 
The Chinese remainder theorem gives an $\OO_N$ isomorphism
\begin{equation}\label{chinese}
\OO_N/\I\cong \bigoplus_{P\in \Div(\I)} \OO_N\big/(\prod_{\PP\mid P}\PP^{n_{P}}),
\end{equation}
which is indeed also one of $\OO_E[G]$-modules.
In a similar way we also get an isomorphism of $\OO_E$-modules
$$\OO_N\big/(\prod_{\PP\mid P} \PP^{n_{P}})\cong \prod_{\PP\mid P}(\OO_N/\PP^{n_{P}})$$
for every $P\in\Div(\I)$. The above isomorphism is easily seen to be $G$-invariant, once the right-hand side is given a $G$-module structure by 
$$(g\cdot(x_\PP))_{\PP_0}=g(x_{g^{-1}(\PP_0)})$$
for every $g\in G$, $(x_\PP)\in \prod_{\PP\mid P}(\OO_N/\PP^{-n_{P}})$ and $\PP_0\mid P$.
A standard argument shows that, for any prime $\PP_0$ above $P$, we have 
\begin{equation}\label{chind}
 \prod_{\PP\mid P}(\OO_N/\PP^{n_{P}})\cong\mathrm{Map}_{D_{\PP_0}}(G,\OO_N/\PP_0^{n_{P}})\cong\Z[G]\otimes_{\Z[D_{\PP_0}]}\OO_N/\PP_0^{n_{P}}
\end{equation}
as $\OO_N$-modules and $\OO_E[G]$-modules. Note also that, for every $n\in \mathbb{N}$ and every $\PP\mid P$, the inclusion $N\to N_P$ induces an isomorphism
$$\OO_N/\PP^{n_{P}}\cong \OO_{N_\PP}/\PP^{n_{P}}\OO_{N_\PP}$$
of $\OO_E[D_\PP]$-modules. This shows the first isomorphism of the lemma.
 
The proof of the second isomorphism follows the same pattern, once one has the following analogue of the Chinese remainder theorem.
\begin{lemma}
Let $\J_1,\,\J_2\subset \OO_N$ be ideals with $\J_1+\J_2=\OO_N$. Then there is an isomorphism of $\OO_N$-modules
$$(\J_1\J_2)^{-1}/\OO_N\stackrel{\sim}{\longrightarrow} (\J_1)^{-1}/\OO_N\times \J_2^{-1}/\OO_N.$$
\end{lemma}
\begin{proof}
We claim that the inclusions $\J_1^{-1}\to (\J_1\J_2)^{-1}$ and $\J_2^{-1}\to (\J_1\J_2)^{-1}$ induce $\OO_N$-isomorphisms
\begin{equation}\label{invchin1}
\tau_1:\J_1^{-1}/\OO_N\to (\J_1\J_2)^{-1}/\J_2^{-1}\quad\textrm{and}\quad \tau_2:\J_2^{-1}/\OO_N\to (\J_1\J_2)^{-1}/\J_1^{-1}
\end{equation} 
and, analogously, the natural projections $(\J_1\J_2)^{-1}/\OO_N\to (\J_1\J_2)^{-1}/\J_1^{-1}$ and $(\J_1\J_2)^{-1}/\OO_N\to (\J_1\J_2)^{-1}/\J_2^{-1}$ induce an $\OO_N$-isomorphism
\begin{equation}\label{invchin2}
\tau:(\J_1\J_2)^{-1}/\OO_N\to (\J_1\J_2)^{-1}/\J_1^{-1}\times (\J_1\J_2)^{-1}/\J_2^{-1}.
\end{equation}  

Write $1=j_1+j_2$ with $j_1\in \J_1$ and $j_2\in \J_2$. To show that $\tau_1,\,\tau_2$ and $\tau$ are injective, we only need to show that $\J_1^{-1}\cap \J_2^{-1}\subseteq \OO_N$ (the reverse inclusion being obvious). If $j\in \J_1^{-1}\cap \J_2^{-1}$, then $j=1\cdot j=j_1j+j_2j$ and both $j_1j$ and $j_2j$ belong to $\OO_N$. This shows that $\tau_1,\tau_2$ and $\tau$ are injective.

To prove the surjectivity of $\tau_1$, let $j\in(\J_1\J_2)^{-1}$. Then $j_1j\in \J_2^{-1}$ and hence $j-j_1j$ belongs to the class of $j$ in $(\J_1\J_2)^{-1}/\J_2^{-1}$. On the other hand $j-j_1j=j_2j\in \J_1^{-1}$, which shows that $\tau_1$ is surjective and the surjectivity of $\tau_2$ follows by a similar argument.

As for the surjectivity of $\tau$, take $y,z\in (\J_1\J_2)^{-1}$. One easily sees that $x=yj_1+zj_2$ belongs to $(\J_1\J_2)^{-1}$ and 
$$x\equiv yj_1\equiv y-yj_2\equiv y \mod{\J_1^{-1}}\quad \textrm{and}\quad x\equiv zj_2\equiv z-zj_1\equiv z \mod{\J_2^{-1}}.$$
This shows that $\tau$ is surjective and complete the proof of our claim.

The lemma then follows since 
\begin{equation}\label{isoon}
(\tau_1\times \tau_2)^{-1}\circ\tau:(\J_1\J_2)^{-1}/\OO_N\to (\J_1)^{-1}/\OO_N\times \J_2^{-1}/\OO_N
\end{equation} 
is an $\OO_N$-isomorphism.
\end{proof}

Observe also that, if $\J_1$ and $\J_2$ moreover are $G$-stable ideals, then the isomorphism (\ref{isoon}) is also an isomorphism of $\OO_E[G]$-modules.
\end{proof}

\subsubsection{} 
Proposition \ref{redtoloc} allows us to focus on local extensions. We put ourselves in the local setting of \S \ref{localr}. In particular, $K/k$ is a $\Gamma$-extension of $p$-adic fields and $F$ is the subfield of $K$ which is fixed by the inertia group $\Delta$.

We begin by considering the local analogue $\psi_{K/k}$ of the isomorphism introduced in (\ref{isom-NG}), namely $\psi_{K/k}:K\otimes_kK\rightarrow\Map(\Gamma,K)$ sends $x\otimes y$ to $\gamma\mapsto x\gamma(y)$. We give $K\otimes_kK$ its natural $(\Gamma\times\Gamma)$-module structure: $(\gamma,\gamma')(x\otimes y)=\gamma(x)\otimes\gamma'(y)$. Then $\psi_{K/k}$ is an isomorphism of $(\Gamma\times\Gamma)$-modules if we let $(\Gamma\times\Gamma)$ act on $\Map(\Gamma,K)$ by
$$\big((\gamma,\gamma')u\big)(\eta)=\gamma\big(u(\gamma^{-1}\eta\gamma')\big)\enspace,$$
for all $u\in\Map(\Gamma,K)$, $\gamma,\gamma',\eta\in\Gamma$. Note that the action of the subgroup $1\times\Gamma$ of $\Gamma\times\Gamma$ is the same as that introduced below (\ref{isom-NG}), once $1\times\Gamma$ is identified with $\Gamma$.

We  define $\Map(\Gamma,K)^{\Delta}=\Map(\Gamma,K)^{\Delta\times1}$ to be the set of invariant maps under the action of the subgroup $\Delta\times1$ of $\Gamma\times\Gamma$. More explicitely, $\Map(\Gamma,K)^{\Delta}$ is the set of maps $u:\Gamma\to K$ such that
$$\delta(u(\eta))=u(\delta\eta)$$
for all $\delta\in\Delta$, $\eta\in\Gamma$. We may view $\Map(\Gamma,K)^{\Delta}$ as an $F$-algebra with the pointwise operations and as a $\Gamma$-module 
where $\Gamma$ acts as $1\times\Gamma$. Then there is an isomorphism of both $F$-algebras and $F[\Gamma]$-modules:
\begin{equation}\label{indiso}
\Map(\Gamma, K)^{\Delta}\stackrel{\sim}{\longrightarrow}
\Q[\Gamma]\otimes_{\Q[\Delta]} K\ ,\quad u\longmapsto \sum_{\gamma\in\Gamma}\gamma^{-1}\otimes u(\gamma)
\end{equation}
where $\Q[\Gamma]\otimes_{\Q[\Delta]} K$ is the tensor product over $\Q[\Delta]$ of the right $\Q[\Delta]$-module $\Q[\Gamma]$ with the left $\Q[\Delta]$-module $K$. This tensor product is given the structure of a $\Gamma$-module via its left-hand factor and the structure of an $F$-algebra via its right-hand factor. 


The isomorphism $\psi_{K/k}$ introduced above yields
an isomorphism of both $F$-algebras and $F[\Gamma]$-modules:
$$\psi_{K/k}:F\otimes_k K\to \Map(\Gamma, K)^{\Delta}$$
(here $F\otimes_k K$ is considered an $F$-algebra via its left factor and as a $\Gamma$-module via its right factor). 
Composing with the isomorphism in (\ref{indiso}), we get an isomorphism 
$$\tilde \psi_{K/k}:F\otimes_k K\stackrel{\sim}{\longrightarrow} \Q[\Gamma]\otimes_{\Q[\Delta]} K\enspace.$$
Note that $\Z[\Gamma]\otimes_{\Z[\Delta]} \OO_K$ is the maximal order of $\Q[\Gamma]\otimes_{\Q[\Delta]} K$ and, using that $F/k$ is unramified, it is not difficult to show that $\OO_F\otimes_{\OO_k} \OO_K$ is the maximal $\OO_F$-order of $F\otimes_k K$ (see \cite[p. 214]{ChaseTors}). 
Therefore $\tilde \psi_{K/k}$ induces the following isomorphism of rings and $\OO_F[\Gamma]$-modules:
$$\OO_F\otimes_{\OO_k} \OO_K\cong\Z[\Gamma]\otimes_{\Z[\Delta]} \OO_K\enspace.$$

\begin{lemma}\label{redtotramts}
For every $n\in\mathbb{N}$, the homomorphism $\tilde \psi_{K/k}$ induces isomorphisms of $\OO_F[\Gamma]$-modules 
\begin{align*}
\OO_F\otimes_{\OO_k}\OO_K/\PP_K^n&\cong \Z[\Gamma]\otimes_{\Z[\Delta]}\OO_K/\PP_K^n;\\
\OO_F\otimes_{\OO_k}\PP_K^{-n}/\OO_K&\cong \Z[\Gamma]\otimes_{\Z[\Delta]}\PP_K^{-n}/\OO_K.
\end{align*} 
\end{lemma}
\begin{proof}
Consider the following commutative diagram of $\OO_F[\Gamma]$-modules
$$
\xymatrix{
0\ar[r]&\OO_F\otimes_{\OO_k}\PP^n_K \ar[d]^{\tilde\psi_{K/k}} \ar[r] & \OO_F\otimes_{\OO_k}\OO_K \ar[d]^{\tilde\psi_{K/k}} \ar[r] &\OO_F\otimes_{\OO_k}\OO_K/\PP_K^n \ar[d]\ar[r]&0\\
0\ar[r]&\Z[\Gamma]\otimes_{\Z[\Delta]}\PP^n_K\ar[r] &  \Z[\Gamma]\otimes_{\Z[\Delta]}\OO_K \ar[r] & \Z[\Gamma]\otimes_{\Z[\Delta]}\OO_K/\PP_K^n\ar[r]&0\\}
$$
It has exact rows since $\OO_F$ (resp. $\Z[\Gamma]$) is a flat $\OO_k$-module (resp. $\Z[\Delta]$-module), being free. The central vertical arrow is an isomorphism, as remarked above. In particular, the right-hand vertical arrow is surjective. But one easily verifies that  
\begin{eqnarray*}
\#\left( \OO_F\otimes_{\OO_k}\OO_K/\PP^n_K\right)&=&(\#\OO_K/\PP^n_K)^{[F:k]}\\
&=&(\#\OO_K/\PP^n_K)^{[\Gamma:\Delta]}\\
&=&\#\left(\Z[\Gamma]\otimes_{\Z[\Delta]}\OO_K/\PP^n_K\right).
\end{eqnarray*}
Therefore the right-hand vertical arrow is an isomorphism and so is the left-hand one by the snake lemma. This proves the first isomorphism of the lemma.

The proof of the second isomorphism is similar: consider the following commutative diagram of $\OO_F[\Gamma]$-modules with exact rows
$$
\xymatrix{
0\ar[r]&\OO_F\otimes_{\OO_k}\OO_K \ar[d]^{\tilde\psi_{K/k}} \ar[r] & \OO_F\otimes_{\OO_k}\PP^{-n}_K \ar[d]^{\tilde\psi_{K/k}} \ar[r] &\OO_F\otimes_{\OO_k}\PP_K^{-n}/\OO_K \ar[d]\ar[r]&0\\
0\ar[r]&\Z[\Gamma]\otimes_{\Z[\Delta]}\OO_K\ar[r] &  \Z[\Gamma]\otimes_{\Z[\Delta]}\PP^{-n}_K \ar[r] & \Z[\Gamma]\otimes_{\Z[\Delta]}\PP_K^{-n}/\OO_K\ar[r]&0\\}
$$
The left-hand vertical arrow is an isomorphism. Therefore it suffices to prove that the central arrow is injective (one then conclude using a cardinality argument as above). For that purpose, it is enough to show that the map $\OO_F\otimes_{\OO_k}\PP^{-n}_K\to F\otimes_{k}K$ is injective, thanks to the following commutative diagram of $\OO_F[\Gamma]$-modules 
$$
\begin{CD}
\OO_F\otimes_{\OO_k}\PP^{-n}_K@>>> F\otimes_{k}K\\
@V\tilde\psi_{K/k} VV@V\tilde\psi_{K/k} VV\\
\Z[\Gamma]\otimes_{\Z[\Delta]}\PP^{-n}_K@>>> \Z[\Gamma]\otimes_{\Z[\Delta]}K
\end{CD}
$$
whose right-hand arrow is an isomorphism. Note that $F\otimes_{k}K$ is the localization of the $\OO_k$-module  $\OO_F\otimes_{\OO_k}\PP^{-n}_K$ at the multiplicative set $k^\times$. The map $\OO_F\otimes_{\OO_k}\PP^{-n}_K\to F\otimes_{k}K$ is then injective because $\OO_F\otimes_{\OO_k}\PP^{-n}_K$ is a torsion free $\mathcal{O}_k$-module.
\end{proof}

The above lemma is somehow unsatisfactory, if compared for example with Proposition \ref{redtotramr}, because it says that the $\Gamma$-modules $\OO_K/\PP_K^n$ and $\PP_K^{-n}/\OO_K$ are induced from some $\Delta$-module only after tensoring with $\OO_F$. In the next subsection, introducing torsion $\Delta$-modules coming from global cyclotomic fields, we will get rid of this scalar extension, at least when $K/k$ is abelian.


The following proposition (which may be considered as a generalization of \cite[Lemma 1.4]{ChaseTors}), shows that the $\Delta$-modules $\OO_K/\PP_K^n$ and $\PP_K^{-n}/\OO_K$ break up in smaller pieces as in Proposition \ref{chlemr}. Recall that $e$ is the order of $\Delta$.  

\begin{proposition}\label{chlem} 
For every $n\in \mathbb{Z}$ with $0\leq n\leq e$, the action of $\OO_F$ on $\OO_K/\PP_K^{n}$ and $\PP_K^{-n}/\OO_K$ factors through an action of $\OO_F/\PP_F$ and we have
$$\OO_K/\PP_K^{n}\cong\bigoplus_{i=0}^{n-1}\PP_K^i/\PP_K^{i+1}$$ 

$$\PP_K^{-n}/\OO_K\cong\bigoplus_{i=1}^{n}\PP_K^{e-i}/\PP_K^{e-i+1}$$
as $\OO_F/\PP_F[\Delta]$-modules.
\end{proposition}
\begin{proof}
Both $\OO_K/\PP_K^n$ and $\PP_K^{-n}/\OO_K$ are $\OO_F/\PP_F[\Delta]$-modules (see the proof of Proposition \ref{chlemr}) and have the filtrations $\{\PP_K^i/\PP^n_K\}_{i=0}^{n}$ and $\{\PP_K^{-i}/\OO_K\}_{i=0}^{n}$, respectively. Since $\OO_F/\PP_F[\Delta]$ is semisimple, we have
$$\OO_K/\PP_K^n\cong\bigoplus_{i=0}^{n-1}\PP_K^i/\PP_K^{i+1} \qquad \textrm{and}\qquad \PP_K^{-n}/\OO_K\cong\bigoplus_{i=1}^{n}\PP_K^{-i}/\PP_K^{-i+1} $$
as $\OO_F/\PP_F[\Delta]$-modules. This yields the first isomorphism. Furthermore multiplication by the $e$th power of any uniformizer of $K$ induces a $\OO_F/\PP_F[\Delta]$-isomorphism 
$$\PP_K^{-i}/\PP_K^{-i+1}\cong \PP_K^{e-i}/\PP_K^{e-i+1}.$$ 
We thus get the second isomorphism.    
\end{proof}

\begin{remark}
We add a comment on the hypothesis on $n$ in the above proposition. If $n>e$, then $\OO_K/\PP_K^n$ is not a semisimple $\OO_F[\Delta]$-module. For, suppose the contrary: then we would have an isomorphism of $\OO_F[\Delta]$-modules
$$\OO_K/\PP_K^n\cong\bigoplus_{i=0}^{n-1}\PP_K^i/\PP_K^{i+1}$$
by the same arguments as in the proof of Proposition \ref{chlem}. But this implies in particular that  $\OO_K/\PP_K^n$ is an $\OO_F/\PP_F$-module, which is clearly not the case since $n>e$. However Proposition \ref{chlem} will be useful in computing the class of $\OO_K/\PP_K^n$ in $\Cl(\Z[\Gamma])$ (in the sense of Section \ref{ctsec}) for arbitrary $n\in\mathbb{N}$, thanks to Proposition \ref{mode}.
\end{remark}

\subsection{Switch to a global cyclotomic field}\label{section:back-to-global}
In this subsection we will perform a further reduction, relating the modules $\OO_K/\PP^n_K$, $\PP^{-n}_K/\OO_K$ and $\RR_{K/F}$ to new torsion Galois modules, associated to the ring of integers of a certain cyclotomic field. 


Recall that $\Delta$ is cyclic of order $e$. As in the Introduction $\mu_e$ denotes the group of $e$th roots of unity in $\overline\Q$ and $\oo$ is the ring of integers of $\Q(\mu_e)$. Let $\chi:\Delta\to \mu_e$ be a character of $\Delta$. For any $\oo$-module $M$, we shall consider the $\oo[\Delta]$-module $M(\chi)$ whose underlying $\oo$-module is $M$ and $\Delta$ acts as $\delta\cdot m=\chi(\delta)m$. We shall be mainly concerned with the case where $M$ is the residue field $\kappa_\pp=\oo/\pp$ of a prime $\pp\subset \oo$ not dividing $e$. 


\subsubsection{} We now explain the relation between the modules introduced above with those of the previous subsections. We come back to the setting of \S\ref{localr}: in particular $K/F$ is a $\Delta$-Galois extension of $p$-adic fields which is totally and tamely ramified of degree $e$.

\begin{lemma}\label{embedding}
If $\chi:\Delta\to \mu_e$ is injective, then there exists an embedding $\iota:\overline\Q\to\overline\Q_p$ such that $\iota\circ\chi=\chi_{_{K/F}}$.
\end{lemma}
\begin{proof}
We first define $\iota:\Q(\mu_e)\to\overline\Q_p$ by setting $\iota(\chi(\delta))=\chi_{_{K/F}}(\delta)$ for every $\delta\in\Delta$. Note that this indeed defines an injective field homomorphism, since $\chi$ and $\chi_{_{K/F}}$ are actually isomorphisms. Then we can extend $\iota$ to an embedding $\overline{\Q}\to\overline{\Q}_p$ in infinitely many ways and any of these extensions satisfies the requirements of the lemma.
\end{proof}

We now fix an injective character $\chi:\Delta\to \mu_e$ and an embedding $\iota:\overline\Q\to\overline\Q_p$ such that $\iota\circ\chi=\chi_{_{K/F}}$. Note that $\iota(\oo)\subseteq \OO_F$ (since $\mu_{e,p}\subset F$) and therefore we can view any $\OO_F$-module as a $\oo$-module via $\iota$. 

\begin{proposition}\label{descent}
Let $\pp\subset\oo$ be the prime ideal above $p$ such that $\iota(\pp)\subset \PP_F$. For every natural integer $i$ and every uniformizer $\pi_{\!_K}$ of $K$, we have an isomorphism of $\OO_F[\Delta]$-modules:
$$\PP_K^i/\PP_K^{i+1}\ \cong \kappa_\pp(\chi^i)\otimes_{\kappa_\pp}
\OO_F/\PP_F\enspace,$$
where the right-hand side of the above isomorphism is an $\OO_F/\PP_F$-module via its right factor and a $\Delta$-module via its left factor.  
\end{proposition}
\begin{proof}
We identify $\OO_F/\PP_F$ with $\OO_K/\PP_K$ via the inclusion $\OO_F\subset \OO_K$. Then sending $[\pi_{\raisebox{2pt}{$\scriptscriptstyle{\!_{K}}$}}^ix]\in \PP_K^i/\PP_K^{i+1}$ to $[1]\otimes[x]\in \kappa_\pp(\chi^i)\otimes_{\kappa_\pp}
\OO_F/\PP_F$ clearly gives a $\OO_F$-isomorphism between $\PP_K^i/\PP_K^{i+1}$ and $\kappa_\pp(\chi^i)\otimes_{\kappa_\pp}
\OO_F/\PP_F$. By (\ref{actionPimod}), $\Delta$ acts as multiplication by $\chi_{\!_{K/F}}^i$ on $\PP_K^i/\PP_K^{i+1}$ and therefore $\delta\cdot[\pi_{\raisebox{2pt}{$\scriptscriptstyle{\!_{K}}$}}^ix]$ maps to
$$[1]\otimes[(\chi^i_{\!_{K/F}})(\delta)\,x]=[1]\otimes\iota\chi^i(\delta)[x]=\chi^i(\delta)[1]\otimes[x]=(\delta\cdot[1])\otimes[x]=\delta\cdot([1]\otimes[x])\enspace.$$
\end{proof}

%
 
We are ready for the main application of the torsion $\oo[\Delta]$-modules we have introduced. They allow us to write the $\Gamma$-modules $\OO_K/\PP_K^{n}$ and $\PP_K^{-n}/\OO_K$ as induced from some $\Delta$-modules, at least if $K/k$ is abelian.  

\begin{proposition}\label{locab}
Let $\pp\subset\oo$ be the prime ideal above $p$ such that $\iota(\pp)\subset \PP_F$. Assume that $K/k$ is abelian and let $0\leq n\leq e$ be an integer. Then $\iota$ induces an inclusion $\kappa_\pp\to\OO_k/\PP_k$ (hence $\OO_k/\PP_k$ is a $\kappa_\pp$-module via $\iota$) and there are isomorphisms of $\OO_k/\PP_k[\Gamma]$-modules
\begin{eqnarray*}
\OO_K/\PP_K^n&\cong& \OO_k/\PP_k\otimes_{\kappa_\pp}\Big(\Z[\Gamma]\otimes_{\Z[\Delta]}\Big(\bigoplus_{i=0}^{n-1}\kappa_\pp(\chi^i)\Big)\Big)\enspace,\\
\PP_K^{-n}/\OO_K&\cong& \OO_k/\PP_k\otimes_{\kappa_\pp}\Big(\Z[\Gamma]\otimes_{\Z[\Delta]}\Big(\bigoplus_{i=1}^{n}\kappa_\pp(\chi^{e-i})\Big)\Big)\enspace,
\end{eqnarray*}
where the right-hand sides of the above isomorphisms are $\OO_k/\PP_k$-modules via their left factors and $\Gamma$-modules via their right factors.
\end{proposition}
\begin{proof}
We prove the first isomorphism, the proof of the second is similar. Our proof is inspired by that of \cite[Theorem 1.7]{ChaseTors}. Using Lemma \ref{redtotramts}, Propositions \ref{chlem} and \ref{descent}, we get
\[\begin{split}
\OO_F/\PP_F\otimes_{\OO_k/\PP_k}\OO_K/\PP_K^n&\cong
\Z[\Gamma]\otimes_{\Z[\Delta]}\OO_K/\PP_K^n\\
&\cong
\Z[\Gamma]\otimes_{\Z[\Delta]}\Big(\bigoplus_{i=0}^{n-1}\PP_K^i/\PP_K^{i+1}\Big)\\
&\cong
\Z[\Gamma]\otimes_{\Z[\Delta]}\Big(\Big(\bigoplus_{i=0}^{n-1}\kappa_\pp(\chi^i)\Big)
\otimes_{\kappa_\pp}\OO_F/\PP_F\Big)\enspace.
\end{split}\]
Observe now that, since $K/k$ is abelian, we have $\mu_{e,p}\subset k$ by Remark \ref{gammachi} and hence $\iota(\oo)\subset \OO_k$. In particular we can write the above isomorphism as
$$\OO_F/\PP_F\otimes_{\OO_k/\PP_k}\OO_K/\PP_K^n\cong
\Big(\Big(\Z[\Gamma]\otimes_{\Z[\Delta]}\Big(\bigoplus_{i=0}^{n-1}\kappa_\pp(\chi^i)\Big)\Big)
\otimes_{\kappa_\pp}\OO_k/\PP_k\Big)
\otimes_{\OO_k/\PP_k}\OO_F/\PP_F.$$
Since the above are isomorphisms of $\OO_F/\PP_F[\Gamma]$-modules of finite length, we can apply the Krull-Schmidt theorem to conclude (see \cite[\S 6, Exercise 2]{CR1}).
\end{proof}

\begin{remark}\label{unrfree}
If $K/k$ is unramified, then Proposition \ref{locab} simply asserts that $\OO_K/\PP_K$ and $\PP_K^{-1}/\OO_K$ are free $\OO_k/\PP_k[\Gamma]$-modules (which is well-known and can be proved directly). In fact, if $K/k$ is unramified, then of course it is abelian (even cyclic) and $\Delta$ is trivial. In particular $e=1$, $\chi$ is trivial and $\kappa_\pp=\F_p$. Then, by Proposition \ref{locab}, we get 
$$\begin{array}{rcccl}
\OO_K/\PP_K&\cong& \OO_k/\PP_k\otimes_{\F_p}\Big(\Z[\Gamma]\otimes_{\Z}\F_p\Big)&\cong& \OO_k/\PP_k[\Gamma];\\
\PP_K^{-1}/\OO_K&\cong& \OO_k/\PP_k\otimes_{\kappa_\pp}\Big(\Z[\Gamma]\otimes_{\Z}\F_p\Big)&\cong& \OO_k/\PP_k[\Gamma].
\end{array}$$
as $\OO_k/\PP_k[\Gamma]$-modules.
\end{remark}

In view of Proposition \ref{locab} and since we are mainly interested in $\T_{K/k}$, $\SSS_{K/k}$ and $\RR_{K/k}$ we introduce the following notation for any prime $\pp\subset \oo$ not dividing $e$:
\begin{align}
T_{\chi}(\pp,\oo[\Delta])&=\bigoplus_{i=1}^{e-1}\kappa_\pp(\chi^i)\enspace,\label{tchi}\\
S_{\chi}(\pp,\oo[\Delta])&=\bigoplus_{i=\frac{e+1}{2}}^{e-1}\kappa_\pp(\chi^i)\enspace,\label{schi}\\
R_{\chi}(\pp,\oo[\Delta])&=\bigoplus_{i=1}^{e-1}\kappa_\pp(\chi^i)^{\oplus i}\enspace\label{rchi}.
\end{align}
For $S_\chi(\pp,\oo[\Delta])$ to be defined, $e$ is required to be odd, an assumption that will always be implicit when needed. 

\begin{remark}\label{tswan}
There is a link between $T_{\chi}(\pp,\oo[\Delta])$ and the Swan module $\Sigma_\Delta(p)=p\Z[\Delta]+\Tr_\Delta\Z[\Delta]$, where $\Delta=\langle\delta\rangle$, $\Tr_\Delta=\sum_{i=0}^{e-1}\delta^i\in\Z[\Delta]$ and $p$ is the re\-si\-dual characteristic of $\pp$.
Using the decomposition of $\kappa_\pp[\Delta]$ given by primitive idempotents, one easily gets that
\begin{equation}\label{gralg}
\kappa_\pp[\Delta]\cong\bigoplus_{i=0}^{e-1}\kappa_\pp(\chi^i)
\end{equation} 
as $\oo[\Delta]$-modules. It follows that
$$\kappa_\pp[\Delta]/\langle\Tr_\Delta\rangle=T_\chi(\pp,\oo[\Delta])\enspace.$$
This already shows that $T_\chi(\pp,\oo[\Delta])$ is independent of the injective character $\chi$. Further, since
$$T(p,\Z[\Delta])=\Z[\Delta]/\Sigma_\Delta(p)=\F_p[\Delta]/\langle\Tr_\Delta\rangle\enspace,$$ 
we get
$$T_\chi(\pp,\oo[\Delta])=T(p,\Z[\Delta])\otimes_{\F_p}\kappa_\pp\enspace.$$
The above isomorphism can be used to eliminate the hypothesis $K/k$ abelian, at least for $\T_{K/k}=\PP_K^{1-e}/\OO_K$. Choosing $\pp$ as in Proposition \ref{locab} and arguing as in the proof of that proposition, we get
\begin{align*}
\OO_F/\PP_F\otimes_{\OO_k/\PP_k}\T_{K/k}
&\cong \Z[\Gamma]\otimes_{\Z[\Delta]}\left(\OO_F/\PP_F\otimes_{\kappa_\pp}T_\chi(\pp,\oo[\Delta])\right)\\
&\cong \Z[\Gamma]\otimes_{\Z[\Delta]}\left(\OO_F/\PP_F\otimes_{\F_p}T(p,\Z[\Delta])\right)\\
&\cong\OO_F/\PP_F\otimes_{\F_p}\left(\Z[\Gamma]\otimes_{\Z[\Delta]}T(p,\Z[\Delta])\right)\\
&\cong\OO_F/\PP_F\otimes_{\OO_k/\PP_k}\OO_k/\PP_k\otimes_{\F_p}\left(\Z[\Gamma]\otimes_{\Z[\Delta]}T(p,\Z[\Delta])\right)\enspace.
\end{align*}
Then, as in the proof of Proposition \ref{locab}, by the Krull-Schmidt theorem we get
$$\T_{K/k}\cong \OO_k/\PP_k\otimes_{\F_p}\left(\Z[\Gamma]\otimes_{\Z[\Delta]}T(p,\Z[\Delta])\right)$$
as $\OO_k/\PP_k[\Gamma]$-modules.
\end{remark}

\subsubsection{}\label{reductionproof}
We now collect the results obtained so far to complete our reduction step. We recall the setting
described at the beginning of this section. Let $N/E$ be a tame $G$-Galois extension of number fields. For any prime $P$ of $\OO_E$ we fix a prime $\PP\subseteq \OO_N$ dividing $P$. Let $D_{\PP}$ (resp. $I_{\PP}$) denote the decomposition
group (resp. the inertia subgroup) of $\PP$ in $G$. Then the cardinality of $I_\PP$ only depends on $P$ and we denote it by $e_P$. Using Lemma \ref{embedding}, we fix an injective character $\chi_{_{\PP}}:I_{\PP}\to \overline{\Q}^{\times}$ and an embedding $\iota_{_{\PP}}:\overline{\Q}\to\overline\Q_p$ (where $p$ is the rational prime below $\PP$ and $\overline\Q_p$ is an algebraic closure of $\Q_p$ containing the completion $N_{\PP}$ of $N$ at $\PP$), such that $\iota_{_{\PP}}\circ\chi_{_{\PP}}=\chi_{_{N_{\PP}/F_{\PP}}}$ where $F_{\PP}=N_{\PP}^{I_{\PP}}$. These choices determine a prime ideal $\pp$ in the ring of integers $\oo_{e_{P}}$ of $\Q(\mu_{e_{P}})\subset\overline{\Q}$ (where $e_{P}=\#I_{\PP}$) satisfying $\iota_{_{\PP}}(\pp)\subseteq \PP\OO_{N_{\PP}}$. Moreover $\OO_{F_\PP}$ is an $\oo_{e_{P}}$-module via $\iota_{_{\PP}}$. Recall that $\Ram(N/E)$ is the set of primes of $E$ that ramify in $N/E$.

\begin{proof}[Proof of Theorem \ref{reduction}]
Using Proposition \ref{redtoloc} and Remark \ref{tswan}: 
\begin{eqnarray*}
\T_{N/E}&\cong &\bigoplus_{P\in \Ram(N/E)} \Z[G]\otimes_{\Z[D_{\PP}]}\T_{N_{\PP}/E_{P}}\\
&\cong&  \bigoplus_{P\in \Ram(N/E)}\Z[G]\otimes_{\Z[D_{\PP}]}\Big( \Z[D_{\PP}]\otimes_{\Z[I_{\PP}]}T(p,\Z[I_{\PP}])\Big)^{\oplus [\OO_E/P:\F_p]}
\end{eqnarray*}
as $\Z[G]$-modules.

By Propositions \ref{redtolocr}, \ref{redtotramr}, \ref{chlemr} and \ref{descent} and with the choices of $\chi_P$ and $\iota_\PP$ described above, we have
isomorphisms of $\Z[G]$-modules:
\begin{eqnarray*}
\RR_{N/E}&\cong &\bigoplus_{P\in \Ram(N/E)} \left(\Z[G]\otimes_{\Z[D_{\PP}]}\RR_{N_{\PP}/E_{P}}\right)^{\oplus[G:D_{\PP}]}\\
&\cong&  \bigoplus_{P\in \Ram(N/E)}\Big(\Z[G]\otimes_{\Z[D_{\PP}]}\Big( \Z[D_{\PP}]\otimes_{\Z[I_{\PP}]}R_{\chi_{_\PP}}(\pp,\oo_{e_{P}}[I_{\PP}])\Big)\Big)^{\oplus [G:D_{\PP}][\OO_N/\PP:\oo_{e_P}/\pp]}\enspace.
\end{eqnarray*}
Suppose now that $N/E$ is locally abelian. Then $E_P$ contains the $e_P$th roots of unity in $\overline{\Q}_p$ (as explained in Remark \ref{gammachi}) and therefore $\iota_\PP$ induces an inclusion $\oo_{e_P}/\pp\to\OO_{E_P}/P\OO_{E_P}\cong \OO_E/P$.  Moreover using Propositions \ref{redtoloc} and \ref{locab} we have isomorphisms of $\Z[G]$-modules:
\begin{eqnarray*}
\SSS_{N/E}&\cong &\bigoplus_{P\in \Ram(N/E)} \Z[G]\otimes_{\Z[D_{\PP}]}\SSS_{N_{\PP}/E_{P}}\\
&\cong&  \bigoplus_{P\in \Ram(N/E)}\Z[G]\otimes_{\Z[D_{\PP}]}\Big( \Z[D_{\PP}]\otimes_{\Z[I_{\PP}]}S_{\chi_{_\PP}}(\pp,\oo_{e_{P}}[I_{\PP}])\Big)^{\oplus [\OO_{E}/P:\oo_{e_{P}}/\pp]}\enspace.
\end{eqnarray*}
\end{proof}

For torsion modules arising from general $G$-stable ideals we no more have isomorphisms as in Theorem \ref{reduction} but still equalities of classes in $\Cl(\Z[G])$, in the sense we shall now explain. 

\subsection{Classes of cohomologically trivial modules}\label{ctsec}
\subsubsection{} In this subsection $G$ is an arbitrary finite group (we do not need it to be the Galois group of a particular extension of number fields). We will interpret the results we have obtained so far in terms of classes in the locally free class group. Recall that a $G$-module $M$ is $G$-cohomologically trivial if, for every $i\in\mathbb{Z}$ and every subgroup $G'<G$, the Tate cohomology group $\hat H^{i}(G',M)$ is trivial. If $A$ is the ring of integers of a number field, let $\Cl(A[G])$ be the locally free class group of $A[G]$ (see \cite[I, \S2]{Frohlich-Alg_numb} for locally free modules and the locally free class group).

\begin{lemma}\label{cohomtriv}
Let $A$ be the ring of integers of a number field. Let $M$ be a finitely generated $A[G]$-module.
\begin{enumerate} 
\item[(i)] $M$ is $A[G]$-projective if and only if it is $A[G]$-locally free. 
\item[(ii)] $M$ is $G$-cohomologically trivial if and only if there exists an $A[G]$-resolution $0\to P_1\to P_0\to M\to 0$ of $M$ with $P_0$ and $P_1$ locally free. In this case the class $(P_0)^{-1}(P_1)$ in $\Cl(A[G])$ is independent of the chosen locally free resolution of $M$ and will be denoted by $(M)_{A[G]}$.
\item[(iii)] If $G$ is a subgroup of a finite group $\tilde{G}$ and $M$ is $G$-cohomologically trivial, then $M\otimes_{A[G]}A[\tilde{G}]$ is $\tilde{G}$-cohomologically trivial and $$(M\otimes_{A[G]}A[\tilde{G}])_{A[\tilde{G}]}=\mathrm{Ind}_{G}^{\tilde{G}}\big((M)_{A[G]}\big)$$ 
where $\mathrm{Ind}_{G}^{\tilde{G}}:\Cl(A[G])\to \Cl(A[\tilde{G}])$ is the map which sends the class $(P)_{A[G]}\in \Cl(A[G])$ of a locally free $A[G]$-module $P$ to the class $(P\otimes_{A[G]}A[\tilde{G}])\in \Cl(A[\tilde{G}])$.  
\end{enumerate}
\end{lemma}
\begin{proof}
For (i) and the first assertion of (ii) see for example \cite[Proposition 4.1]{Chinburg-deRham} ((i) is a classical result of Swan). The last assertion of (ii) follows immediately from Schanuel's lemma. 

To prove (iii), suppose that $M$ is $G$-cohomologically trivial. Then, by (i) and (ii), there exists exact sequence of $A[G]$-modules 
$$0\to P_1\to P_0\to M\to 0$$
with $P_0,P_1$ projective. Observe that $A[\tilde{G}]$ is a free $A[G]$-module. In particular the functor $\--\otimes_{A[G]}A[\tilde{G}]$ from the category of $A[G]$-modules to that of $A[\tilde{G}]$-modules is exact and we get an exact sequence of $A[\tilde{G}]$-modules
\begin{equation}\label{estens}
0\to P_1\otimes_{A[G]}A[\tilde{G}]\to P_0\otimes_{A[G]}A[\tilde{G}]\to M\otimes_{A[G]}A[\tilde{G}]\to 0.
\end{equation}
Note that, for $i=0,1$, $P_i\otimes_{A[G]}A[\tilde{G}]$ is $A[\tilde{G}]$-projective since $P_i$ is $A[G]$-projective (this can be easily seen using the characterization of projective modules as direct summand of free modules). In particular, the exact sequence (\ref{estens}) implies that $M\otimes_{A[G]}A[\tilde{G}]$ is $\tilde{G}$-cohomologically trivial by (i) and (ii). Moreover we have  
\begin{eqnarray*}
(M\otimes_{A[G]}A[\tilde{G}])_{A[\tilde{G}]}&=& (P_0\otimes_{A[G]}A[\tilde{G}])^{-1}(P_1\otimes_{A[G]}A[\tilde{G}])\\
&=&\mathrm{Ind}_G^{\tilde{G}}\big((P_0)_{A[G]}\big)^{-1}\mathrm{Ind}_G^{\tilde{G}}\big((P_1)_{A[G]}\big)\\
&=&\mathrm{Ind}_G^{\tilde{G}}\big((P_0)_{A[G]}^{-1}(P_1)_{A[G]}\big)\\
&=&\mathrm{Ind}_G^{\tilde{G}} \big((M)_{A[G]}\big)
\end{eqnarray*}
in $\Cl(A[\tilde{G}])$.
\end{proof}

In this section we will use the above lemma when $A=\Z$ but later we will also need the case where $A$ is the ring of integers of a cyclotomic field. If $M$ is a finitely generated $\Z[G]$-module which is cohomologically trivial, we will denote $(M)_{\Z[G]}\in\Cl(\Z[G])$ simply by $(M)$.

\subsubsection{}
We first put ourselves in the local setting of \S \ref{localr}. In particular $K/k$ is a tame $\Gamma$-Galois extension of $p$-adic fields with inertia subgroup $\Delta$ and $F=K^{\Delta}$. Note that for every $a,b\in \mathbb{Z}$ with $b\geq a$, the $\Z[\Gamma]$-module $\PP_K^a/\PP_K^b$ is $\Gamma$-co\-ho\-mo\-logically trivial. This follows immediately from the fact that $\PP_K^a$ and $\PP_K^b$ are $\Gamma$-co\-ho\-mo\-logically trivial (see \cite[Theorem 2]{Ullom-cohomology}).

\begin{proposition}\label{mode}
For every $m, n\in\mathbb{N}$ such that $n \equiv m \!\mod{e}$, we have 
\begin{equation}\label{morceaux}
(\OO_K/\PP_K^n)=(\OO_K/\PP_K^m)\in \Cl(\Z[\Gamma])\enspace.
\end{equation}
\begin{proof}
Without loss of generality we may assume that $n\geq m$ and write $n=m+ae$ for some $a\in\mathbb{N}$. It is clear that
$$(\OO_K/\PP_K^n)=(\OO_K/\PP_K^m)(\PP_K^m/\PP_K^n)=(\OO_K/\PP_K^m)\prod_{j=1}^{a}(\PP_K^{m+(j-1)e}/\PP_K^{m+je})$$
in $\Cl(\Z[\Gamma])$. Thus we only have to prove that, for every $b\in\mathbb{N}$, $(\PP_K^b/\PP_K^{b+e})=0$ in $\Cl(\Z[\Gamma])$. Arguing as in the proof of Proposition \ref{chlemr}, we observe that $\PP_K^b/\PP_K^{b+e}$ is an $\OO_F/\PP_F[\Delta]$-module and, since $\OO_F/\PP_F[\Delta]$ is semisimple, we have an isomorphism of $\OO_F/\PP_F[\Delta]$-modules
$$\PP_K^b/\PP_K^{b+e}\cong \bigoplus_{i=0}^{e-1}\PP_K^{b+i}/\PP_K^{b+i+1}\enspace.$$
As remarked in the proof of Proposition \ref{chlemr}, $\PP_K^{b+i}/\PP_K^{b+i+1}$ is an $\OO_F/\PP_F$-vector space of dimension $1$ on which $\Delta$ acts by multiplication by $\chi_{_{\!K/F}}^{b+i}$. Thus using the decomposition of $\OO_F/\PP_F[\Delta]$ given by primitive idempotents we get
\begin{equation}\label{idempo}
\bigoplus_{i=0}^{e-1}\PP_K^{b+i}/\PP_K^{b+i+1}\cong\OO_F/\PP_F[\Delta]\cong \OO_F/\PP_F\otimes_{\F_p}\F_p[\Delta]
\end{equation}
as $\OO_F/\PP_F[\Delta]$-modules.
Now it easily follows from Lemma \ref{redtotramts} that 
$$\OO_F/\PP_F\otimes_{\OO_k/\PP_k}\PP_K^b/\PP_K^{b+e}\cong \Z[\Gamma]\otimes_{\Z[\Delta]}\PP_K^b/\PP_K^{b+e}$$
as $\OO_F/\PP_F[\Gamma]$-modules.
Therefore using (\ref{idempo}) we get isomorphisms
\begin{align*}
\OO_F/\PP_F\otimes_{\OO_k/\PP_k}\PP_K^b/\PP_K^{b+e}&\cong \Z[\Gamma]\otimes_{\Z[\Delta]}\left(\OO_F/\PP_F\otimes_{\F_p}\F_p[\Delta]\right)\\
&\cong\OO_F/\PP_F\otimes_{\F_p}\F_p[\Gamma]\\
&\cong\OO_F/\PP_F\otimes_{\OO_k/\PP_k}\left(\OO_k/\PP_k\otimes_{\F_p}\F_p[\Gamma]\right)\enspace
\end{align*}
of $\OO_F/\PP_F[\Gamma]$-modules. As in the proof of Proposition \ref{locab}, we can apply the Krull-Schmidt theorem and deduce that $\PP_K^b/\PP_K^{b+e}$ and $\OO_k/\PP_k\otimes_{\F_p}\F_p[\Gamma]$ are isomorphic $\OO_k[\Gamma]$-modules (and hence in particular as $\Z[\Gamma]$-modules). Now $\F_p[\Gamma]$ is a cohomologically trivial $\Gamma$-module whose class in $\Cl(\Z[\Gamma])$ is trivial, thanks to the $\Z[G]$-free re\-so\-lution  
$$0\to p\Z[G]\to \Z[G]\to \F_p[G]\to 0.$$ 
We have thus proved what we wanted. 
\end{proof}
\end{proposition}

\subsubsection{}\label{ctglob}
We now come back to the global setting of \S \ref{reductionproof}. Thus $N/E$ is a tame $G$-Galois extension of number fields. Note that every $G$-stable fractional ideal of $N$ is $\Z[G]$-projective (see \cite[Proposition 1.3]{Ullom-normal_bases}) hence locally free by Lemma \ref{cohomtriv} (i). In particular, if $\I$ is a $G$-stable ideal of $\OO_N$, then $\OO_N/\I$ and $\I^{-1}/\OO_N$ are $G$-cohomologically trivial (by Lemma \ref{cohomtriv} (ii)). Therefore we can consider the classes $(\OO_N/\I)$ and $(\I^{-1}/\OO_N)$ in $\Cl(\Z[G])$ and in fact 
$$(\OO_N/\I)=(\I)(\OO_N)^{-1}\quad\textrm{and}\quad(\I^{-1}/\OO_N)=(\OO_N)(\I^{-1})^{-1}.$$  
Similarly, $\RR_{N/E}$ defines a class in $\Cl(\Z[G])$. In fact $\Map(G,\OO_N)$ is $\OO_N[G]$-free of rank $1$ (hence $\Z[G]$-free of rank $[N:\Q]$) and $\OO_N\otimes_{\OO_E}\OO_N$ is $\OO_N[G]$-locally free (since $\OO_N$ is $\OO_E[G]$-locally free by Noether's theorem). Thus, in this case,
$$(\RR_{N/E})=(\OO_N\otimes_{\OO_E}\OO_N)\in \Cl(\Z[G])\enspace.$$

Note also that, for every prime $P$ of $\OO_E$ and any integer $i$, the $I_{\PP}$-module $(\oo_{e_{\PP}}/\pp)(\chi_{\PP}^i)$ is cohomologically trivial. In fact, for every $i\in\mathbb{Z}$ and every subgroup $I<I_\PP$, $\hat H^{i}(I,(\oo_{e_{\PP}}/\pp)(\chi_{\PP}^i))$ is annihilated by $e_P$ (see \cite[Chapitre VIII, Corollaire 1 to Proposition 3]{Serre}) and $p$ (since $p$ annihilates $(\oo_{e_{\PP}}/\pp)(\chi_{\PP}^i)$). Since $N/E$ is tame, we have $(p,e_P)=1$ and hence $\hat H^{i}(I,(\oo_{e_{\PP}}/\pp)(\chi_{\PP}^i))=0$. Thanks to Lemma \ref{cohomtriv} (ii) this allows us to consider the class $\big((\oo_{e_{\PP}}/\pp)(\chi_{\PP}^i)\big)\in\Cl(\Z[I_{\PP}])$.

\begin{proposition}\label{preburns}
Let $\I$ be a $G$-stable ideal of $\OO_N$ and assume that $N/E$ is locally abelian at $P\in \Div(\I)$. For every prime $P\in\Div(\I)$, fix a prime $\PP$ of $\OO_N$ dividing $P$ and let $n_P$ be the valuation of $\I$ at $\PP$ (this indeed depends only on $P$). For every choice of injective characters $\chi_\PP:I_\PP\to \overline{\Q}^\times$ for every prime $\PP$ as above, one can find primes $\pp\subset \oo_{e_P}$ and injections $\oo_{e_P}/\pp\to \OO_N/\PP$ such that we have equalities
$$(\OO_N/\I)=\prod_{P\in \Div(\I)}\prod_{i=0}^{m_P-1}\mathrm{Ind}_{I_\PP}^{G}\left((\oo_{e_P}/\pp)(\chi_{_\PP}^{i})\right)^{[\OO_{E}/P:\oo_{e_{P}}/\pp]}$$
$$(\I^{-1}/\OO_N)=\prod_{P\in \Div(\I)}\prod_{i=1}^{m_P}\mathrm{Ind}_{I_\PP}^{G}\left((\oo_{e_P}/\pp)(\chi_{_\PP}^{e-i})\right)^{[\OO_{E}/P:\oo_{e_{P}}/\pp]}$$
in $\Cl(\Z[G])$, where $m_P$ is the smallest nonnegative integer congruent to $n_P$ modulo $e_P$. In particular, if $\I$ is coprime with the different of $N/E$, then $(\OO_N/\I)=(\I^{-1}/\OO_N)=1$.
\end{proposition}
\begin{proof}
We prove only one of the two displayed equalities, namely the first one, since the other follows by similar arguments.
First of all note that we have an isomorphism of $\OO_E[G]$-modules 
\begin{equation}\label{isodec}
\OO_N/\I \cong \bigoplus_{P\in \Div(\I)}\Z[G]\otimes_{\Z[D_{\PP}]} \OO_{N_{\PP}}/\PP^{n_P}\OO_{N_{\PP}}
\end{equation}
by Proposition \ref{redtoloc}. By Propositions \ref{mode} and \ref{locab} we have
\begin{eqnarray*}(\OO_{N_{\PP}}/\PP^{n_P}\OO_{N_{\PP}})&=&(\OO_{N_{\PP}}/\PP^{m_P}\OO_{N_{\PP}})\\
&=& \Big(\Z[D_\PP]\otimes_{\Z[I_{\PP}]}\Big(\bigoplus_{i=0}^{m_P-1}(\oo_{e_P}/\pp)(\chi_{_\PP}^i)\Big)\Big)^{[\OO_E/P:\oo_{e_P}/\pp]}
\end{eqnarray*}
in $\Cl(\Z[D_{\PP}])$, which together with (\ref{isodec}) and Lemma \ref{cohomtriv} (iii) gives the first equality of the proposition.

To prove the last assertion, note that if $\I$ is coprime with the different of $N/E$, then for every $\PP\mid \I$ we have $I_{\PP}=1$. In particular, $\chi_\PP$ is trivial, $e_P=1$ and $\oo_{e_P}/\pp=\F_p$ (thus we are in situation similar to that of Remark \ref{unrfree}). Therefore, for every $i\in\Z$, using Lemma \ref{cohomtriv} (iii)
\begin{eqnarray*}
\mathrm{Ind}_{I_\PP}^{G}\left((\oo_{e_P}/\pp)(\chi_{_\PP}^{i})\right)&=&\left((\oo_{e_P}/\pp)(\chi_{_\PP}^{i})\otimes_{\Z[I_\PP]}\Z[G]\right)\\
&=&\left(\F_p\otimes_{\Z}\Z[G]\right)\\
&=&(\F_p[G]).
\end{eqnarray*}
As in the proof of Proposition \ref{mode}, we observe that $(\F_p[G])=1$, which concludes the proof of the proposition.
\end{proof}

The above proposition can be used to prove the following interesting result. Recall that we regard $\OO_N\otimes_{\OO_E}\OO_N$ as a $\Z[G]$-module with the action defined in \S\ref{defr}, \textit{i.e.} $G$ only acts on the right-hand factor. 

\begin{proposition}\label{ONn}
We have 
$$(\OO_N\otimes_{\OO_E}\OO_N)=(\OO_N)^{[N:E]}\quad\textrm{in $\Cl(\Z[G])$}.$$
\end{proposition}
\begin{proof}
Set $n=[N:E]$. By the structure theorem for $\OO_E$-modules, we know that $\OO_N$ is $\OO_E$-isomorphic to $\OO_E^{\oplus (n-1)}\oplus J$, where $J$ is an ideal of $\OO_E$. By Chebotarev's density theorem, we can find an ideal $I$ of $\OO_E$ belonging to the ideal class of $J$ and such that $I$ is coprime with the discriminant of $N/E$. In particular $\OO_N$ is also $\OO_E$-isomorphic to $\OO_E^{\oplus (n-1)}\oplus I$ and
$$
\OO_N\otimes_{\OO_E}\OO_N\cong \big(\OO_E\otimes_{\OO_E}\OO_N\big)^{\oplus (n-1)}\oplus \big(I\otimes_{\OO_E}\OO_N\big)
\cong \OO_N^{\oplus (n-1)}\oplus I\OO_N
$$
as $\OO_E[G]$-modules (since $G$ only acts on the right-hand term of $\OO_N\otimes_{\OO_E}\OO_N$). In particular we get
\begin{equation}\label{decoclass}
(\OO_N\otimes_{\OO_E}\OO_N)=(\OO_N)^{n-1}(I\OO_N) \quad\textrm{in $\Cl(\Z[G])$}.
\end{equation}
Now $I\OO_N$ is of course a $G$-stable ideal of $\OO_N$ since $I$ is an ideal of $\OO_E$. In particular $I\OO_N$ is locally free because $N/E$ is tame (see \cite[Proposition 1.3]{Ullom-normal_bases}) and $\OO_N/I\OO_N$ is $G$-cohomologically trivial by Lemma \ref{cohomtriv} (ii). Moreover we have
$$(I\OO_N)=(\OO_N)(\OO_N/I\OO_N)\quad \textrm{in $\Cl(\Z[G])$.}$$
Note that $I\OO_N$ is coprime with the different of $N/E$. In particular $(\OO_N/I\OO_N)=1$ by Proposition \ref{preburns} and therefore $(I\OO_N)=(\OO_N)$. Plugging this equality in (\ref{decoclass}) we get the statement of the proposition.
\end{proof}

\subsubsection{} We end this section by showing how Theorem \ref{reduction} can be used to reduce the proof of Theorem \ref{maincor} to that of Theorem \ref{main} 
\begin{proof}[Proof of Theorem \ref{maincor} assuming Theorem \ref{main}]
By Theorem \ref{reduction} and using Lemma \ref{cohomtriv} (iii), we have the following equalities in $\Cl(\Z[G])$:
\begin{eqnarray*}
(\RR_{N/E})&=& \prod_{P\in \Ram(N/E)}\mathrm{Ind}_{I_\PP}^G\big(R_{\chi_{_\PP}}(\pp,\oo_{e_{P}}[I_{\PP}])\big)^{[G:D_{\PP}][\OO_N/\PP:\oo_{e_P}/\pp]}\enspace,\\
(\T_{N/E})&=& \prod_{P\in \Ram(N/E)}\mathrm{Ind}_{I_\PP}^G\big(T(p,\Z[I_{\PP}])\big)^{[\OO_E/P:\F_p]}
\end{eqnarray*}
and, if $N/E$ is locally abelian, 
\begin{eqnarray*}
(\SSS_{N/E})&=& \prod_{P\in \Ram(N/E)}\mathrm{Ind}_{I_\PP}^G\big(S_{\chi_{_\PP}}(\pp,\oo_{e_{P}}[I_{\PP}])\big)^{[\OO_{E}/P:\oo_{e_{P}}/\pp]}\enspace.
\end{eqnarray*}
By Theorem \ref{main}, for every prime $P\in\Ram(N/E)$ and every prime $\PP\mid P$ in $\OO_N$, we have 
$$(T(p,\Z[I_\PP]))=(S_{\chi_\PP}(\pp,\oo_{e_P}[I_\PP]))=(R_{\chi_\PP}(\pp,\oo_{e_P}[I_\PP]))=1$$
in $\Cl(\Z[I_\PP])$. Thus $(\T_{N/E})=1$ which implies $(\OO_N)=(\C_{N/E})$. Moreover $(\RR_{N/E})=1$ which gives $(\OO_{N}\otimes_{\OO_E}\OO_N)=(\OO_N[G])$ and, since $\OO_N[G]$ is $\Z[G]$-free of rank $[N:\Q]$, we deduce $(\OO_{N}\otimes_{\OO_E}\OO_N)=1$. In particular we also have $(\OO_N)^{[N:E]}=1$ by Proposition \ref{ONn}. Finally, if $N/E$ is locally abelian, we have $(\SSS_{N/E})=1$ which implies $(\OO_N)=(\A_{N/E})$. The proof of Theorem \ref{maincor} is then achieved. 
\end{proof}

\section{Hom-representatives}\label{section:hom-des}

We now come to the proof of Theorem \ref{main}, which will be achieved in two steps. In this section, we apply Fröhlich's machinery to get a first description of Hom-representatives of the classes involved in its statement. Then in the next section we use Stickelberger's theorem to refine this description and complete the proof.

We are thus in the cyclotomic setting introduced in the previous section, namely we fix an integer $e$, a cyclic group $\Delta$ of order $e$ and an injective character $\chi:\Delta\to \mu_e$, where $\mu_e$ is the group of $e$th roots of unity in $\overline{\Q}$. We let $\oo$ denote the ring of integers of the cyclotomic field $\mathbb{Q}(\mu_e)$. Let $p$ denote a rational prime such that $p\nmid e$ and let $\pp\subset \oo$ denote a prime ideal above $p$. We set $\kappa=\oo/\pp$, 
$$T_\Z=T(p,\Z[\Delta])\ ,\quad R=R_\chi(\pp,\oo[\Delta])\ ,\quad S=S_\chi(\pp,\oo[\Delta])\enspace.$$ 
We fix a primitive $e$th root of unity $\zeta\in\mu_e$ and we let $\delta\in\Delta$ be defined by $\chi(\delta)=\zeta$.

\subsection{Hom description of the class group}\label{hdcg}

In this section and the following one, we are interested in determining classes in the class group $\Cl(\Z[\Delta])$. In order to do so, at some point we shall have to consider class groups of a group algebra with a larger coefficient ring. Further in Section \ref{analytic} we shall also need a description of the class group of the group algebra $\Z[G]$ where $G$ is any finite group. So we will recall Fröhlich's Hom-description in its most general form. If $L$ is any number field with $L\subset\overline\Q$, we set $\Omega_L=\mathrm{Gal}(\overline{\Q}/L)$ and we let $\oo_L$ and $J(L)$ denote the ring of integers and the id\`ele group of $L$, respectively. 

\subsubsection{} Fr\"ohlich's Hom-description of $\Cl(\oo_L[G])$, where $L$ is a number field and $G$ is a finite group, is the group isomorphism
\begin{equation}\label{hom-des}
\Cl(\oo_L[G])\cong\frac{\Hom_{\Omega_L}(R_G,J(L'))}{\Hom_{\Omega_L}(R_G,(L')^{\times\!})\Det({\mathcal U}(\oo_L[G]))}
\end{equation}
given by the explicit construction of a representative homomorphism of the class of any locally free rank one module, see \cite[Theorem 1]{Frohlich-Alg_numb}. This construction will be shown and used in the next subsections. We now briefly explain the objects involved in \eqref{hom-des}, referring to \cite{Frohlich-Alg_numb} for a more complete account of Fr\"ohlich's Hom-description.

We begin with the upper part, where $R_G$ is the additive group of virtual characters of $G$ with values in $\overline{\Q}$. The number field $L'$ is ``big enough'', in particular it is Galois over $\Q$, contains $L$ and the values of the characters of $G$. In our cyclotomic setting described above, we shall only be concerned with the cases where $L=\Q$ or $\Q(\mu_e)$ and $G=\Delta$: since $\Delta$ is cyclic of order $e$, in these cases one can take $L'=\Q(\mu_e)\subset \overline{\Q}$. The homomorphisms in $\Hom_{\Omega_L}(R_G,J(L'))$ are those which commute with the natural actions of $\Omega_L$ on $R_G$ and $J(L')$. 

In the lower part, $\Hom_{\Omega_L}(R_G,(L')^\times)$ is the subgroup of $\Hom_{\Omega_L}(R_G,J(L'))$ yielded by the diagonal embedding of $(L')^\times$ in $J(L')$. The second factor needs more explanations. First
$${\mathcal U}(\oo_L[G])=\prod_{\mathfrak{l}}\oo_{L_{\mathfrak{l}}}[G]^{\times}\subseteq\prod_{\mathfrak{l}}{L_\mathfrak{l}}[G]^{\times}\enspace,$$
where $\mathfrak{l}$ runs over all places of $L$ and $\oo_{L_{\mathfrak{l}}}$ denotes the ring of integers of a completion $L_\mathfrak{l}$ of $L$ at $\mathfrak{l}$ (with $\oo_{L_{\mathfrak{l}}}=L_{\mathfrak{l}}$ if $\mathfrak{l}$ is archimedean). Let $x=(x_\mathfrak{l})_\mathfrak{l}\in\prod_{\mathfrak{l}}{L_\mathfrak{l}}[G]^{\times}$, the character function $\Det(x)=(\Det(x_\mathfrak{l}))_\mathfrak{l}$ is defined componentwise. For each place $\mathfrak{l}$ of $L$ the `semi-local' component $\Det(x_\mathfrak{l})$ takes values in $(L'\otimes_L L_\mathfrak{l})^\times$, embedded in $J_\mathfrak{l}(L')=\prod_{\LL\mid\mathfrak{l}}(L'_\LL)^\times$, where $\LL$ runs over the prime ideals of $\oo_{L'}$ above $\mathfrak{l}$, through the isomorphism of $L'$-algebras
\begin{equation}\label{embeddings}
L'\otimes_L L_\mathfrak{l}\cong\prod_{\LL\mid\mathfrak{l}}L'_\LL
\end{equation}
built on the various embeddings of $L'$ in $\bar L_\mathfrak{l}$, a given algebraic closure of $L_\mathfrak{l}$, that fix $\mathfrak{l}$.
By linearity we only need to define the character function $\Det(x_\mathfrak{l})$ on the irreducible characters $\theta$ of $G$. Write $x_\mathfrak{l}=\sum_{g\in G}x_{\mathfrak{l},g}g$, then $\Det_{\theta}(x_\mathfrak{l})$ is the image in $J_\mathfrak{l}(L')$, under isomorphism \eqref{embeddings}, of the determinant of the matrix
$$\sum_{g\in G}x_{\mathfrak{l},g}\Theta(g)\enspace,$$
where $\Theta$ is any matrix representation of character $\theta$ and the $(i,j)$-entry of the above matrix $\sum_{g\in G}\Theta_{i,j}(g)\otimes x_{\mathfrak{l},g}$ indeed belongs to $L'\otimes_L L_\mathfrak{l}$. 

Note that, by $\Omega_L$-equivariance, the values of $\Det(x_\mathfrak{l})=\big(\Det(x_\mathfrak{l})_\LL\big)_{\LL\mid\mathfrak{l}}$ in $J_\mathfrak{l}(L')$ are determined by those of any component $\Det(x_\mathfrak{l})_{\LL}$, see \cite[II, Lemma 2.1]{Frohlich-Alg_numb}. In the following we may thus implicitely assume that a place $\LL$ of $L'$ is fixed above each place $\mathfrak{l}$ of $L$ and focus on the  $\LL$-component $\Det(x_\mathfrak{l})_\LL$, that we shall indeed plainly denote by $\Det(x_\mathfrak{l})$, omitting the unnecessary subscript. Let $l$ denote the rational place below $\mathfrak l$ and $\overline{\Q}_l$ an algebraic closure of $\Q_l$ containing $L'_\LL$. The resulting local function $\Det(x_\mathfrak{l})$ belongs to $\Hom_{\Omega_{L_\mathfrak{l}}}(R_{G,l},(L'_\LL)^\times)$, where $R_{G,l}$ is the group of virtual characters of $G$ with values in $\overline\Q_l$ and $\Omega_{L_\mathfrak{l}}=\Gal(\overline\Q_l,L_\mathfrak{l})$. 
\subsubsection{} We now assume that $G$ is abelian. With the above notation and conventions, we get 
$$\Det_{\theta}(x_\mathfrak{l})=\sum_{g\in G}x_{\mathfrak{l},g}\theta(g)\in L'_{\mathfrak{L}}\enspace,$$
where we have implicitely embedded $L'$ in $L'_\LL$ (in accordance with the above choice of a place $\LL$ above $\mathfrak{l}$). 
\begin{proposition}\label{prop:Det}
With the above notation and assuming that $G$ is abelian, the group homomorphism $\Det:{L_\mathfrak{l}}[G]^{\times}\to\Hom_{\Omega_{L_\mathfrak{l}}}(R_{G,l},(L'_\LL)^\times)$ is injective.
\end{proposition}
\begin{proof}
See \cite[(II.5.2)]{Frohlich-Alg_numb}.
\end{proof}
We shall use the above result in Section \ref{section:stickelberger}.
\subsection{Hom-representative of $(T_{\Z})$}\label{hom-rep-TZ} 
We shall consider, as in Remark \ref{tswan}, the Swan module 
$\Sigma_\Delta(p)=p\Z[\Delta]+\Tr_\Delta\Z[\Delta]$ and its
associated torsion module $T_{\Z}=\Z[\Delta]/\Sigma_\Delta(p)$. In other words we have the following exact sequence of $\Z[\Delta]$-modules:
\begin{equation*}\label{exact}
0\rightarrow\Sigma_\Delta(p)\rightarrow\Z[\Delta]\rightarrow
T_{\Z}\rightarrow 0.
\end{equation*}
As remarked by Swan (see \cite[\S 6]{SwanPRFG}), $\Sigma_\Delta(p)$ is $\Z[\Delta]$-projective hence locally free by Lemma \ref{cohomtriv} (i). In particular, by Lemma \ref{cohomtriv} (ii), $T_{\Z}$ is $\Delta$-cohomologically trivial and we have an equality $\big(T_{\Z}\big)=\big(\Sigma_\Delta(p)\big)$ in $\Cl(\Z[\Delta])$. We now follow Fr\"ohlich's recipe to build a representative morphism $v$ for this class.

If $x$ and $y$ are elements of a same set, we let $\updelta_{x,y}$ denote their Kronecker delta, namely $\updelta_{x,y}=1$ if $x=y$, $0$ otherwise.
\begin{lemma}\label{swan-repr1}
Let $v\in\Hom_{\Omega_\Q}(R_\Delta,J(\Q))$ be defined by
$$v(\chi^h)_q=\left\{
\begin{array}{ll}1& \textrm{if $q\ne p$,}\\
p^{1-\updelta_{h,e}}& \textrm{if $q= p$,}
\end{array}
\right.$$
where $q$ is any rational prime and $h\in\{1,\ldots,e\}$. Then $v$ represents the class of $\Sigma_\Delta(p)$ through Fr\"ohlich's Hom-description of $\Cl(\Z[\Delta])$.
\end{lemma}
\begin{proof}
As remarked above the $\mathbb{Z}[\Delta]$-module $\Sigma_\Delta(p)$ is locally free. For any rational prime $q$, we will now find a generator $\alpha_{q}$ of the free $\mathbb{Z}_q[\Delta]$-module $\mathbb{Z}_q\otimes_{\mathbb{Z}}\Sigma_\Delta(p)$. 
\begin{itemize}
	\item $q\ne p$: since $p$ is invertible in $\mathbb{Z}_q$, we
          have
          $\mathbb{Z}_q\otimes_{\mathbb{Z}}\Sigma_\Delta(p)=\mathbb{Z}_q[\Delta]$,
          so we can take $\alpha_{q}=1$.  
	\item $q=p$: set $\varepsilon_0=\frac{1}{e}\Tr_\Delta$ and
          $\varepsilon_1=1-\varepsilon_0$. Note that
          $\varepsilon_0,\varepsilon_1\in \mathbb{Z}_p[\Delta]$ since
          $e\in\mathbb{Z}_p^\times$. We have
          $\mathbb{Z}_p\otimes_{\mathbb{Z}}
          \Sigma_\Delta(p)=p\mathbb{Z}_p[\Delta]+\varepsilon_0\mathbb{Z}_p[\Delta]$ 
and, since $\varepsilon_i\varepsilon_j=\updelta_{i,j}\varepsilon_i$ for $i,j\in\{0,1\}$,
$$p=(\varepsilon_0+p\varepsilon_1)(p\varepsilon_0+\varepsilon_1)\quad\textrm{and}\quad\varepsilon_0=(\varepsilon_0+p\varepsilon_1)\varepsilon_0\enspace,$$
so that
	$$p\mathbb{Z}_p[\Delta]+\varepsilon_0\mathbb{Z}_p[\Delta]\subseteq (\varepsilon_0+p\varepsilon_1)\mathbb{Z}_p[\Delta].$$
	On the other hand $\varepsilon_0+p\varepsilon_1$ clearly
        belongs to
        $p\mathbb{Z}_p[\Delta]+\varepsilon_0\mathbb{Z}_p[\Delta]$,
        hence $\mathbb{Z}_p\otimes_{\mathbb{Z}}
        \Sigma_\Delta(p)=(\varepsilon_0+p\varepsilon_1)\mathbb{Z}_p[\Delta]$
        and we can take $\alpha_{p}=\varepsilon_0+p\varepsilon_1$. 
\end{itemize}
By Fr\"ohlich's theory, the morphism
$\chi^h\mapsto\big(\Det_{\chi^h}(\alpha_q)\big)_q$ represents the class of
$\Sigma_\Delta(p)$ in $\Hom_{\Omega_\Q}(R_\Delta,J(\Q))$. The basic
computations $\Det_{\chi^h}(\varepsilon_0)=\updelta_{h,e}$ and
$\Det_{\chi^h}(\varepsilon_1)=1-\updelta_{h,e}$ yield the result.
\end{proof}

\begin{remark}\label{t-rep}
Dividing $v$ by the global valued equivariant morphism $\tilde c_v$
defined by $\tilde c_v(\chi^h)=p^{1-\updelta_{h,e}}$, $1\le h\le
e$, yields another representative morphism of the class of
$\Sigma_\Delta(p)$ in $\Cl(\Z[\Delta])$, with values in $J(\Q)$: for
any $1\le h\le e$,
$$(v\tilde c_v^{-1})(\chi^h)_q=\left\{\begin{array}{ll}p^{\updelta_{h,e}-1}&
\textrm{if $q\ne p$,}\\ 
1& \textrm{if $q= p$.}\\
\end{array}\right.$$
If $q\nmid pe$, then $p^{1-\updelta_{h,e}}=\Det_{\chi^h}(\varepsilon_0+{p}\varepsilon_1)$, where $\varepsilon_0$, $\varepsilon_1$ are defined as in the proof of Lemma \ref{swan-repr1}, and satisfy $\varepsilon_0+{p}\varepsilon_1\in\mathbb{Z}_q[\Delta]^\times$. 
If $q\mid e$, then $p\in\Z_q^\times\subset\mathbb{Z}_q[\Delta]^\times$ and $\Det_{\chi^h}(p)=p$ for every $1\le h\le e$. Let $\beta\in\UU(\Z[\Delta])$ be defined by $\beta_q=p$ if $q\mid e$, $\beta_q=\varepsilon_0+{p}\varepsilon_1$ if $q\nmid pe$ and $\beta_q=1$ otherwise. Then $v\tilde c_v^{-1}\Det(\beta)$ is Ullom's representative morphism of the class of $\Sigma_\Delta(p)$ given for instance in \cite[(I.2.23)]{Frohlich-Alg_numb}. 
\end{remark}

\subsection{Hom-representatives of $(R)$ and $(S)$}\label{hom-rep-ki} 

In this subsection we use Fröhlich's construction of a representative homomorphism of the class of $\kappa(\chi^i)$, where $i$ is any integer such that $0\le i\le e-1$. These representative homomorphisms yield representatives for the classes of the torsion modules under study, namely $R$ and $S$. 

Recall that $\Delta=\langle\delta\rangle$ and that the $\oo[\Delta]$-module $\kappa(\chi^i)$ is defined to be $\kappa=\oo/\pp$ as $\oo$-module, with action of $\Delta$ given by  $\delta\cdot x=\chi^i(\delta)x=\zeta^ix$ for any $x\in \kappa$. 

\subsubsection{}
Let us fix an integer $0\leq i\leq e-1$, and let $\phi_i:\oo[\Delta]\to \kappa(\chi^i)$ be the only $\oo[\Delta]$-module homomorphism which sends $1$ to $1$, hence $\delta$ to $[\zeta^i]$, the class of $\zeta^i$ in $\kappa$. Note that $\phi_i$ is surjective and set 
$$M_i=\pp\oo[\Delta]+(\delta-\zeta^i)\oo[\Delta]\subset \oo[\Delta]\enspace.$$
Then the sequence of $\oo[\Delta]$-modules 
\begin{equation}\label{esmi}
0\to M_i\to \oo[\Delta]\stackrel{\phi_i}{\longrightarrow}\kappa(\chi^i)\to 0
\end{equation} 
is exact (since clearly $M_i\subseteq\ker(\phi_i)$ and $\#(\oo[\Delta]/M_i)=\#(\oo/\pp[\Delta]/(\delta-[\zeta^i]))=\#\oo/\pp$). 

In the next proposition, we will show, by finding explicit local generators, that $M_i$ is a locally free $\oo[\Delta]$-module. Anyway, this fact can be also shown as follows (see also the proof of \cite[Proposition 4.1]{Chinburg-deRham}): $\oo[\Delta]$ is $\Delta$-cohomologically trivial (it is a free $\Z[\Delta]$-module) and the same holds $\kappa(\chi^i)$ as observed in \S \ref{ctglob}. Therefore from the above exact sequence, we see that $M_i$ is $\Delta$-cohomologically trivial. Since it is also $\oo$-torsion free (being a submodule of $\oo[\Delta]$), we deduce that it is $\oo[\Delta]$-projective (this can be seen following the proof of \cite[Chapitre IX, Th\'eor\`eme 7]{Serre}, with $\Z$ replaced by $\oo$, and using \cite[Section 14.4, Exercice 1]{Serre-reprlin}). 
In particular, by Lemma \ref{cohomtriv} (i), the $\oo[\Delta]$-module $M_i$ is locally free and we have, by Lemma \ref{cohomtriv} (ii),
\begin{equation}\label{minus}
(\kappa(\chi^i))_{\oo[\Delta]}=(\oo[\Delta])_{\oo[\Delta]}^{-1}(M_i)_{\oo[\Delta]}=(M_i)_{\oo[\Delta]}\quad \mbox{ in }\Cl(\oo[\Delta])\enspace.
\end{equation}

For any place $\mathfrak{q}$ of $\oo$ we denote by $\oo_{\mathfrak{q}}$ the completion of $\oo$ at $\mathfrak{q}$ (note that $\oo_{\mathfrak{q}}=\mathbb{C}$ when $\mathfrak{q}$ is infinite). With a harmless abuse of notation, we will denote by $\zeta$ the image of $\zeta$ under the embedding $\oo\to\oo_\pp$.

\begin{proposition}\label{Mi}
For every place $\mathfrak{q}$ of $\oo$, $\oo_{\mathfrak{q}}\otimes_{\oo} M_i=x_{i,\mathfrak{q}}\oo_{\mathfrak{q}}[\Delta]$
with
$$x_{i,\mathfrak{q}}=\left\{\begin{array}{ll}
1&\textrm{if $\mathfrak{q}\ne \pp$,}\\
1+(p-1)\varepsilon_i&\textrm{if $\mathfrak{q}= \pp$,}
\end{array}\right.$$ 
where $\varepsilon_i=\frac{1}{e}\sum\limits_{j=0}^{e-1} \zeta^{ij}\delta^{-j}\in \oo_{\pp}[\Delta]$. 
\end{proposition}
\begin{proof}
If $\mathfrak{q}\ne\pp$, since $\pp\oo_{\mathfrak{q}}=\oo_{\mathfrak{q}}$, we have $\oo_{\mathfrak{q}}[\Delta]=\oo_{\mathfrak{q}}\otimes_{\oo} M_i$, so we can take $x_{i,\mathfrak{q}}=1$. 

Assume $\mathfrak{q}=\pp$. For every $0\leq k\leq e-1$, consider the idempotent
$$\varepsilon_k=\frac{1}{e}\sum_{j=0}^{e-1} \zeta^{kj}\delta^{-j}\in \oo_{\pp}[\Delta]\enspace.$$ 
Then \smash{$1=\sum\limits_{k=0}^{e-1} \varepsilon_k$} and $\varepsilon_h\varepsilon_k=\updelta_{h,k}\varepsilon_h$ and therefore 
$$\oo_\pp[\Delta]=\bigoplus_{k=0}^{e-1}\varepsilon_k\oo_\pp[\Delta]\enspace.$$
Now set $\tilde\phi_i=\phi_i\otimes \mathrm{id}:\oo[\Delta]\otimes_{\oo} \oo_{\pp}=\oo_{\pp}[\Delta]\to \kappa(\chi^i)\otimes_{\oo} \oo_{\pp}=\kappa(\chi^i)$, then $\tilde\phi_i(\varepsilon_i)=1$ and
$$\pp\oo_{\pp}[\Delta]+(1-\varepsilon_i )\oo_{\pp}[\Delta]\subseteq \mathrm{ker}(\tilde \phi_i)=\oo_{\pp}\otimes_{\oo} M_i\enspace.$$ 
By the above properties of the idempotents we have
$$\oo_{\pp}[\Delta]/\big(\pp\oo_{\pp}[\Delta]+(1-\varepsilon_i )\oo_{\pp}[\Delta]\big)\cong \varepsilon_i\oo_{\pp}[\Delta]/\varepsilon_i\pp\oo_{\pp}[\Delta]\cong \kappa(\chi^i)\enspace,$$
so that 
$$\pp\oo_{\pp}[\Delta]+(1-\varepsilon_i )\oo_{\pp}[\Delta]=\oo_{\pp}\otimes_{\oo} M_i\enspace.$$
We have indeed $\pp\oo_{\pp}[\Delta]+(1-\varepsilon_i)\oo_{\pp}[\Delta]=(1+(p-1)\varepsilon_i)\oo_{\pp}[\Delta]$: using the equalities
\begin{align*}
p&=(1+(p-1)\varepsilon_i)(p-(p-1)\varepsilon_i)\enspace,\\
1-\varepsilon_i&=(1+(p-1)\varepsilon_i)(1-\varepsilon_i)\enspace,
\end{align*}
we get $\pp\oo_{\pp}[\Delta]+(1-\varepsilon_i)\oo_{\pp}[\Delta]\subseteq(1+(p-1)\varepsilon_i)\oo_{\pp}[\Delta]$, since $(p)=\pp\oo_{\pp}$; the reverse inclusion follows from 
$$1+(p-1)\varepsilon_i=p\varepsilon_i +1-\varepsilon_i\enspace.$$
Therefore we can take $x_{i,\pp}=1+(p-1)\varepsilon_i$.
\end{proof}
In view of (\ref{minus}), we get the following representative homomorphism of the class of $\kappa(\chi^i)$ in $\Cl(\oo[\Delta])$.
\begin{corollary}\label{relrep}
The homomorphism $v_i$ with
values in the id\`eles group $J(\Q(\zeta))$, defined at any
place $\qq$ of $\oo$ by  
$$v_i(\chi^h)_{\mathfrak{q}}=\mathrm{Det}_{\chi^h}(x_{i,\mathfrak{q}})=\left\{
\begin{array}{ll}
p&\textrm{if $\mathfrak{q}=\pp$, $i \equiv h\pmod{e}$,}\\
1&\textrm{otherwise.}
\end{array}\right.$$
represents the class $(\kappa(\chi^i))_{\oo[\Delta]}$ in
$\mathrm{Hom}_{\Omega_{\Q(\zeta)}}(R_{\Delta},J(\Q(\zeta)))$. 
\end{corollary}
\begin{proof}
By Fröhlich's theory, Equality \eqref{minus} and Proposition \ref{Mi}, we know that $(v_i)_{\mathfrak{q}}=\mathrm{Det}(x_{i,\mathfrak{q}})$ represents the class $(\kappa(\chi^i))_{\oo[\Delta]}$. Let $h\in\{0,\ldots,e-1\}$, then
$$\mathrm{Det}_{\chi^h}(\varepsilon_i)=\frac{1}{e}\sum_{j=0}^{e-1}\zeta^{(i-h)j}=\left\{
\begin{array}{ll}
0&\textrm{if $i\not \equiv h\pmod{e}$}\\
1&\textrm{if $i \equiv h\pmod{e}$}
\end{array}\right.,$$
because $\zeta^{i-h}$ is a root of the polynomial $\sum_{j=0}^{e-1}X^j$ precisely when $(i-h)\not \equiv 0\pmod{e}$. The result follows.
\end{proof}

In order to get a representative homomorphism for the class $(\kappa(\chi^i))\in \Cl(\Z[\Delta])$, we just need to take the norm of $v_i$, namely 
$$\mathscr{N}(v_i)=\mathscr{N}_{\Q(\zeta)/\Q}(v_i)$$
represents $(\kappa(\chi^i))$ in $\mathrm{Hom}_{\Omega_{\Q}}(R_{\Delta},J(\Q(\zeta)))$, see \cite[Theorem 2]{Frohlich-Alg_numb}. Before recalling the definition of $\mathscr{N}_{\Q(\zeta)/\Q}$, we introduce some notation and make a remark.

For any $\alpha\in(\Z/e\Z)^\times$, let
$\sigma_\alpha\in\mathrm{Gal}(\Q(\zeta)/\Q)$ be the
automorphism defined by $\sigma_\alpha(\zeta)=\zeta^\alpha$. For any
integer $n$, we let $\bar n=n\mod e$ denote its class modulo $e$, and we may write
$\sigma_n$ instead of $\sigma_{\bar n}$ if $n$ is coprime with $e$.
\begin{remark}\label{cyclochar}
The map $\alpha\mapsto\sigma_\alpha$ is in fact a group isomorphism 
$$\sigma:(\Z/e\Z)^\times\to\mathrm{Gal}(\Q(\zeta)/\Q)$$
which sends the subgroup $\langle\bar p\rangle$ to the decomposition subgroup of $\pp\mid p$ (see \cite[\S 13.2, Corollary to Theorem 2]{Ireland-Rosen}), specifically $\pp^{\sigma_{p}}=\pp$. Hence, if $\Lambda\in (\Z/e\Z)^\times/\langle\bar p\rangle$, we may denote by $\pp^{\sigma_\Lambda}$ the ideal $\pp^{\sigma_\lambda}$ where
$\lambda$ is any lift of $\Lambda$ in $(\Z/e\Z)^\times$. The prime ideals above $p$ in $\oo$ are exactly the conjugates $\pp^{\sigma_{\Lambda}}$ with $\Lambda\in(\Z/e\Z)^\times/\langle\bar p\rangle$ and, for $\alpha\in (\Z/e\Z)^\times$,
\begin{equation}\label{jinalpha}
\pp^{\sigma_{\alpha}}=\pp^{\sigma_{\Lambda}}\iff \alpha\in \Lambda\enspace.
\end{equation}
\end{remark}

By definition, one has
$$\mathscr{N}(v_i)(\chi^h)_{\mathfrak{q}}=\left(\prod_{k}v_i\big((\chi^h)^{\sigma_{k}}\big)^{{\sigma_{k}}^{-1}}\right)_{\mathfrak{q}}=\prod_{k}\left(v_i(\chi^{hk})_{\mathfrak{q}^{\sigma_{k}}}\right)^{\sigma_{k}^{-1}}\enspace,$$
where the product runs over the integers $k$ such that $0\le k\le e-1$
and $k$ is coprime to $e$.

\begin{proposition}\label{nvichq}
For any place $\qq$ of $\oo$ and for any $0\le h\le e-1$, we have
$$\mathscr{N}(v_i)(\chi^{h})_{\mathfrak{q}}=\left\{\begin{array}{ll}
1&\textrm{if $\mathfrak{q}\nmid p$,}\\
p^{n(\Lambda,i,h)}&\textrm{if $\mathfrak{q}=\pp^{\sigma_{\Lambda}}$ for some $\Lambda\in (\Z/e\Z)^\times/\langle\bar p\rangle$,}\\
\end{array}\right.$$
where we have set $n(\Lambda,i,h)=\#\{\alpha\in\Lambda: \alpha\bar i=\bar h\}\enspace.$
\end{proposition}
\begin{proof}
The case $\mathfrak{q}\nmid p$ follows immediately from the above, so we assume that $\mathfrak{q}=\pp^{\sigma_{\Lambda}}$ for some $\Lambda\in (\Z/e\Z)^\times/\langle\bar p\rangle$. Then $\mathfrak{q}^{\sigma_{k}}=\pp$ if and only if ${\bar k}^{-1}\in\Lambda$, by (\ref{jinalpha}). Thus, using Corollary \ref{relrep},
$$v_i(\chi^{hk})_{\mathfrak{q}^{\sigma_{k}}}
=\left\{
\begin{array}{ll}
p&\textrm{if ${\bar k}^{-1}\in \Lambda$, $i\equiv hk\pmod{e}$;}\\
1&\textrm{otherwise.}
\end{array}\right. $$
The result follows.
\end{proof}

\subsubsection{}
Since the Hom-description (\ref{hom-des}) is a group isomorphism, the classes 
$$(R)=\prod_{i=1}^{e-1}\left(\kappa(\chi^i)\right)^i\quad\textrm{and, if $e$ is odd,} \quad (S)=\prod_{i=\frac{e+1}{2}}^{e-1}\left(\kappa(\chi^i)\right)$$
are represented respectively in $\mathrm{Hom}_{\Omega_{\Q}}(R_{\Delta},J(\Q(\zeta)))$ by the homomorphisms:
$$r=\prod_{i=1}^{e-1}\mathscr{N}(v_i)^{i}\quad\textrm{and} \quad s=\prod_{i=\frac{e+1}{2}}^{e-1}\mathscr{N}(v_i)\enspace.$$
We immediately get the following result.
\begin{corollary}\label{rs-rep}
let $0\le h\le e-1$ then, if $\mathfrak{q}\nmid p$:
$$r(\chi^h)_{\mathfrak{q}}=s(\chi^h)_{\mathfrak{q}}=1$$
and, if $\qq=\pp^{\sigma_{\Lambda}}$ for some $\Lambda\in(\Z/e\Z)^\times/\langle\bar
p\rangle$:
$$
r(\chi^h)_{\qq}=p^{\sum_{i=1}^{e-1}in(\Lambda,i,h)}\enspace,\quad
s(\chi^h)_{\qq}=p^{\sum_{i=\frac{e+1}{2}}^{e-1}n(\Lambda,i,h)}\enspace.
$$
\end{corollary}

\subsubsection{}
We now explicitly compute the numbers $n(\Lambda,i,h)$ introduced above for
$\Lambda\in(\Z/e\Z)^\times/\langle\bar p\rangle$,
$i,h\in\{0,\ldots,e-1\}$. We extend the definition to
$\Lambda'\in(\Z/e'\Z)^\times/\langle(p\!\mod e')\rangle$ for 
any divisor $e'$ of $e$, $i',h'\in\Z$, by setting:
$$n(\Lambda',i',h')=\#\{\alpha'\in\Lambda': \alpha'(i'\!\!\!\mod
e')=(h'\!\!\!\mod e')\}\enspace.$$
Of course $n(\Lambda',i',h')$ only depends on $\Lambda'$ and the
residue classes $(i'\!\!\!\mod e')$ and $(h'\!\!\!\mod e')$ of $i'$
and $h'$ modulo $e'$.
For any divisor $d$ of $e$, let $f_d$ denote the (multiplicative)
order of $p$ modulo $d$ (thus $f_e=f$). The greatest common divisor of
integers $a,b$ is denoted by $\gcd(a,b)$.

\begin{lemma}\label{nlih}
Let $\Lambda\in(\Z/e\Z)^\times/\langle\bar p\rangle$,
$i,h\in\{0,\ldots,e-1\}$, then:
\begin{enumerate}[(i)]
\item $n(\Lambda,i,h)\not=0\Rightarrow \gcd(i,e)=\gcd(h,e)$;
\item suppose $\gcd(i,e)=\gcd(h,e)=d$ and set $i=di'$, $h=dh'$, $e=de'$, one
  has
$$n(\Lambda,i,h)
=\left\{\begin{array}{ll}
f/f_{e'} & \textrm{if $(h'\!\!\!\mod e')\in(i'\!\!\!\mod e')\Lambda'$,}\cr
0 & \textrm{otherwise,}

\end{array}\right.
$$
where $\Lambda'=(\Lambda\mod e')\in(\Z/e'\Z)^\times/\langle(p\!\mod e')\rangle$.
\end{enumerate}
\end{lemma}
\begin{proof}
Suppose $n(\Lambda,i,h)\not=0$. Let $\alpha\in\Lambda$ be such that
$\alpha\bar i=\bar h$ and let $a\in\alpha$ (so $a\in\Z$), then
$ai\equiv h\pmod{e}$ and the equality $\gcd(i,e)=\gcd(h,e)$ follows since $\gcd(a,e)=1$. 

We now assume the condition of assertion (ii) is satisfied and use the
same notations. If $d=e$,
$n(\Lambda,i,h)=n(\Lambda,0,0)=\#\Lambda=f=f/f_1$ and $0\in 0\Lambda'$
is always satisfied. Otherwise, $(i',e')=1$ and one has, for any
$\alpha'\in(\Z/e'\Z)^\times$:
$$\alpha'(i'\!\!\!\mod e')=(h'\!\!\!\mod e')\Longleftrightarrow\alpha'=(h'\!\!\!\mod e')(i'\!\!\!\mod e')^{-1}\enspace,$$
hence $n(\Lambda',i',h')=1$ or $0$ depending on whether $(h'\!\!\!\mod
e')(i'\!\!\!\mod e')^{-1}$ belongs to $\Lambda'$ or not. We thus only have
to show that 
\begin{equation}\label{nlih'}
n(\Lambda,i,h)=\frac{f}{f_{e'}}\,n(\Lambda',i',h')\enspace.
\end{equation}
As above, let $\alpha\in\Lambda$ and $a\in\alpha$. We may rewrite
$n(\Lambda,i,h)$ as 
\begin{align*}
n(\Lambda,i,h)
&=\#\{0\leq k\leq f-1: \alpha\overline{p}^k\bar i=\bar h\}\\
&=\#\{0\leq k\leq f-1: a{p}^ki\equiv h\pmod{e}\}\\
&=\#\{0\leq k\leq f-1: a{p}^ki'\equiv h'\pmod{e'}\}
\enspace.
\end{align*}
Note that, if we set $\alpha'=(\alpha\mod e')$, then $a\in\alpha'$ and
$\alpha'\in\Lambda'$, hence, similarly:
$$n(\Lambda',i',h')=\#\{0\leq k\leq f_{e'}-1: a{p}^ki'\equiv h'\pmod{e'}\}
\enspace.$$
The result follows since $f_{e'}$ is the order of $p$ in $(\Z/e'\Z)^\times$.
\end{proof}

\subsection{The contents of $s$ and $r$}\label{section:content}
In this subsection we compute the contents of the id\`eles $\mathscr{N}(v_i)(\chi^h)$, $s(\chi^h)$ and $r(\chi^h)$ for $0\le i,h\le e-1$. Recall that the content of an id\`ele $x=(x_\qq)_\qq\in J(\Q(\zeta))$ is the fractional ideal $cont(x)=\prod_\qq\qq^{\val_\qq(x_\qq)}$ of $\Q(\zeta)$, where $\val_\qq$ is the $\qq$-valuation and the product runs over finite prime ideals $\qq$ of $\oo$. 

Since the valuation of $p$ at a prime ideal $\qq=\pp^{\sigma_{\Lambda}}$ with $\Lambda\in(\Z/e\Z)^\times/\langle\bar p\rangle$ equals $1$, it follows from Proposition \ref{nvichq} and Corollary \ref{rs-rep} that 
\begin{align}
cont(\mathscr{N}(v_i)(\chi^h))& = \pp^{\sum\limits_{\Lambda}n(\Lambda,i,h)\,\sigma_{\Lambda}}\label{cnvi}\\
cont(s(\chi^h))& = \pp^{\sum\limits_{\Lambda}\sum_{i=\frac{e+1}{2}}^{e-1}n(\Lambda,i,h)\,\sigma_{\Lambda}}\label{cs}\\
cont(r(\chi^h))& = \pp^{\sum\limits_{\Lambda}\sum_{i=1}^{e-1}in(\Lambda,i,h)\,\sigma_{\Lambda}}\label{cr}
\end{align} 
where in each sum $\Lambda$ runs over $(\Z/e\Z)^\times/\langle\bar p\rangle$.

\subsubsection{}\label{contentnvi}
Since $\mathscr{N}(v_i)$, $s$ and $r$ are $\Omega_\Q$-equivariant, their values on $R_\Delta$ are determined by the values at $\chi^d$, with $d\mid e$. Namely, if $h$ is an integer and $d$ is the greatest common divisor of $h$ and $e$, we write $h=dh'$ and get 
\begin{equation}\label{equivariance}
\mathscr{N}(v_i)(\chi^h)=\mathscr{N}(v_i)\big((\chi^d)^{\sigma_{h'}}\big)=\mathscr{N}(v_i)(\chi^d)^{\sigma_{h'}}
\end{equation}
and analogously for $s$ and $r$.

For any $d\mid e$, write $e=de'$ and set $\zeta_{e'}=\zeta^{d}$ (thus $\zeta_e=\zeta$); for $\alpha'\in(\Z/e'\Z)^\times$, let $\sigma_{e',\alpha'}\in\mbox{Gal}(\Q(\zeta_{e'})/\Q)$ be the automorphism sending $\zeta_{e'}$ to $\zeta_{e'}^{\alpha'}$ (thus $\sigma_{e,\alpha'}=\sigma_{\alpha'}$). Since $e'\mid e$, $\Q(\zeta_{e'})\subseteq\Q(\zeta)$, hence $\sigma_{e',\alpha'}$ can be lifted in $\mbox{Gal}(\Q(\zeta)/\Q)$ (in $\varphi(e)/\varphi(e')$ different ways). To ease notation, if $j$ is an integer with $(j,e')=1$, we may write $\sigma_{e',j}$ instead of $\sigma_{e',(j\!\!\mod e')}$.

We also set
$\oo_{e'}=\Z[\zeta_{e'}]$ and $\pp_{e'}=\pp\cap \oo_{e'}$ (thus $\oo_e=\oo$ and $\pp_e=\pp$).
If $\alpha'\in(\Z/e'\Z)^\times$, the ideal $\pp_{e'}^{\sigma_{e',\alpha'}}$ only depends on the class of $\alpha'$ modulo $(p\!\!\mod e')$. So if $\Lambda'\in(\Z/e'\Z)^\times/\langle(p\!\!\mod e')\rangle$, we denote by $\pp_{e'}^{\sigma_{\Lambda'}}$ the ideal $\pp_{e'}^{\sigma_{e',\alpha'}}$ where $\alpha'$ is any lift of $\Lambda'$ in $(\Z/e'\Z)^\times$.  
\begin{lemma}\label{decomposition}
Let $d\mid e$ and set $e=de'$. Let $\Lambda'\in(\Z/e'\Z)^\times/\langle(p\mod e')\rangle$, then
$$\pp^{\sum\limits_{\Lambda\in\Lambda'}\sigma_\Lambda}=\pp_{e'}^{\sigma_{\Lambda'}}\oo\enspace,$$
where the sum is on the elements $\Lambda$ of the coset $\Lambda'$ in $(\Z/e\Z)^\times/\langle(p\mod e)\rangle$.
\end{lemma}
\begin{proof}
Let $\Lambda\in(\Z/e\Z)^\times/\langle\bar p\rangle$ and $\Lambda'\in(\Z/e'\Z)^\times/\langle(p\!\!\mod e')\rangle$ then, since $\pp\mid\pp_{e'}\oo$,
$$\Lambda\in\Lambda'\ \Rightarrow\ 
{\sigma_\Lambda}_{|_{\Q(\zeta_{e'})}}=\sigma_{\Lambda'}\ \Rightarrow\ 
\pp^{\sigma_\Lambda}\mid\pp_{e'}^{\sigma_{\Lambda'}}\oo\enspace.$$
It follows that $\pp^{\sum_{\Lambda\in\Lambda'}\sigma_\Lambda}\mid\pp_{e'}^{\sigma_{\Lambda'}}\oo$. Since $\Q(\zeta)/\Q(\zeta_{e'})$ is unramified at $\pp_{e'}$, the number of primes above $\pp_{e'}$ in $\oo$ equals $\frac{\varphi(e)}{\varphi(e')}/\frac{f}{f_{e'}}=\frac{\#(\Z/e\Z)^\times/\langle\bar p\rangle}{\#(\Z/e'\Z)^\times/\langle(p\!\!\mod e')\rangle}=\#\Lambda'\,$ (as a coset in $(\Z/e\Z)^\times/\langle\bar p\rangle$). The result follows.
\end{proof}

\begin{proposition}\label{prop:contentnvi}
Let $d\mid e$ and set $e=de'$, then
$$cont(\mathscr{N}(v_i)(\chi^d))=
\left\{\begin{array}{ll}
\oo & \mbox{if }\gcd(i,e)\not=d \\
\left(\pp_{e'}^{\sigma_{\Lambda'_i}}\oo\right)^{f/f_{e'}} & \mbox{if } \gcd(i,e)=d
\end{array}\right.$$
where $i=di'$ and $\Lambda'_i\in(\Z/e'\Z)^\times/\langle(p\mod e')\rangle$ is such that $(i'\mod e')^{-1}\in\Lambda'_i$.
\end{proposition}
\begin{proof}
Since $\gcd(d,e)=d$, the result is clear from \eqref{cnvi} and Lemma \ref{nlih} in the case $\gcd(i,e)\not=d$, hence we now assume $\gcd(i,e)=d$ and write $i=di'$. Then $i'$ and $e'$ are coprime so let $\Lambda'_i\in(\Z/e'\Z)^\times/\langle(p\mod e')\rangle$ be such that $(i'\mod e')^{-1}\in\Lambda'_i$.
From Lemma \ref{nlih} we know that $n(\Lambda,i,d)=f/f_{e'}$ if $\Lambda\in\Lambda'_i$, $0$ otherwise, hence from \eqref{cnvi} we get
$$cont(\mathscr{N}(v_i)(\chi^d))=\pp^{\frac{f}{f_{e'}}\sum\limits_{\Lambda\in\Lambda'_i}\sigma_\Lambda}$$
and the result follows using Lemma \ref{decomposition}.
\end{proof}

\begin{remark}\label{kappacontent}
Using (\ref{equivariance}), it follows that, for every $i,h\in\{0,\,\ldots,\, e-1\}$, the ideal $cont(\mathscr{N}(v_i)(\chi^h))$ is either trivial or, if $\gcd(i,e)=\gcd(h,e)=d$, the extension to $\oo$ of an ideal of $\oo_{e'}$ ($e'=e/d$) whose absolute norm is congruent to $1$ modulo $e$. (Indeed, by definition of $f_{e'}$, the absolute norm of $\pp_{e'}^{f/f_{e'}}$ and of its conjugates is $(p^{f/f_{e'}})^{f_{e'}}=p^f$ which is congruent to $1$ modulo $e$.)
\end{remark}

\subsubsection{}
We begin by computing the content of $r$: the strategy for the content of $s$ will be similar but the calculations for $r$ are simpler.

For any divisor $e'$ of $e$, we denote by $\ZZ_{e'}$ the subgroup of $(\Z/e\Z)^\times$ of elements congruent to $1$ modulo $e'$, namely 
$$\ZZ_{e'}=\sigma^{-1}\big(\Gal(\Q(\zeta)/\Q(\zeta_{e'}))\big)\enspace,$$
where $\sigma:\Gal(\Q(\zeta)/\Q(\zeta_{e'}))\to(\Z/e\Z)^\times$ is the isomorphism of Remark \ref{cyclochar}. 
We also introduce the relative norm and Stickelberger's element: 
\begin{align*}
N_{e,e'}&=\sum_{\alpha\in \ZZ_{e'}}\sigma_\alpha\quad\in\ \Z[\mathrm{Gal}(\Q(\zeta)/\Q(\zeta_{e'}))]\enspace;\\
\Theta_{e'}&=\frac{1}{e'}\sum_{\substack{1\le j\le e'-1\\(j,e')=1}}j{\sigma}_{e',j}^{-1}\quad\in\ \Q[\mathrm{Gal}(\Q(\zeta_{e'})/\Q)]\enspace.
\end{align*}
Note that $N_{e,1}$ equals the absolute norm $N_{\Q(\zeta)/\Q}$. 

\begin{proposition}\label{contentr}
Let $d\mid e$ and set $e=de'$, then
$$cont(r(\chi^d))=\pp^{e\Theta_{e'}N_{e,e'}}\enspace.$$
\end{proposition}
\begin{proof}
We begin with the right-hand side. First $\pp^{N_{e,e'}}=\pp_{e'}^{f/f_{e'}}\oo$, so that $\pp^{e\Theta_{e'}N_{e,e'}}$ equals $\pp_{e'}\oo$ to the exponent:
$$d\frac{f}{f_{e'}}\sum_{\Lambda'}\Bigg(\sum_{\substack{1\le j\le e'-1\\j\!\!\!\mod e'\in\Lambda'}}j\Bigg){\sigma}_{\Lambda'}^{-1}\enspace,$$
where the first sum is on $\Lambda'\in(\Z/e'\Z)^\times/\langle(p\!\!\mod e')\rangle$.

We go on with the left-hand side. We start from Formula (\ref{cr}). Since $\gcd(d,e)=d$, using Lemma \ref{nlih}, $cont(r(\chi^d))$ equals $\pp$ to the exponent:
$$\sum\limits_{\Lambda}\Bigg(\sum\limits_{\substack{1\le i'\le e'-1\\1\in(i'\Lambda\!\!\!\mod e')}}di'\frac{f}{f_{e'}}\Bigg)\sigma_{\Lambda}
=d\frac{f}{f_{e'}}\sum\limits_{\Lambda'}\Bigg(\sum\limits_{\substack{1\le i'\le e'-1\\1\in(i'\!\!\!\mod e')\Lambda'}}i'\Bigg)\sum\limits_{\Lambda\in\Lambda'}\sigma_{\Lambda}\enspace,
$$
where $\Lambda$ runs over $(\Z/e\Z)^\times/\langle\bar
p\rangle$ in the first sum to the left, $\Lambda'$ runs over
$(\Z/e'\Z)^\times/\langle(p\!\!\mod e')\rangle$ in the
first sum to the right, and is seen as a coset of
$(\Z/e\Z)^\times/\langle\bar p\rangle$ in the last sum
to the right (for the reduction modulo $e'$). Using Lemma \ref{decomposition},
we get that $cont(r(\chi^d))$ equals $\pp_{e'}\oo$ to the exponent:
$$d\frac{f}{f_{e'}}\sum\limits_{\Lambda'}\Bigg(\sum\limits_{\substack{1\le
    i'\le e'-1\\1\in(i'\!\!\!\mod
    e')\Lambda'}}i'\Bigg)\sigma_{\Lambda'}
=d\frac{f}{f_{e'}}\sum\limits_{\Lambda'}\Bigg(\sum\limits_{\substack{1\le
    i'\le e'-1\\(i'\!\!\!\mod
    e')\in\Lambda'^{-1}}}i'\Bigg)\sigma_{\Lambda'}
\enspace,$$
and the equality follows.
\end{proof}

\subsubsection{}\label{further}
To deal with $s$, we need to introduce further elements in various group
algebras. For any divisors $d$ and $e'$ of $e$, let 
$$\HH_{e'}=\{\beta'\in (\Z/e'\Z)^\times\,:\,\exists\,\,  b\in\beta',\, \frac{e'+1}{2}\leq b\leq e'-1\}$$
and set
\begin{align*}
H_{e'}&=\sum_{\alpha'\,\in\,\HH_{e'}}\sigma_{e',\alpha'}^{-1}\quad\in\ \Z[\mathrm{Gal}(\Q(\zeta_{e'})/\Q)]\enspace,\\
H_{e,d}&=\sum_{\substack{\alpha\in(\Z/e\Z)^\times\\ \alpha\bar d\,\in\,\HH_e} }\sigma_{\alpha}^{-1}\quad\in\ \Z[\mathrm{Gal}(\Q(\zeta)/\Q)]\enspace.
\end{align*}
Note that $H_e=H_{e,1}$.
\begin{lemma}\label{contents1}
Let $d\mid e$, then $cont(s(\chi^d))=\pp^{H_{e,d}}$.
\end{lemma}
\begin{proof}
Rewrite $H_{e,d}$ as
$$H_{e,d}=\sum_{\Lambda}\sum_{\substack{\beta\in\Lambda\\ \beta^{-1}\bar d\,\in\,\HH_e} }\sigma_{\beta}\enspace,$$
where $\Lambda$ runs over $(\Z/e\Z)^\times/\langle\bar p\rangle$. It follows that
$$\pp^{H_{e,d}}=\pp^{\sum_{\Lambda}\#\{\beta\in\Lambda\,:\,\beta^{-1}\bar d\,\in\,\HH_e\}\sigma_{\Lambda}}\enspace.$$
But $\#\{\beta\in\Lambda\,:\,\beta^{-1}\bar d\in \HH_e\}=\sum\limits_{i=\frac{e+1}{2}}^{e-1}n(\Lambda,i,d)$, which yields the result using (\ref{cs}).
\end{proof}

\begin{lemma}\label{contents2} 
Let $d\mid e$ and set $e=de'$, then 
$$H_{e,d}={H}_{e'}N_{e,e'}\quad\mathrm{and}\quad H_{e'}=(2-\sigma_{e',2})\Theta_{e'}\enspace.$$
\end{lemma}
\begin{remark}
Note that the first equality is not ambiguous, even though it contains a slight abuse of notation: $H_{e'}$ belongs to $\Z[\Gal(\Q(\zeta_{e'})/\Q)]$ whereas $N_{e,e'}$ and $H_{e,d}$ belong to $\Z[\Gal(\Q(\zeta_{e})/\Q)]$. Nevertheless, if $\alpha'\in(\Z/e'\Z)^\times$, then $\sigma_{\alpha}N_{e,e'}$ does not depend on the choice of a lift $\alpha$ of $\alpha'$ in $(\Z/e\Z)^\times$. Indeed
$$\sigma_{\alpha}N_{e,e'}=\sum_{\beta\in \ZZ_{e'}}\sigma_{\alpha\beta}$$
is the sum of all the lifts of $\sigma_{e',\alpha'}$ in $\Gal(\Q(\zeta)/\Q)$. Thus in order to get an equality in $\Z[\Gal(\Q(\zeta_{e})/\Q)]$, one just has to replace each term $\sigma_{e',\alpha'^{-1}}N_{e,e'}$ of the sum $H_{e'}N_{e,e'}$ by $\sigma_{\alpha^{-1}}N_{e,e'}$, where $\alpha$ is any lift of $\alpha'$ in $(\Z/e\Z)^\times$.
\end{remark} 
\begin{proof}
We begin with the first equality. We suppose we have fixed a lift $\alpha\in(\Z/e\Z)^\times$ of each $\alpha'\in(\Z/e'\Z)^\times$. In view of the previous remark,
$$H_{e'}N_{e,e'}=\sum_{\substack{\alpha'\in\HH_{e'}\\\beta\in
    \ZZ_{e'}}}\sigma_{\alpha^{-1}\beta}=\sum_{\substack{\alpha'\in\HH_{e'}\\\beta\in
    \ZZ_{e'}}}\sigma_{\alpha\beta}^{-1}\enspace,$$
since $\ZZ_{e'}$ is a subgroup of $\Gal(\Q(\zeta)/\Q)$. We are thus left to
show that the map 
\begin{equation}\label{map}
(\alpha',\beta)\in\HH_{e'}\times \ZZ_{e'}\longmapsto\alpha\beta\in\{\gamma\in(\Z/e\Z)^\times:\gamma\bar d\in\HH_{e}\}
\end{equation}
is a well-defined one-to-one correspondence.

We first show that the if $(\alpha',\beta)\in \HH_{e'}\times \ZZ_{e'}$, then indeed $\alpha\beta\in\{\gamma\in(\Z/e\Z)^\times:\gamma\bar d\in\HH_{e}\}$. Of course $\alpha\beta\in(\Z/e\Z)^\times$, so we have to show that $\alpha\beta\bar
d\in\HH_e$. Let $a\in\alpha'$ be such that $\frac{e'+1}{2}\le a\le e'-1$. Observe that $ad\in \alpha\beta\overline{d}$. In fact $\beta\in \ZZ_{e'}$, there exists $t\in\Z$ such that $1+te'\in\beta$ and clearly
$$ad\equiv ad(1+te') \pmod{e}.$$ 
Since clearly $ad(1+te')\in \alpha\beta\overline{d}$, we deduce $ad\in \alpha\beta\overline{d}$. Now $\frac{e+1}{2}\le \frac{e+d}{2}\le ad\le e-d\le e-1$, hence $\alpha\beta\bar
d\in\HH_e$. Thus the map in (\ref{map}) is well-defined.

To show that it is one-to-one, we define its inverse. Suppose $\gamma\in(\Z/e\Z)^\times$ is such that $\gamma\bar
d\in\HH_e$. Since $d\mid e$, there exists $c\in\gamma$ such that
$\frac{e+1}{2}\le cd\le e-1$, hence $\frac{e'+1}{2}\le c\le e'-1$. This shows that 
$\alpha'=(c\mod e')\in(\Z/e'\Z)^\times$ belongs to $\HH_{e'}$. In particular $\alpha'$ is the image of $\gamma$ in $(\Z/e'\Z)^\times$. Since $\alpha'$ is also the image of $\alpha$ in $(\Z/e'\Z)^\times$, there exists a unique $\beta\in
\ZZ_{e'}$ such that $\gamma=\alpha\beta$. It is clear that the map sending $\gamma$ to $(\alpha',\beta)$ is the inverse of (\ref{map}).

We now show the second equality of the lemma. To simplify notations
in this proof, we write $\sigma_{k}$ instead of $\sigma_{e',k}$. We note that
$\sigma_{2}=\sigma_{\frac{e'+1}{2}}^{-1}$ and compute
$$e'(2-\sigma_{2})\Theta_{e'}=\sum_{\substack{1\le j\le e'-1\\(j,e')=1}}2j{\sigma}_{j}^{-1}-\sum_{\substack{1\le j\le e'-1\\(j,e')=1}}j{\sigma}_{\frac{e'+1}{2}j}^{-1}\enspace.$$
Let $j,j'\in\{k\in\Z:1\le k\le e'-1, (k,e')=1\}$, then
$j'\equiv\frac{e'+1}{2}j\pmod{e'}$ if and
only if $j\equiv 2j'\pmod{e'}$, hence the coefficient of
$\sigma_{j'}^{-1}$ in the second sum above is $2j'$ if $1\le
j'\le\frac{e'-1}{2}$, $2j'-e'$ otherwise. Therefore
$$e'(2-\sigma_{2})\Theta_{e'}=e'\sum_{\substack{\frac{e'+1}{2}\le j\le
    e'-1\\(j,e')=1}}{\sigma}_{j}^{-1}$$
and the result follows.

\end{proof}
Combining Lemmas \ref{contents1} and \ref{contents2} yields:
\begin{proposition}\label{contents}
Let $d\mid e$ and set $e=de'$, then
$$cont(s(\chi^d))=\pp^{(2-\sigma_{e',2}){\Theta}_{e'}N_{e,e'}}\enspace.$$
\end{proposition}


\subsection{Arithmetically realizable classes}\label{section:arcburns}

The result we have obtained so far in this section, combined with those of Section \ref{ctsec} allow us to give an alternative proof of a result of Burns, concerning the Galois module structure of ideals in tame locally abelian extensions of number fields. 

We start by recalling Burns's definitions and notation. Let $G$ be a finite group and let $\RRR^a_\Q(G)$ denote the subgroup of $\Cl(\Z[G])$ which is generated by classes of $G$-stable (integral) ideals of tame locally abelian $G$-Galois extensions of number fields. Let also $\RRR^a_{\Q,0}(G)$ denote the subgroup of $\RRR^a_\Q(G)$ consisting of classes of rings of integers of tame locally abelian $G$-Galois extensions of number fields. To describe the difference between $\RRR^a_{\Q,0}(G)$ and $\RRR^a_\Q(G)$, consider the subgroup $\Cl(\Z[G])_{\mathscr{C},\Sigma}$ of $\Cl(\Z[G])$ which is defined as follows. Let first $C$ be a cyclic subgroup of $G$ and denote by $\Cl(\Z[C])_\Sigma$ the subgroup of $\Cl(\Z[C])$ ge\-ne\-ra\-ted by classes which can be represented in the ideal-theoretic Hom-description (see \cite[Note 4 to Chapter I]{Frohlich-Alg_numb}) by functions whose value on a character $\chi:C\to \overline{\Q}$ is a fractional ideal of $\Q(\chi)$ supported only above rational primes congruent to $1$ modulo $\#C$. If $\mathscr{C}$ is the set of cyclic subgroups of $G$, then
$$\Cl(\Z[G])_{\mathscr{C},\Sigma}=\prod_{C\in\mathscr{C}}\mathrm{Ind}_C^G\Cl(\Z[C])_\Sigma\enspace.$$ 

The result of Burns \cite[Theorem 1.1]{BurnsARC} which we are going to reprove reads as follows.

\begin{theorem}
There is an inclusion
$$\RRR^a_\Q(G)\subseteq \RRR^a_{\Q,0}(G)\cdot\Cl(\Z[G])_{\mathscr{C},\Sigma}\enspace.$$
\end{theorem}
\begin{remark}
Burns also proves that one has equality in the above theorem when $G$ is abelian. Moreover his result holds with $\Q$ replaced by any number field $L$ which is absolutely unramified at primes dividing $\#G$ and $\Z$ replaced by the ring of integers of $L$.
\end{remark}

\begin{proof}
Let $N/E$ be a tame locally abelian $G$-Galois extension of number fields and let $\I$ be a $G$-stable ideal of $\mathcal{O}_N$. We have to show that 
$$(\I)\in \RRR^a_{\Q,0}(G)\cdot\Cl(\Z[G])_{\mathscr{C},\Sigma}\enspace.$$
Of course it is enough to show that
$$(\OO_N/\I)=(\I)(\OO_N)^{-1}\in\Cl(\Z[G])_{\mathscr{C},\Sigma}\enspace.$$ 
Note that, for any prime $\PP$ of $\OO_N$, the inertia subgroup $I_{\PP}$ of $\PP$ in $N/E$ is cyclic, since $N/E$ is tame. Thanks to Proposition \ref{preburns}, we are then reduced to show that, for every $0\leq i\leq e_P-1$, $(\oo_{e_{P}}/\pp(\chi_{\PP}^i))\in\Cl(\Z[I_\PP])_\Sigma$ with notation as in \S \ref{reductionproof}. Let $\mathscr{N}(v_i)$ be the representative homomorphism of $\big((\oo_{e_{P}}/\pp)(\chi_{\PP}^i)\big)$ described in Proposition \ref{nvichq}. Since $\mathscr{N}(v_i)$ is $\Omega_{\Q}$-equivariant, we deduce that $\mathscr{N}(v_i)(\chi_{\PP}^h)\in J(\Q(\chi_\PP^h))$. By Propositon \ref{nvichq}, $\mathscr{N}(v_i)(\chi_{\PP}^h)$ has component $1$ at every prime dividing $e_P$. Thus to obtain a function representing $\big((\oo_{e_{P}}/\pp)(\chi_{\PP}^i)\big)$ in the ideal-theoretic Hom-description, one only needs to take the content of $\mathscr{N}(v_i)(\chi_{\PP}^h)$, as an ideal of $\Q(\chi_\PP^h)$ (see for instance \cite[p. 429]{FrohlichAGMSTE}). By Remark \ref{kappacontent}, the content of $\mathscr{N}(v_i)(\chi_{\PP}^h)$, as an ideal of $\Q(\chi_{\PP}^h)$, has absolute norm $1$ modulo $e_{P}=\#I_{\PP}$.
 This implies that $\big((\oo_{e_{\PP}}/\pp)(\chi_{\PP}^i)\big)\in \Cl(\Z[I_{\PP}])_\Sigma$, thanks to the following result, whose proof uses the arguments of \cite[pp. 388-389]{BurnsARC}.
\end{proof}

\begin{proposition}
Let $C$ be a cyclic group and let $\Cl'(\Z[C])_\Sigma$ be the subgroup of $\Cl(\Z[C])$ generated by classes which can be represented in the ideal-theoretic Hom-description by functions whose values are fractional ideals of absolute norm congruent to $1$ modulo $c=\#C$. Then $\Cl'(\Z[C])_{\Sigma}= \Cl(\Z[C])_{\Sigma}$. 
\end{proposition}
\begin{proof}
Let $\chi:C\to \overline{\Q}^\times$ be a fixed injective character of $C$. Then the $\Omega_\Q$-conjugacy classes of characters of $C$ are represented by $\chi^d$, with $d$ running through the divisors of $c$. In particular, as remarked in \S\ref{contentnvi} for the idèle-theoretic Hom-description, a representative homomorphism of an element of $\Cl(\Z[C])$ in the ideal-theoretic Hom-description is determined by its value on $\chi^d$ for every divisor $d$ of $c$. For any such divisor $d$, let $\Cl_\mathfrak{c}(\Q(\chi^d))$ denote the ray class group of $\Q(\chi^d)$ modulo $\mathfrak{c}$, where $\Q(\chi^d)$ is the extension of $\Q$ generated by the values of $\chi^d$ and $\mathfrak{c}$ is the product of the ideal generated by $c$ and the archimedean primes of $\Q(\chi^d)$. Consider the homomorphism
$$\pi_{C}:\bigoplus_{d\mid c}\Cl_\mathfrak{c}(\Q(\chi^d))\to \Cl(\Z[C])$$
which is defined as follows. For any $d\mid c$, let $a_d$ be a class in $\Cl_\mathfrak{c}(\Q(\chi^d))$ and choose a fractional ideal $\mathfrak{a}_d$ of $\Q(\chi^d)$ with the following properties: 
\begin{itemize}
\item[(i)] the numerator and denominator of $\mathfrak{a}_d$ are coprime with $c$; 
\item[(ii)] the class of $\mathfrak{a}_d$ in $\Cl_\mathfrak{c}(\Q(\chi^d))$ is $a_d$. 
\end{itemize}
Define $\pi_C((a_d)_{d\mid c})\in \Cl(\Z[C])$ to be the class represented by the function which, for any divisor $d\mid c$, sends $\chi^d$ to $\mathfrak{a}_d$. In fact $\pi_C((a_d)_{d\mid c})$ is independent of the choice of the ideals $\mathfrak{a}_d$ with the above properties, as we now show. For every $d\mid c$, choose a fractional ideal $\mathfrak{b}_d$ with properties (i) and (ii). Then, for every $d\mid c$, there exists a totally positive element $x_d\in \Q(\chi^d)$ such that $x_d\equiv 1 \pmod {(c)}$ and $\mathfrak{b}_d=x_d\mathfrak{a}_d$. We claim that the map $\chi^d\mapsto x_d\Z[\chi^d]$ lies in the denominator of the ideal theoretic Hom-description of $\Cl(\Z[C])$. In other words we have to show that for every place $\qq$ dividing $\mathfrak{c}$, there exists $u_q\in\Z_q[C]^\times$ such that $\Det_{\chi^d}(u_q)=x_d$ for every $d\mid c$, where $q$ is the place of $\Q$ below $\qq$ and we identify $x_d$ with its image in $\Q(\chi^d)_\qq^\times$. When $\qq$ is archimedean this follows from the second part of \cite[I, Proposition 2.2]{Frohlich-Alg_numb}, while when $\qq$ is finite it is a consequence of the following lemma.

\begin{lemma}
Let $q$ be a rational prime dividing $c$ and let $\xi_{c}$ be a primitive $c$th root of unity in an algebraic closure $\overline{\Q}_q$ of $\Q_q$. Then
$$\Hom_{\Omega_{\Q_q}}(R_{C,q},1+c\Z_q[\xi_c])\subset \Det(\Z_q[C]^\times).$$
\end{lemma}
\begin{proof} 
Let $\widehat C$ be the group of $\overline{\Q}_q$-characters of $C$. Consider the isomorphism of $\overline{\Q}_q$-algebras
$$\overline{\Q}_q[C]\to \mathrm{Map}(\widehat C,\overline{\Q}_q)$$
which sends $\sum_{\gamma\in C} a_\gamma \gamma\in \Q_q[C]$ to the map $\chi\mapsto \sum_{\gamma\in C}a_\gamma\chi(\gamma)$. Note that $\Omega_{\Q_q}$ acts on both $\overline{\Q}_q[C]$ (via its action on $\overline{\Q}_q$) and $\mathrm{Map}(\widehat C,\overline{\Q}_q)$ (via its actions on $\widehat C$ and $\overline{\Q}_q$) and the above isomorphism is $\Omega_{\Q_q}$-equivariant. Taking invariants we get an isomorphism of $\Q_q$-algebras
$$\Q_q[C]\to \mathrm{Map}_{\Omega_{\Q_q}}(\widehat C,\Q_q).$$
In particular, if $\mathfrak{M}_q$ is the maximal order of $\Q_q[C]$, we get an isomorphism of rings 
$$\mathfrak{M}_q\to\mathrm{Map}_{\Omega_{\Q_q}}(\widehat C,\Z_q[\xi_c]),$$
since $\Z_q[\xi_c]$ is the maximal order of $\Q_q(\xi_c)$. 
From this we get a group isomorphism 
$$1+c\mathfrak{M}_q\to \mathrm{Map}_{\Omega_{\Q_q}}(\widehat C,1+c\Z_q[\xi_c]).$$
Extending maps on characters by linearity, we can identify $\mathrm{Map}_{\Omega_{\Q_q}}(\widehat C,1+c\Z_q[\xi_c])$ with $\Hom_{\Omega_{\Q_q}}(R_{C,q},1+c\Z_q[\xi_c])$. Thus the above isomorphism is the usual $\Det$ map
$$\Det:1+c\mathfrak{M}_q\to \Hom_{\Omega_{\Q_q}}(R_{C,q},1+c\Z_q[\xi_c]).$$ 
This implies in particular that $1+c\mathfrak{M}_q\subset \mathfrak{M}_q^\times$, since $1+c\Z_q[\xi_c]\subset \Z_q[\xi_c]^\times$ because $q\mid c$. Moreover $c\mathfrak{M}_q\subset \Z_q[C]$ (see \cite[Theorem 41.1]{Reiner}) and in particular $1+c\mathfrak{M}_q\subset \Z_q[C]\cap \mathfrak{M}_q^\times=\Z_q[C]^\times$ (for the last equality see \cite[Exercise 4, Section 25]{Reiner}). Hence
$$\Hom_{\Omega_{\Q_q}}(R_{C,q},1+c\Z_q[\xi_c])= \Det(1+c\mathfrak{M}_q)\subset \Det(\Z_q[C]^\times).$$
\end{proof}

So the map $\pi_C$ is well-defined and one easily checks that it is a group homomorphism. Moreover it follows immediately from the ideal-theoretic Hom-description that $\pi_C$ is surjective. 

Now let $\Cl_{\mathfrak{c},\Sigma}(\Q(\chi^d))$ (resp. $\Cl'_{\mathfrak{c},\Sigma}(\Q(\chi^d))$) be the subgroup of $\Cl_\mathfrak{c}(\Q(\chi^d))$ which is generated by the classes of fractional ideals of $\Q(\chi^d)$ supported only above rational primes congruent to $1$ modulo $c$ (resp. by the classes of fractional ideals of $\Q(\chi^d)$ which have absolute norm congruent to $1$ modulo $c$). Observe that $\pi_C$ maps $\bigoplus_{d\mid c}\Cl_{\mathfrak{c},\Sigma}(\Q(\chi^d))$ (resp. $\bigoplus_{d\mid c}\Cl'_{\mathfrak{c},\Sigma}(\Q(\chi^d))$) surjectively onto $\Cl(\Z[C])_\Sigma$ (resp. $\Cl'(\Z[C])_\Sigma$). We claim that $\Cl_{\mathfrak{c},\Sigma}'(\Q(\chi^d))\subseteq \Cl_{\mathfrak{c},\Sigma}(\Q(\chi^d))$ for any divisor $d$ of $c$. In fact, fix $d\mid c$ and let for the moment $x\in \Cl_{\mathfrak{c}}(\Q(\chi^d))$ be any class. Then Chebotarev's density theorem implies that the set of primes of $\Q(\chi^d)$ belonging to $x$ has positive Dirichlet's density $\delta>0$ (see for instance \cite[Chapter VII, Theorem 7.2]{MilneCFT}). Note that $\delta$ is also equal to the density of the set of primes of $\Q(\chi^d)$ belonging to $x$ and splitting completely in $\Q(\chi^d)/\Q$ (\cite[Chapter IV, Corollary 4.6]{MilneCFT}). In particular there exists a prime $\mathfrak{p}$ of $\Q(\chi^d)$ representing $x$ and splitting completely in $\Q(\chi^d)/\Q$. Now suppose that $x\in \Cl_{\mathfrak{c},\Sigma}'(\Q(\chi^d))$: then the absolute norm of $x$ is the trivial class in the ray class group of $\Q$ modulo $(c)$ times the archimedean prime of $\Q$. This means that the absolute norm of $\mathfrak{p}$ is generated by a positive integer congruent to $1$ modulo $c$. But since $\mathfrak{p}$ splits completely, the absolute norm of $\mathfrak{p}$ coincides with the prime ideal $p\Z$ below $\mathfrak{p}$. Hence $x$ contains a prime whose underlying rational prime is congruent to $1$ modulo $c$, which means that $x\in  \Cl_{\mathfrak{c},\Sigma}(\Q(\chi^d))$. Thus $\Cl_{\mathfrak{c},\Sigma}'(\Q(\chi^d))\subset \Cl_{\mathfrak{c},\Sigma}(\Q(\chi^d))$ for every divisor $d$ of $c$ as claimed and clearly the reverse inclusion $\Cl_{\mathfrak{c},\Sigma}(\Q(\chi^d))\subset \Cl'_{\mathfrak{c},\Sigma}(\Q(\chi^d))$ also holds. Therefore 
$$\Cl(\Z[C])_\Sigma=\pi_C(\bigoplus_{d\mid c}\Cl_{c,\Sigma}(\Q(\chi^d)))=\pi_C(\bigoplus_{d\mid c}\Cl'_{c,\Sigma}(\Q(\chi^d)))=\Cl'(\Z[C])_\Sigma.$$
\end{proof}


\section{Explicit unit elements}\label{section:stickelberger}
This section is devoted to finding explicit unit elements associated to the classes of $T_\Z$, $S$ and $R$ in $\Cl(\Z[\Delta])$, yielding the proof of their triviality. We stick to the notation introduced in the previous section. 


\subsection{A cyclotomic unit to describe $(T_\Z)$}
In this subsection we will prove the triviality of $T_{\mathbb{Z}}$ in $\Cl(\Z[\Delta])$, which has already been proved by Chase using a well-known result of Swan. We hope that our proof could serve as a guiding path for the cases of $S$ and $R$. As to our knowledge, every proof of the triviality of $(T_\Z)$ uses cyclotomic units. In our approach, we need them to modify the representative homomorphism $v$ of $T_{\mathbb{Z}}$ (constructed in \S\ref{hom-rep-TZ}) by an equivariant function on characters of $\Delta$ with values in $\overline{\Q}^\times$. When we deal with the classes of $S$ and $R$, we will essentially replace cyclotomic units by Jacobi and Gauss sums (compare (\ref{cv}) with (\ref{cs2}) and (\ref{cr2})).

\begin{proposition}\label{swan}
The class of $\Sigma_\Delta(p)=p\Z[\Delta]+\Tr_\Delta\Z[\Delta]$ is represented in
$\Hom_{\Omega_\Q}(R_\Delta,J(\Q(\zeta)))$  by the morphism whose
$\qq$-component at a prime ideal $\qq$ of $\oo$ is
$$\left\{
\begin{array}{ll}
1 & \textrm{if }\qq\nmid e\enspace,\\
\Det(p^{-1}u_t)& \textrm{if }\qq\mid e\enspace,
\end{array}
\right.$$
where $u_t=1+\delta+\cdots+\delta^{p-1}\in\Z[\Delta]$. Further,
if $\qq\cap\Z=q\Z$ with $q\not=p$, then $u_t\in\Z_q[\Delta]^\times$. As a direct 
consequence, we get: 
$$\big(T_{\Z}\big)=\big(\Sigma_\Delta(p)\big)=1\quad \textrm{in }\Cl(\Z[\Delta])\enspace.$$
\end{proposition}
\begin{proof}
Recall the expression of the representative homomorphism $v$ of $(T_\Z)$ given in Lemma \ref{swan-repr1}. In order to show that $v$ belongs to the denominator of the Hom-description, we shall modify it by a global valued equivariant morphism, with values in the cyclotomic field $\Q(\zeta)$. We thus now look at $v$ as a morphism with values in the id\`eles of $\Q(\zeta)$, through the natural embedding $J(\Q)\subseteq J(\Q(\zeta))$ given by $(x_q)_q\mapsto(x_\qq)_\qq$ with $x_\qq=x_q$ if $\qq$ is a prime ideal of $\oo$ above the rational prime $q$. The content of the id\`ele $v(\chi^h)\in J(\Q(\zeta))$ is then the principal ideal $(p^{1-\updelta_{h,e}})$ of $\oo$.  

We consider the morphism $c_v\in\Hom(R_\Delta,\Q(\zeta)^\times)$ defined by 
\begin{equation}\label{cv}
c_v(\chi^e)=1\ ,\quad c_v(\chi^h)=\frac{1-\zeta^h}{1-\zeta^{ph}}\,p
\end{equation}
for $h\in\{1,\ldots,e-1\}$. This morphism clearly commutes with the action of ${\Omega_\Q}$, hence $vc_v^{-1}$ also represents the class of $\Sigma_\Delta(p)$. Note that $\frac{1-\zeta^{ph}}{1-\zeta^{h}}$ is a cyclotomic unit, thus belongs to $\oo^\times$. 
It follows that $vc_v^{-1}$ takes unit idelic values, namely belongs to $\Hom_{\Omega_\Q}(R_\Delta,\UU\big(\Q(\zeta)\big))$. We deduce from \cite[(I.2.19)]{Frohlich-Alg_numb}, in which the $+$ sign disappears since the abelian group $\Delta$ has no symplectic character, that 
\begin{equation}\label{max-order}
vc_v^{-1}\in\Det(\UU(\M))\enspace,
\end{equation}
where $\M$ denote the maximal order in $\Q[\Delta]$. Since $\M_q(=\M\otimes_\Z\Z_q)=\Z_q[\Delta]$ whenever $q\nmid e$ (see \cite[Proposition 27.1]{CR1}) and $\M_\infty=\mathbb{R}[\Delta]$, 
we get that $(vc_v^{-1})_\qq\in\Det(\Z_q[\Delta]^\times)$ if $\qq\nmid e$ and $\qq\cap\Z=q\Z$. We now prove that the same relationship holds when $\qq\mid e$. It will follow from the next result.
\begin{lemma}
Let $\qq$ be a prime ideal of $\oo$ above a rational prime $q$ different from $p$, then
$$(vc_v^{-1})_\qq=\Det(p^{-1}u_t)$$
and $u_t\in\Z_q[\Delta]^\times$. Hence $(vc_v^{-1})_\qq\in\Det(\Z_q[\Delta]^\times)$.
\end{lemma}
\begin{proof}
Let $h\in\{1,\ldots,e\}$, then
$$\Det_{\chi^h}(u_t)=\left\{\begin{array}{ll}
p & \textrm{ if }h=e,\\
1+\zeta^h+\cdots+\zeta^{h(p-1)}=\frac{1-\zeta^{ph}}{1-\zeta^h}
& \textrm{ otherwise,}
\end{array}\right.\enspace.$$
Since $\qq\nmid p$, one has $v(\chi^h)_\qq=1$, hence $p^{-1}\Det_{\chi^h}(u_t)=(vc_v^{-1})_\qq(\chi^h)$. It follows that 
$(vc_v^{-1})_\qq=\Det(p^{-1}u_t)$ as announced. 

Further, we know by (\ref{max-order}) that there exists $w_q\in
\mathcal{M}_q^\times$ such that  
$$(vc_v^{-1})_\qq=\Det(w_q)\enspace.$$ 
It follows that
$\Det(w_q)=\Det(p^{-1}u_t)$, so by Proposition \ref{prop:Det} we must have $w_q=p^{-1}u_t$. 
Since $q\not= p$, 
$$p^{-1}u_t\in\mathcal{M}_q^\times\cap\Z_q[\Delta]
=\Z_q[\Delta]^\times\enspace,$$
(see \cite[\S 25, Exercise 4]{Reiner}); the same holds for $u_t$.
\end{proof}
Combining with the above result, we get that $vc_v^{-1}\in\Det(\UU(\Z[\Delta]))$, namely $vc_v^{-1}$ lies in the denominator of the Hom-Description of $\Cl(\Z[\Delta])$ and thus $(\Sigma_\Delta(p))=1$. Further we can change $vc_v^{-1}$ by multiplying its $\qq$-component, whenever by $\qq\nmid e$, by its inverse (which also belongs to $\Det(\Z_q[\Delta]^\times)$ if $\qq\cap\Z=q\Z$), to get the announced representative of $(\Sigma_\Delta(p))$. This ends the proof of Proposition \ref{swan}.
\end{proof}
Clearly the trivial homomorphism sending any character to $1$ also represents the class of $\Sigma_\Delta(p)$, since this class is trivial. In the representative morphism given in Proposition \ref{swan} (and in Theorem \ref{main}) we have chosen to keep the $\qq$-components for $\qq\mid e$ as they appeared in the proof because of their arithmetic meaning, in order to recall the link between the Galois structure of the Swan module and the cyclotomic units.

The above proposition proves the assertions concerning $(T_\mathbb{Z})$ in Theorem \ref{main}. In the rest of this section we will deal with the assertions concerning $(S)$ and $(R)$. 

\subsection{Gauss and Jacobi sums to describe $(R)$ and $(S)$}

\subsubsection{}
We start introducing Gauss and Jacobi sums associated to the residue fields of the intermediate extensions of $\Q(\zeta)/\Q$. We denote by $\mu_\infty$ the subgroup of roots of unity in $\overline\Q^\times$ and we let $\xi$ denote an element of order $p$ of $\mu_\infty$.

Let ${e'}$ be any divisor of $e$ and recall that $\oo_{e'}=\Z[\zeta_{e'}]$ and $\pp_{e'}=\pp\cap \oo_{e'}$ respectively denote the ring of integers and the prime ideal below $\pp$ in the subfield $\Q(\zeta_{e'})$ of $\Q(\zeta)$. Let $\theta$ denote a multiplicative character of $\oo_{e'}/\pp_{e'}$, namely an homomorphism $(\oo_{e'}/\pp_{e'})^\times\to\mu_\infty$, extended to $\oo_{e'}/\pp_{e'}$ by the convention $\theta(0)=0$. The Gauss sum relative to $\theta$ is defined as: 
$$G(\theta)=\sum_{x\in\oo_{e'}/\pp_{e'}}\theta(x)\xi^{\mathrm{Tr}_{e'}(x)}\enspace,$$
where $\mathrm{Tr}_{e'}:\oo_{e'}/\pp_{e'}\to \Z/p\Z$ denotes the residue field trace homomorphism. We will mostly be concerned with the case where $\theta=\left(\frac{}{\pp_{e'}}\right)^{-1}$ is the inverse of the ${e'}$th power residue symbol, and we set $G_{e'}=G(\left(\frac{}{\pp_{e'}}\right)^{-1})$. Specifically, for any $x\in(\oo_{e'}/\pp_{e'})^\times$, $\left(\frac{x}{\pp_{e'}}\right)$ is the ${e'}$th root of unity defined by the congruence
$$\left(\frac{x}{\pp_{e'}}\right)\equiv x^{\frac{p^{f_{\!{e'}}}\,-\,1}{{e'}}}\quad \pmod{\pp_{e'}}\enspace.$$

Suppose $\theta$ is a multiplicative character of $\oo_{e'}/\pp_{e'}$ that takes values in $\{0\}\cup\mu_{e'}$, then $G(\theta)\in\oo_{e'}[\xi]$. Let $\tau\in\Gal(\Q(\zeta_{e'},\xi)/\Q(\zeta_{e'}))$ and let $\beta\in\F_p^\times$ be such that $\tau(\xi)=\xi^\beta$, then the multiplication-by-$\beta$ map is a bijection of $\oo_{e'}/\pp_{e'}$ onto itself and 
\begin{equation}\label{galois-gauss}
G(\theta)^{\tau}=\sum_{x\in\oo_{e'}/\pp_{e'}}\theta(x)\xi^{\beta\mathrm{Tr}_{e'}(x)}=\sum_{x\in\oo_{e'}/\pp_{e'}}\theta(\beta^{-1}x)\xi^{\mathrm{Tr}_{e'}(x)}=\theta(\beta)^{-1}G(\theta)\enspace.
\end{equation}
It follows that 
\begin{equation}\label{invarianceGdd}
{G(\theta)}^{e'}\ \in\ \oo_{e'}\enspace,
\end{equation}
in particular ${G_{e'}}^{e'}\in\oo_{e'}$.

In the same setting, let $\theta$, $\theta'$ be multiplicative characters of $\oo_{e'}/\pp_{e'}$. The Jacobi sum relative to $\theta$ and $\theta'$ is:
$$J(\theta,\theta')=\sum_{x\in \oo_{e'}/\pp_{e'}}\theta(x)\theta'(1-x)\enspace.$$ 
If $\theta\theta'$ is a non trivial character, then one has (\cite[\S 8, Theorem 1]{Ireland-Rosen}):
$$J(\theta,\theta')=\frac{G(\theta)G(\theta')}{G(\theta\theta')}\enspace.$$
We set $J_{e'}=J(\left(\frac{}{\pp_{e'}}\right)^{-1},\left(\frac{}{\pp_{e'}}\right)^{-1})$. Since $e$ is odd, one has $\left(\frac{}{\pp_{e'}}\right)^2\not=1$, hence 
$$J_{e'}={G_{e'}}^{2-\sigma_{{e'},2}}$$
(where, with a slight abuse of notation, $\sigma_{{e'},2}$ is lifted to $\mathrm{Gal}(\Q(\zeta_{e'},\xi)/\Q)$ in such a way that $\sigma_{{e'},2}(\xi)=\xi$) and, using (\ref{galois-gauss}),
$$J_{e'}\ \in\ \oo_{e'}\enspace.$$

As it is well-known, the factorization of Gauss and Jacobi sums can be written in a nice way using the Stickelberger element. More precisely, we have the following classical result about the ideals of $\oo_{e'}$ generated respectively by ${G_{e'}}^{e'}$ and $J_{e'}$. For details and for a proof see for instance \cite[Theorem 2 in Chapter 14, Proposition 15.3.2 and its proof in Chapter 15]{Ireland-Rosen}.  

\begin{theorem}[Stickelberger]\label{stickel}
One has
$(2-\sigma_{{e'},2})\Theta_{e'}\in\Z[\mathrm{Gal(\Q(\zeta_{e'})/\Q)}]$, and  
$$
({G_{e'}}^{e'})={\pp_{e'}}^{{e'}\Theta_{e'}}\quad,\qquad (J_{e'})={\pp_{e'}}^{(2-\sigma_{{e'},2})\Theta_{e'}}\enspace.
$$
\end{theorem}

\subsubsection{}
In view of Propositions \ref{contentr} and \ref{contents} we deduce, since $\pp^{N_{e,e'}}=\pp_{e'}^{f/f_{e'}}$, that the contents of the representative homomorphisms $s$ and $r$, evaluated at $\chi^d$ with $d\mid e$ and $e=de'$, are principal ideals given by:
\begin{align*}
cont(s(\chi^d))&=\Big({J_{e'}}^{f/f_{e'}}\Big)\enspace,\\
cont(r(\chi^d))&=\Big({G_{e'}}^{ef/f_{e'}}\Big)\enspace.
\end{align*}
We now use the above generators of the content ideals to define elements $c_s$ and $c_r$ in $\mathrm{Hom}_{\Omega_{\Q}}(R_{\Delta},\Q(\zeta)^\times)$. We have to multiply $J_{e'}^{f/f_{e'}}$ and $G_{e'}^{ef/f_{e'}}$ by suitable units in order to ensure that $sc_s^{-1}$ and $rc_r^{-1}$ lie in $\mathrm{Det}(\mathcal{U}(\Z[\Delta]))$. 
Let $c_s$ and $c_r$ denote the only homomorphisms in $\mathrm{Hom}_{\Omega_{\Q}}(R_{\Delta},\Q(\zeta)^\times)$ such that, for any $d\mid e$, 
\begin{align}
c_s(\chi^d)&=-(-J_{e'})^{f/f_{e'}}\ ,\label{cs2}\\
c_r(\chi^d)&=(-1)^{e}(-G_{e'})^{ef/f_{e'}}\label{cr2}\enspace,
\end{align}
where $e=de'$. 

The first step to our goal is easy. Let $\M$ denote the maximal order in $\Q[\Delta]$.
\begin{corollary}\label{ordmax}
The homomorphisms $sc_s^{-1}$ and $rc_r^{-1}$ belong to $\mathrm{Det}(\UU(\M))$.
\end{corollary}
\begin{proof}
It is clear from above that, for any divisor $d$ of $e$, one has the following equalities of ideals: 
$$
\big(c_s(\chi^d)\big)=cont\big(s(\chi^d)\big)\,,\ \big(c_r(\chi^d)\big)=cont\big(r(\chi^d)\big)\enspace.
$$
It follows that $(sc_s^{-1})_{\mathfrak{q}}$ (resp. $(rc_r^{-1})_{\mathfrak{q}}$) takes unit values for every place $\mathfrak{q}$. The result thus follows from \cite[(I.2.19)]{Frohlich-Alg_numb} as in the proof of Proposition \ref{swan}.
\end{proof}
The former result does not depend on the signs that appear in the
definitions of $c_s$ and $c_r$. The importance of the choice of
these signs will become clear in the proof of our next result, which
will occupy the rest of this subsection.
\begin{theorem}\label{desiderio}
The homomorphisms $sc_s^{-1}$ and $rc_r^{-1}$ belong to $\mathrm{Det}(\mathcal{U}(\Z[\Delta]))$. In particular $(S)$ and $(R)$ are trivial in $\Cl(\Z[\Delta])$. 
\end{theorem}
This result, together with Lemmas \ref{us} and \ref{ur}, is a reformulation of Theorem \ref{main}, as far as $(S)$ and $(R)$ are concerned.

As in the proof of Proposition \ref{swan}, one deduces from Corollary \ref{ordmax} that whenever $\qq\nmid e$, the $\qq$-component of $sc_s^{-1}$ and $rc_r^{-1}$ belongs to $\Det(\Z_q[\Delta]^\times)$ where $\qq\cap\Z=q\Z$, so we focus on the case $\qq\mid e$. 
Recall from Corollary \ref{rs-rep} that, when $\qq\mid e$,
\begin{equation}\label{qmide}
(sc_s^{-1})_\qq=c_s^{-1}\ ,\quad
(rc_r^{-1})_\qq=c_r^{-1}\enspace,
\end{equation}
where $c_s^{-1}$ (resp. $c_r^{-1}$) is seen as a
morphism with values in $\Q(\zeta)$, diagonally embedded in
$J_q(\Q(\zeta))=\prod_{\qq\mid q}\Q(\zeta)_\qq$.
\subsubsection{The proof for $S$.}\label{proofS}

For any $i=0,\,\ldots,\,e-1$, set
$$A_i=\left\{x\in \oo/\pp\,:\, \left(\frac{x}{\pp}\right)^{-1}\left(\frac{1-x}{\pp}\right)^{-1}=\zeta^i\right\}$$
and $n_i=\#A_i$, then
$$J_e=\sum_{i=0}^{e-1}n_i\zeta^i.$$
The main result is the following.
\begin{lemma}\label{us}
Let $u_s=\sum\limits_{i=0}^{e-1}n_i\delta^i\in \Z[\Delta]$. For any
prime ideal $\qq$ of $\oo$ above a rational prime $q$ such that $q\mid e$, $u_s\in\Z_q[\Delta]^\times$, and 
$$(sc_s^{-1})_\qq=\Det(u_s^{-1})\ \in\ \Det(\Z_q[\Delta]^\times)\enspace.$$
\end{lemma}
\begin{proof}
We first show that, for any $d\mid e$ 
\begin{equation}\label{hd}
\Det_{\chi^d}(u_s)=-(-J_{e'})^{f/f_{e'}}\enspace,
\end{equation}
where $e=de'$ as usual. Thanks to the Davenport-Hasse theorem \cite[Theorem 1 and Exercice 18 in Chapter 11]{Ireland-Rosen} 
$$(-1)^{f/f_{e'}-1}{J_{e'}}^{f/f_{e'}}=\sum_{x\in \oo/\pp}\left(\frac{\overline{N}_{e,e'}(x)}{\pp_{e'}}\right)^{-1}\left(\frac{\overline{N}_{e,e'}(1-x)}{\pp_{e'}}\right)^{-1}\enspace,$$
where $\overline{N}_{e,e'}:\oo/\pp\to
\oo_{e'}/\pp_{e'}$ is the residual relative norm map. For any $x\in \oo/\pp$,
$$\left(\frac{x}{\pp}\right)^d\equiv \left(x^{\frac{p^f-\,1}{e}}\right)^d\equiv  \left(x^{\sum_{t=0}^{f/f_{\!e'}-1}p^{tf_{\!e'}}}\right)^{\frac{p^{f_{\!e'}}\,-\,1}{e'}}\equiv \overline{N}_{e,e'}(x)^{\frac{p^{f_{\!e'}}\,-\,1}{e'}} \pmod{\pp}\enspace.$$
Thus
\begin{equation}\label{lifted-character}
\left(\frac{x}{\pp}\right)^d=\left(\frac{\overline{N}_{e,e'}(x)}{\pp_{e'}}\right)
\end{equation}
and therefore
$$-(-J_{e'})^{f/f_{e'}}=\sum_{x\in \oo/\pp}\left(\frac{x}{\pp}\right)^{-d}\left(\frac{1-x}{\pp}\right)^{-d}=\sum_{i=0}^{e-1} \sum_{x\in A_i}\zeta^{id}=\sum_{i=0}^{e-1} n_i\zeta^{id}=\Det_{\chi^d}(u_s)\enspace.$$
Assume that $\qq\mid e$. For any divisor $d$ of $e$ (once again $e=de'$), we deduce from (\ref{qmide}) and (\ref{hd}) that 
$$(sc_s^{-1})_\qq(\chi^d)=-(-J_{e'})^{-f/f_{e'}}=\Det_{\chi^d}(u_s^{-1})\enspace.$$
The assertion of Corollary \ref{ordmax} implies that there exists $w_q\in \mathcal{M}_q^\times$ such that  
$$(sc_s^{-1})_\qq=\Det(w_q^{-1})\enspace.$$ 
It follows that $\Det(w_q)=\Det(u_s)$, so by Proposition \ref{prop:Det} we must have $w_q=u_s\in\mathcal{M}_\qq^\times \cap\Z_q[\Delta]=\Z_q[\Delta]^\times$ (using again \cite[Exercise 4, Section 25]{Reiner}).
\end{proof}
The rest of the proof of Theorem \ref{desiderio} and Theorem \ref{main}, concerning $S$, is similar to that of Proposition \ref{swan} above.
\subsubsection{The proof for $R$.}
The proof for $R$ goes the same way as that for $S$, with the
additional difficulty that the expression of ${G_e}^e$ is not as
explicit as that of $J_e$. Nevertheless, we can make it explicit using
the multinomial formula. Let $\theta$ denote a multiplicative
character of $\oo/\pp$ and let $\CC_e$ denote the set of
$p^f$-uples of integers $(k_1,\cdots,k_{p^f})$ such that
$\sum_{h=1}^{p^f}k_h=e$. We number arbitrarily the elements of 
$\oo/\pp$ as $x_h$, $1\le h\le {p^f}$, then
$${G(\theta)}^e=\sum_{\CC_e}\frac{e!}{\prod_{h=1}^{p^f}k_h!}\,\theta\bigg(\prod_{h=1}^{p^f}x_h^{k_h}\bigg)\xi^{\Tr(\sum_{h=1}^{p^f}k_hx_h)}\enspace,$$
where $\Tr$ is the trace from $\oo/\pp$ to $\F_p$. For $j\in\F_p$ let
$\CC_{e,j}$ denote the set of $p^f$-uples $(k_1,\cdots,k_{p^f})$ of
$\CC_e$ such that $\Tr(\sum_{h=1}^{p^f}k_hx_h)=j$ and set
$$g_j(\theta)=\sum_{\CC_{e,j}}\frac{e!}{\prod_{h=1}^{p^f}k_h!}\,\theta\bigg(\prod_{h=1}^{p^f}x_h^{k_h}\bigg)\enspace,$$
then
$${G(\theta)}^e=\sum_{j\in\F_p}g_j(\theta)\xi^j=\sum_{j\in\F_p\setminus\{1\}}\big(g_j(\theta)-g_{1}(\theta)\big)\xi^j\enspace.$$ 

We now assume that $\theta$ takes values in $\{0\}\cup\mu_e$, so
${G(\theta)}^e\in\oo$ by (\ref{invarianceGdd}) and hence
$${G(\theta)}^e=g_0(\theta)-g_{1}(\theta)\enspace.$$
Note incidentally that the $g_j(\theta)$ for $j\not=0$ are all
equal. 
For $j\in\F_p$ and $i\in\{0,\ldots,e-1\}$, we let
$\CC_{e,j,i}(\theta)$ denote the set of $p^f$-uples
$(k_1,\cdots,k_{p^f})$ of $\CC_{e,j}$ such that 
$\theta\big(\prod_{h=1}^{p^f}x_h^{k_h}\big)=\zeta^i$, and we
  set 
$$m_i(\theta)=\sum_{\CC_{e,0,i}(\theta)}\frac{e!}{\prod_{h=1}^{p^f}k_h!}-\sum_{\CC_{e,1,i}(\theta)}\frac{e!}{\prod_{h=1}^{p^f}k_h!}\enspace,$$
so that $m_i(\theta)\in\Z$ for each $i$ and 
$${G(\theta)}^e=\sum_{i=0}^{e-1}m_i(\theta)\zeta^i\enspace.$$
The main interest of this construction lies in the following result. 
\begin{proposition}\label{theta-d}
Let $\theta$ denote a multiplicative character of $\oo/\pp$ taking
values in $\{0\}\cup\mu_e$. Let $d\mid e$, then
$$G\big(\theta^d\big)^e=\sum_{i=0}^{e-1}m_i(\theta)\zeta^{id}\enspace.$$
\end{proposition}
\begin{proof}
As usual we set $e=de'$. Let $i\in\{0,\ldots,e-1\}$. If
$d\nmid i$, then clearly $\CC_{e,j,i}(\theta^d)$ is empty for any
$j\in\F_p$, thus $m_i(\theta^d)=0$. Further, if
$i\in\{0,\ldots,e'-1\}$ and $j\in\F_p$, one easily checks the equality
of sets:   
$$\CC_{e,j,id}(\theta^d)=\bigcup_{k=0}^{d-1}\CC_{e,j,i+ke'}(\theta)\enspace.$$
It follows that 
$$m_{id}(\theta^d)=\sum_{k=0}^{d-1}m_{i+ke'}(\theta)\enspace,$$
hence
$$G\big(\theta^d\big)^e=\sum_{i=0}^{e'-1}m_{id}(\theta^d)\zeta^{id}=\sum_{i=0}^{e'-1}\sum_{k=0}^{d-1}m_{i+ke'}(\theta)\zeta^{id}=\sum_{i=0}^{e-1}m_i(\theta)\zeta^{id}\enspace.$$
\end{proof}
We can now show a statement which is similar to those of the previous
subsections. Recall that $G_e=G(\big(\frac{}{\pp}\big)^{-1})$. For each
$i\in\{0,\ldots,e-1\}$, we set $m_i=m_i(\big(\frac{}{\pp}\big)^{-1})$.  
\begin{lemma}\label{ur}
Let $u_r=\sum\limits_{i=0}^{e-1}m_i\delta^i\in \Z[\Delta]$. For any
prime ideal $\qq$ of $\oo$ above a rational prime $q$ such that $q\mid e$, $u_r\in\Z_q[\Delta]^\times$ and
$$(rc_r^{-1})_{\qq}=\Det(u_r^{-1})\ \in\ \Det(\Z_q[\Delta]^\times)\enspace.$$
\end{lemma}
\begin{proof}
Let $d\mid e$ and set $e=de'$.
The Davenport-Hasse theorem for the lifted Gauss sum \cite[Theorem 1 in Chapter 11]{Ireland-Rosen} states that
$$-G(\bigg(\frac{\overline{N}_{e,e'}(\ )}{\pp_{e'}}\bigg)^{-1})=(-G_{e'})^{f/f_{e'}}\enspace.$$
Raising to the power $e$, using (\ref{lifted-character}) and Proposition \ref{theta-d}, we get
$$(-1)^{e}(-G_{e'})^{ef/f_{e'}}=G\big(\Big(\frac{}{\pp}\Big)^{-d}\big)^{e}=\sum_{i=0}^{e-1}m_i\zeta^{id}=\Det_{\chi^d}(u_r)\enspace,$$
namely
$$\Det_{\chi^d}(u_r^{-1})=c_r^{-1}(\chi^d)=(rc_r^{-1})_\qq(\chi^d)$$
whenever $\qq\mid e$. The end of the proof is as in \S\ref{proofS}.
\end{proof}

Using Lemma \ref{ur} we easily conclude the proof of Theorem \ref{desiderio} and Theorem \ref{main} (as far as $R$ is concerned), as in the proof of Proposition \ref{swan}.

%
%

\section{The class of the square root of the inverse different in locally abelian extensions}\label{analytic}
In this section we put ourselves in the global setting described in the Introduction. In particular $N/E$ is a tame $G$-Galois extension of number fields. When $N/E$ is locally abelian and $\C_{N/E}$ is a square, Theorem \ref{maincor}, together with Taylor's theorem, implies that $(\A_{N/E})\in\Cl(\Z[G])$ is determined by the Artin root numbers of symplectic representations of $G$ (see Corollary \ref{A=tW}). Using this, one immediately deduces that $(\A_{N/E})=1$ if $N/E$ has odd degree (in fact this is true even without assuming that $N/E$ is locally abelian by a result of Erez, \cite[Theorem 3]{Erez2}) or is abelian. We will also prove that $(\A_{N/E})=1$ if no real place of $E$ becomes complex in $N/E$. We will then show that this hypothesis is necessary and get Theorem \ref{surprise} as a consequence. In fact, in the case where $G$ is the binary tetrahedral group, we will exhibit a totally complex tame $G$-Galois extension $N/\Q$ such that $\C_{N/\Q}$ is a square and $(\A_{N/\Q})\ne 1$. To explain how we found this example, we shall first verify that $(\A_{N/\Q})= 1$ if $N/\Q$ is a tame locally abelian $G$-Galois extension such that $\C_{N/\Q}$ is a square and $G$ is a group of order at most $24$ which is not isomorphic to the binary tetrahedral group. 

\subsection{Root numbers of extensions unramified at infinity}
In order to describe explicitly the relation between $(\A_{N/E})$ and the Artin root numbers, we need to recall some properties of these numbers and the definition of the Fr\"ohlich--Cassou-Noguès class $t_GW_{N/E}\in \Cl(\Z[G])$.  

\subsubsection{}\label{rnp}
We will omit the definition of the root numbers and just recall some of their standard properties (see \cite{Martinet}). Let $\Gamma$ be a finite group and let $K/k$ be a $\Gamma$-extension of local or global fields of characteristic $0$. Let $\chi:\Gamma\to \mathbb{C}$ be a complex virtual character and let $W(K/k, \chi)\in \mathbb{C}$ denote the root number of $\chi$. Then:
\begin{itemize}
\item if $\chi_1, \chi_2:\Gamma\to\mathbb{C}$ are virtual characters, then $W(K/k, \chi_1+\chi_2)=W(K/k, \chi_1)W(K/k, \chi_2)$; 
\item if $\Gamma'$ is a subgroup of $\Gamma$, corresponding to the subextension $K/k'$, and $\phi:\Gamma'\to \mathbb{C}$ is a virtual character, then $W(K/k, \mathrm{Ind}_{\Gamma'}^{\Gamma}\phi)=W(K/k', \phi)$;
\item if $\overline{\Gamma}$ is a quotient of $\Gamma$, corresponding to the subextension $K'/k$, and $\phi:\overline{\Gamma}\to \mathbb{C}$ is a virtual character, then $W(K/k, \mathrm{Inf}_{\overline{\Gamma}}^{\Gamma}\phi)=W(K'/k, \phi)$.
\end{itemize}

If $K/k$ is an extension of local fields, we have
\begin{equation}\label{conjdet}
W(K/k,\chi)W(K/k,\overline{\chi})=\mathrm{det}_\chi(-1).
\end{equation} 
Here $\overline{\chi}$ denotes the conjugate character of $\chi$ and $\mathrm{det}_{\chi}:k^\times\to \mathbb{C}^\times$ is the map obtained by composing determinant of $\chi$ with the local reciprocity map $k^\times\to \Gamma$. 

If $K/k$ is an extension of number fields, there is a decomposition 
\begin{equation}\label{locdec}
W(K/k,\chi)=\prod_{P}W(K_{\PP}/k_P,\chi_{_{D_\PP}}).
\end{equation}
Here the product is taken over all places $P$ of $k$; for a such $P$, $\PP$ is any fixed place of $K$ above $P$ and $\chi_{_{D_\PP}}$ denotes the restriction of $\chi$ to the decomposition group $D_{\PP}$ of $\PP$ (which we view as a character of $\mathrm{Gal}(K_{\PP}/k_P)$).

\subsubsection{}\label{sympltG}
We come back to the global setting where $N/E$ is a tame $G$-Galois extension of number fields. Let $S_G$ denote the group of symplectic characters, namely the free abelian group generated by the characters of the irreducible symplectic representations of $G$. Recall  (see \cite[Section 13.2]{Serre-reprlin}, \cite[III]{Martinet}) that an irreducible representation $\rho:G\to GL(V)$ on a complex vector space $V$ is called symplectic if $V$ admits a nontrivial $G$-invariant alternating bilinear form. Symplectic characters are real-valued and their degree is even. Moreover an irreducible character $\chi:G\to \mathbb{C}$ is symplectic if and only if its Frobenius-Schur indicator is $-1$:
$$\frac{1}{\#G}\sum_{g\in G}\chi(g^2)=-1.$$

Let $W_{N/E}\in \mathrm{Hom}(S_G,\mathbb{C}^\times)$ be defined by $W_{N/E}(\theta)=W(N/E,\theta)$ for $\theta\in S_G$. As shown by Fr\"ohlich, we actually have 
$W_{N/E}\in \mathrm{Hom}_{\Omega_\mathbb{Q}}(S_G,\{\pm 1\})$ (\cite[I, Proposition 6.2]{Frohlich-Alg_numb}).

Let $L'/\mathbb{Q}$ be a Galois extension containing all the values of the characters of $G$. We now recall the definition of the map
$$t_G: \mathrm{Hom}_{\Omega_\mathbb{Q}}(S_G,\{\pm 1\})\to \Cl(\Z[G])$$ 
which is due to Ph. Cassou-Nogu\`es. One first defines a map 
$$t'_G:\mathrm{Hom}_{\Omega_\mathbb{Q}}(S_G,\{\pm 1\})\to\mathrm{Hom}_{\Omega_\mathbb{Q}}(R_G,J(L'))$$ 
as follows. Let $f\in \mathrm{Hom}_{\Omega_\mathbb{Q}}(S_G,\{\pm 1\})$ and let $\theta$ be an irreducible character of $G$. If $\mathfrak{l}$ is a place of $L'$, the id\`ele $t'_G(f)(\theta)\in J(L')$ has $\mathfrak{l}$-component
$$\tilde{f}(\theta)_{\mathfrak{l}}=\left\{\begin{array}{ll}f(\theta)&\textrm{if $\mathfrak{l}$ is finite and $\theta$ is symplectic}\\1&\textrm{otherwise.}\end{array}\right.$$
The map $t_G$ is then obtained by composing $t'_G$ with the projection 
$$\Hom_{\Omega_\Q}(R_G,J(L'))\to \Cl(\Z[G])$$
induced by the Hom-description (see Section \ref{hdcg}).

The class $t_GW_{N/E}$ appears in the following celebrated result of M. Taylor (\cite[Theorem 1]{tay2}). 

\begin{theorem}[M. Taylor]\label{pearl}
Let $N/E$ be a tame $G$-Galois extension of number fields. Then 
$$(\OO_N)=t_GW_{N/E}\quad \textrm{in $\Cl(\Z[G])$}.$$
\end{theorem}

Combining the above theorem with Theorem \ref{maincor} we get the following result.

\begin{corollary}\label{A=tW}
Let $N/E$ be a tame locally abelian $G$-Galois extension of number fields such that $\C_{N/E}$ is a square. Then
$$(\A_{N/E})=t_GW_{N/E}\quad \textrm{in $\Cl(\Z[G])$}.$$
In particular, if $G$ has no symplectic representations (for instance if $G$ is abelian or has odd order), then $(\A_{N/E})$ is trivial.
\end{corollary}

\subsubsection{}\label{quaternion}
We now want to prove that $(\A_{N/E})$ is also trivial if $N/E$ is a tame locally abelian $G$-Galois extension which is unramified at infinite places, namely if real places of $E$ do not become complex in $N$. Under this assumption we will show that $W_{N/E}$ is in fact trivial and for this we need some results on symplectic characters, which we now recall. 

We dispose of a useful induction theorem for symplectic characters which is due to Martinet (see \cite[III, Theorem 5.1]{Martinet}). Before giving its statement, recall that an irreducible complex character of $G$ is said to be quaternionic if it has degree $2$ and is lifted from a symplectic character of a quotient of $G$ isomorphic to the generalized quaternion group $H_{4n}$ for some $n\geq 2$. For every natural number $n\geq 2$, the quaternion group $H_{4n}$ of order $4n$ is given by the following presentation 
\begin{equation}\label{h4npres}
H_{4n}=\langle \sigma,\tau\,|\, \sigma^n=\tau^2, \tau^4=1, \tau^{-1}\sigma\tau=\sigma^{-1}\rangle.\end{equation}

\begin{theorem}[Martinet]\label{indsym}
A symplectic character of $G$ can be written as a $\Z$-linear combination of cha\-racters of the form $\mathrm{Ind}_H^G\theta$ for some subgroups $H$ of $G$ where either 
\begin{itemize}
\item$\theta=\psi+\overline{\psi}$ with $\psi:H\to \mathbb{C}$ an irreducible character of degree $1$ or
\item $\theta$ is a quaternionic character of $H$.
\end{itemize}
\end{theorem}
 
In order to use the above theorem, we will need some facts about generalized quaternions $H_{4n}$ and their complex characters. Note first that $\langle\sigma\rangle$ is normal in $H_{4n}$ since it has index $2$. One easily sees that the maximal abelian quotient of $H_{4n}$ is $H_{4n}/\langle\sigma^{2}\rangle$, which has order $4$. Thus $H_{4n}$ has precisely four irreducible characters of degree $1$. As for the remaining irreducible characters of $H_{4n}$, let $\varphi:\langle\sigma\rangle\to\mathbb{C}^\times$ be an injective homomorphism. Thus $\varphi$ can be also thought as an irreducible character of $\langle\sigma\rangle$ of degree $1$. Using Mackey's criterion (\cite[Proposition 23]{Serre-reprlin}), it is easy to show that, if $1\leq k\leq 2n-1$ and $k\ne n$, then  $\mathrm{Ind}_{\langle\sigma\rangle}^{H_{4n}}\varphi^k$ is an irreducible character of $G$ of degree $2$. Computing explicitly the values $\mathrm{Ind}_{\langle\sigma\rangle}^{H_{4n}}\varphi^k$ (via \cite[Th\'eor\`eme 12]{Serre-reprlin}), one finds that $\mathrm{Ind}_{\langle\sigma\rangle}^{H_{4n}}\varphi^k$ is always real-valued and
$$\mathrm{Ind}_{\langle\sigma\rangle}^{H_{4n}}\varphi^k=\mathrm{Ind}_{\langle\sigma\rangle}^{H_{4n}}\varphi^h$$  
if and only if $\varphi^h=\overline{\varphi}^k$, \textit{i.e.} $h=2n-k$. Hence we have $n-1$ distinct irreducible characters of degree $2$ and a standard counting argument (see \cite[Corollary 2 to Proposition 5]{Serre-reprlin}) shows that these, together with the four degree-one characters mentioned above, give the set of all irreducible complex characters of $H_{4n}$. 

\begin{remark}\label{H8char}
When $n=2$, we recover the classical quaternion group $H_8$. In particular $H_8$ has four one-dimensional characters, which we denote by $\psi_i$, $i=1,2,3,4$ (with $\psi_1$ denoting the trivial character) and one two-dimensional irreducible character, which we denote by $\phi$. It is well-known that $\phi$ is symplectic (this can be shown by computing its Frobenius-Schur indicator, using the explicit values of $\phi$ given for instance in \cite[Exercice 3, Section 12.2]{Serre-reprlin}).
\end{remark}


In the following lemma, which will be useful in the proof of the next proposition, we describe the abelian subgroups of $H_{4n}$ and the restriction to such subgroups of an irreducible character of $H_{4n}$ of degree $2$.  

\begin{lemma}\label{quatrescyclic}
Let $\theta:H_{4n}\to \mathbb{C}$ be an irreducible character of degree $2$. Let $H'$ be an abelian subgroup of $H_{4n}$. Then either
\begin{itemize}
\item $H'\subseteq \langle\sigma\rangle$ and there exists a character $\rho:H'\to \mathbb{C}$ of degree $1$ such that $\mathrm{Res}^{H_{4n}}_{H'}\theta=\rho+\overline{\rho}$ or
\item  $H'$ is cyclic of order $4$, contains $\langle\tau^2\rangle$ and $\mathrm{Res}^{H_{4n}}_{H'}\theta=\mathrm{Ind}_{\langle\tau^2\rangle}^{H'}\rho$ where $\rho:\langle\tau^2\rangle\to \mathbb{C}$ is a character of degree $1$.
\end{itemize}
\end{lemma}
\begin{proof}
By the above discussion, we know that there exists a character $\chi:\langle\sigma\rangle\to \mathbb{C}$ of degree $1$ such that $\theta=\mathrm{Ind}_{\langle\sigma\rangle}^{H_{4n}}\chi$. Then we have (see \cite[Proposition 22]{Serre-reprlin})
$$ \mathrm{Res}^{H_{4n}}_{H'}\theta =\mathrm{Res}^{H_{4n}}_{H'}\mathrm{Ind}_{\langle\sigma\rangle}^{H_{4n}}\chi=\sum_{s\in \mathscr{S}} \mathrm{Ind}_{H'\cap\langle\sigma\rangle}^{H'}\chi_s$$
where $\mathscr{S}$ is a set of representatives of $H'\backslash H_{4n}/\langle\sigma\rangle$ and, for $s\in \mathscr{S}$, $\chi_s:H'\cap\langle\sigma\rangle\to \mathbb{C}$ is the character defined by $\chi_s(x)=\chi(s^{-1}xs)$ for every $x\in H'\cap\langle\sigma\rangle$. 

We now distinguish two cases. First suppose that $H'\subset \langle\sigma\rangle$, \textit{i.e.} $H'\cap\langle\sigma\rangle=H'$. Then we can take $\mathscr{S}=\{\mathrm{id},\tau\}$ and $\chi_{\mathrm{id}}=\mathrm{Res}^{\langle\sigma\rangle}_{H'}\chi$, $\chi_{\tau}=\mathrm{Res}^{\langle\sigma\rangle}_{H'}\overline{\chi}$ (since $\chi(\tau^{-1}x\tau)=\chi(x^{-1})=\overline{\chi(x)}$ for $x\in H'$). Thus we can take $\rho=\mathrm{Res}^{\langle\sigma\rangle}_{H'}\chi$ to get the result. 

Now suppose that $H'\not\subset \langle\sigma\rangle$. We first show that $H'$ contains $\langle\tau^2\rangle$ and is cyclic of order $4$. Let $x\in H'$ and $x\not \in  \langle\sigma\rangle$. One easily sees that $x\sigma x^{-1}=\sigma^{-1}$, in particular $x$ acts as $-1$ on the cyclic group $H'\cap\langle\sigma\rangle$. Since $H'$ is abelian we also have $xyx^{-1}=y$ for every $y\in H'\cap\langle\sigma\rangle$. This shows that $H'\cap\langle\sigma\rangle$ is a subgroup of exponent $2$ of $\langle\sigma\rangle$. Therefore we have $H'\cap\langle\sigma\rangle=\langle\tau^2\rangle$. Note that $H'\cap\langle\sigma\rangle$ has index $2$ in $H'$ because 
$$4n=\#(H'\cdot \langle\sigma\rangle)=\frac{\#\langle\sigma\rangle \#H'}{\#(H'\cap\langle\sigma\rangle)}=\frac{2n\#H'}{\#(H'\cap\langle\sigma\rangle)}.$$
Hence $H'$ has order $4$ and is cyclic since $\tau^2$ is the only element of order $2$ in $H_{4n}$ (as it follows by an easy calculation using the presentation of $H_{4n}$ given in (\ref{h4npres})). We now come to the character $\mathrm{Res}^{H_{4n}}_{H'}\theta$. Since $H'\not\subset \langle\sigma\rangle$, we can choose the set of representatives $\mathscr{S}$ to be $\{\mathrm{id}\}$, obtaining 
$$ \mathrm{Res}^{H_{4n}}_{H'}\theta= \mathrm{Ind}_{\langle\tau^2\rangle}^{H'}\chi_{\mathrm{id}}.$$
Hence taking $\rho=\chi_{\mathrm{id}}=\mathrm{Res}_{\langle\tau^2\rangle}^{\langle\sigma\rangle}\chi$ we get the result.
\end{proof}

Recall from the Introduction that in a tame Galois extension the inverse different is a square if and only if the inertia subgroups all have odd order. 

\begin{proposition}\label{infnonram}
Let $N/E$ be a tame locally abelian $G$-Galois extension of number fields whose inverse different is a square. Suppose moreover that no archi\-me\-dean place of $E$ ramify in $N$ (\textit{i.e.} real places stay real). Then $W_{N/E}=1$.   
\end{proposition}
\begin{proof}
We have to prove that $W(N/E,\chi)=1$ for every symplectic character $\chi$ of $G$. Thanks to Theorem \ref{indsym}, $\chi$ can be written as
$$\theta=\sum_{H<G}n_H\mathrm{Ind}_H^G\theta_H$$
where $n_H\in\Z$ and $\theta_H:H\to \mathbb{C}$ is either a quaternionic character of $H$ or can be written as $\theta_H=\psi+\overline{\psi}$ for some irreducible character $\psi:H\to \mathbb{C}$ of degree $1$.
Thanks to the properties of the Artin root number recalled in \S\ref{rnp}, we have
$$W(N/E,\chi)=\prod_{H<G}W(N/N^H,\theta_H)^{n_H}.$$
Of course, to prove that $W(N/E,\chi)=1$, it will be sufficient to prove that $W(N/N^H,\theta_H)=1$ for every subgroup $H<G$. Observe that, for every subgroup $H<G$, $N/N^H$ is a tame locally abelian extension whose inverse different is a square (since its inertia subgroups have odd order) and no archimedean place of $N^H$ ramify in $N$ (in other words $N/N^H$ satisfy the hypotheses of the proposition). Therefore, replacing if necessary $G$ by one of its subgroup $H$ and $N/E$ by $N/N^H$, it suffices to show that $W(N/E,\theta)=1$ where $\theta:G\to \mathbb{C}$ is either a quaternionic character of $G$ or $\theta=\psi+\overline{\psi}$ for some irreducible character $\psi:G\to \mathbb{C}$ of degree $1$. 

Suppose first that $\theta$ is quaternionic. In particular,  for some $n\geq 2$, there exists a surjection $G\to H_{4n}$ and a symplectic character $\theta':H_{4n}\to \mathbb{C}$ such that $\theta=\mathrm{Inf}_{H_{4n}}^G\theta'$. Then, if $N/N'$ is the subextension corresponding to the kernel of $G\to H_{4n}$ (thus $H_{4n}\cong \mathrm{Gal}(N'/E)$), we have $W(N/E,\theta)=W(N'/E,\theta')$. So we are reduced to show that $W(N'/E,\theta')=1$ and to do this we can assume that $\theta'$ is irreducible. In particular $\theta'$ has degree $2$ (see the list of irreducible characters of $H_{4n}$ given in \S\ref{quaternion}). By (\ref{locdec}), it is sufficient to show that, for every place $P$ of $E$, we have $W(N'_{\PP}/E_P,\theta'_{D_\PP})=1$, where $\theta'_{D_\PP}$ denotes the restriction of $\theta'$ to the decomposition group $D_{\PP}$ of a fixed place $\PP$ of $N'$ above $P$. Note that $N'/E$ is locally abelian since $N/E$ is. In particular $D_{\PP}$ is abelian and by Lemma \ref{quatrescyclic} we have either 
\begin{itemize}
\item[(a)] $\theta'_{D_\PP}=\rho+\overline{\rho}$ for some character $\rho:D_{\PP}\to \mathbb{C}$ of degree $1$ or 
\item[(b)] $D_{\PP}$ is cyclic of order $4$, contains $\langle\tau^2\rangle$ and $\theta'_{D_\PP}=\mathrm{Ind}_{\langle\tau^2\rangle}^{D_{\PP}}\rho$ for some character $\rho:\langle\tau^2\rangle\to\mathbb{C}$ of degree $1$. 
\end{itemize}

In case (a) we have 
$$W(N'_{\PP}/E_P,\theta'_{D_\PP})=W(N'_{\PP}/E_P,\rho)W(N'_{\PP}/E_P,\overline{\rho})=\mathrm{det}_\rho(-1).$$
Let $r_{_P}:E_P^\times\to D_{\PP}$ denote the local reciprocity map so that $\rho\circ r_P=\mathrm{det}_\rho$ (since $\rho$ has degree $1$). If $P$ is archimedean, then $N'_\PP/E_P$ is trivial by hypothesis and in particular $r_{_P}(-1)=1$ and $\mathrm{det}_\rho(-1)=1$. If instead $P$ is a finite place, then, by class field theory, $r_{_P}(-1)$ belongs to the inertia subgroup $I_{\PP}$ of $\PP$ in $N'/E$, since $-1$ is a unit of $E_P$. In particular, $r_P(-1)^{e_P}=1$ and $e_P=\#I_\PP$ is odd ($N'/E$ has inertia subgroups of odd order since the same holds for $N/E$). But we also have $r_{_P}(-1)^2=1$ and therefore $r_{_P}(-1)=1$, which implies that $\mathrm{det}_\rho(-1)=1$ also in this case.  

In case (b), if $\rho$ is trivial, then $W(N'_{\PP}/E_P,\theta'_{_{D_\PP}})=W(N'_{\PP}/E'_{P'},\rho)=1$ (here $N'_{\PP}/E'_{P'}$ is the subextension corresponding to $\langle\tau^2\rangle\subset D_{\PP}$), since the local root number of the trivial character is trivial. If instead $\rho$ is nontrivial (\textit{i.e.} $\rho$ is the sign character of $\langle\tau^2\rangle$), then one easily sees that $\theta'_{_{D_\PP}}=\mathrm{Ind}_{\langle\tau^2\rangle}^{D_{\PP}}\rho=\nu+\overline\nu$, where $\nu:D_{\PP}\to \mathbb{C}^\times$ is a character of order $4$. Thus we conclude as in case (a).

The case where $\theta=\psi+\overline{\psi}$ for some irreducible character $\psi:H\to \mathbb{C}^\times$ is also similar to the above case (a). 
\end{proof}

\begin{remark}
Let $\PP$ be a non-real archimedean place of $N$ and suppose that $\PP$ lies above a real place $P$ of $E$. Then $E_P=\R$ and $N_\PP=\mathbb{C}$ and the reciprocity map $r_{_P}:\R^\times\to \mathrm{Gal}(\mathbb{C}/\R)$ is nontrivial on $-1$ (in fact $\mathrm{Ker}(r_{_P})=\{x\in\R^\times,x>0\}$). For this reason the arguments of the proof of the above proposition do not apply in the case where real places of $E$ are allowed to become complex in $N$. In fact, as we shall see, the hypothesis on real places of Proposition \ref{infnonram} is necessary. 
\end{remark}

Combining Proposition \ref{infnonram} with Corollary \ref{A=tW} and Theorem \ref{maincor}, we get the following result.
 
\begin{theorem}
Let $N/E$ be a tame locally abelian $G$-Galois extension of number fields whose inverse different is a square. Suppose moreover that no archi\-me\-dean place of $E$ ramifies in $N$. Then 
$$(\A_{N/E})=(\OO_N)=1$$
in $\Cl(\Z[G])$.
\end{theorem}


\subsection{An inverse different whose square root has nontrivial class}\label{explicitexample}
In this section we want to find a group $G$ and a tame locally abelian $G$-Galois extension $N/\mathbb{Q}$ whose inverse different is a square and $(\A_{N/\Q})\ne 1$ in $\Cl(\Z[G])$. The first step is to find a good candidate for the group $G$, which we would also like to be of smallest possible order. For this reasons, we deduce some conditions $G$ has to satisfy in order for an extension with the above properties to exist. 

\begin{lemma}\label{properties}
Suppose that $N/\Q$ is a tame locally abelian $G$-Galois extension whose inverse different is a square and $(\A_{N/\Q})\ne 1$ in $\Cl(\Z[G])$. Then
\begin{enumerate}
\item[(i)] $G$ is generated by elements of odd order;
\item[(ii)] $G$ has an irreducible symplectic representation.
\end{enumerate}
\end{lemma}
\begin{proof}
Since $\mathbb{Q}$ has no nontrivial extension unramified at every finite prime, $G$ is generated by the inertia subgroups of finite primes. As recalled before Proposition \ref{infnonram}, these subgroups have odd order in our situation. Furthermore, if $G$ has no symplectic representation, then $(\A_{N/\Q})=1$ by Corollary \ref{A=tW}. 
\end{proof} 

\subsubsection{}
We now show that there is no group of order smaller than $24$ satisfying properties (i) and (ii) of Lemma \ref{properties}.

\begin{lemma}\label{2groups}
Let $H$ be a group. Then the following are equivalent:
\begin{itemize}
\item $H$ has no proper normal subgroup of index a power of $2$;
\item $H$ is generated by elements of odd order.
\end{itemize}
\end{lemma}
\begin{proof}
Let $H'$ be the subgroup generated by the elements of odd order of $H$. Then $H'$ is normal since conjugation preserves the order of elements. We claim that $H/H'$ has order a power of $2$: let $x\in H/H'$ be a class represented by $h\in H$. Write the order of $h$ as $2^ab$ with $a,b\in\mathbb{N}$ and $b$ odd. Then $x$ has order dividing $2^a$ since $h^{2^a}$ has odd order, hence it belongs to $H'$. This shows our claim. Moreover any normal subgroup $H''$ of $H$ of index a power of $2$ contains $H'$, since any element of odd order of $H$ has trivial image in $H/H''$. From this we easily deduce the equivalence of the two assertions in the statement.
%
\end{proof}

\begin{lemma}\label{order20}
Let $H$ be a group of order $20$. Then $H$ is not generated by elements of odd order. 
\end{lemma}
\begin{proof}
Observe that, by Sylow's theorem, $H$ acts transitively by conjugation on the set of Sylow $5$-subgroups. Let $a$ denote the cardinality of this set. Then $a$ equals the index of the normalizer of any Sylow $5$-subgroup in $H$ and thus divides $4$, which is the index of any Sylow $5$-subgroup in $H$. We also know that $a$ is congruent to $1$ modulo $5$, by Sylow's theorem. Then $a=1$, that is there is only one Sylow $5$-subgroup which is therefore normal in $H$ and has index $4$. We conclude using Lemma \ref{2groups}.
%
%
\end{proof}
We shall use the following fact in the proof of the next result.
\begin{remark}\label{sympquat}
Let $\rho:G\to GL(V)$ be an irreducible symplectic representation on a complex vector space $V$ of dimension $2m$. By scalar restriction, $\rho$ defines a real representation $V_\mathbb{R}$ of $G$, of dimension $4m$. The commuting algebra 
$$D=\{f\in \mathrm{End}(V_\mathbb{R})\, :\, \textrm{$f$ is $G$-invariant}\}$$ 
clearly contains $\mathbb{C}$ (\textit{i.e.} scalar multiplications by elements of $\mathbb{C}$) and is in fact isomorphic to the division algebra of quaternions $\mathbb{H}$ (see \cite[Remarque 2, \S13.2]{Serre-reprlin}). Thus $V$ becomes a $\mathbb{H}$-vector space of dimension $m$, which we denote by $V_\mathbb{H}$, and we get a homomorphism $\rho_{_\mathbb{H}}:G\to GL(V_\mathbb{H})$ which fits in a commutative diagram
$$\xymatrix{
G \ar@{->}[rd]_{\rho_{_\mathbb{H}}} \ar@{->}[rr]^\rho
&&GL(V)\\
&GL(V_{\mathbb{H}})\ar@{^{(}->}[ru]}$$
\end{remark}

\begin{proposition}\label{<24}
Let $H$ be a group of order less than $24$ which is generated by elements of odd order. Then $H$ has no irreducible symplectic representations.
\end{proposition}
\begin{proof}
Recall that, for any irreducible representation $\rho$ of $H$, we have $\deg \rho \mid \# H$ (see \cite[Corollaire 2 to Proposition 16]{Serre-reprlin}) and $\deg \rho \leq 4$ since $(\deg \rho)^2\leq \#H<24$ (see \cite[Corollaire 2 to Proposition 5]{Serre-reprlin}). 

We argue by contradiction. Suppose that $H$ has an irreducible symplectic representation $\rho_s:H\to GL_d(\mathbb{C})$ where $d=\deg \rho_s$. Then $d=2$ or $4$, since irreducible symplectic representations have even degree, and $\#H$ is even.

Suppose that $d=4$, then $\# H\geq d^2=16$ and $d=4 \mid \# H$. This implies that $\# H=16$ or $20$. The first possibility is trivially excluded (a nontrivial $2$-group is certainly not generated by elements of odd order), while the second is ruled out by Lemma \ref{order20}. 

Then we must have $d=2$. By Remark \ref{sympquat}, this implies that $\rho_s$ factors through an homomorphism $\rho_{s,\mathbb{H}}:H\to \mathbb{H}^\times\subset GL_2(\mathbb{C})$. In particular $H$ has a quotient $\bar{H}$ isomorphic to a finite subgroup of $\mathbb{H}^\times$. The finite subgroups of $\mathbb{H}^\times$ are well-known (see for instance \cite[p. 305]{CR2}): since $\# \bar{H}< 24$, $\bar{H}$ is either cyclic or generalized quaternion. On the one hand, $\bar{H}$ cannot be cyclic, since otherwise $\rho_s$ would be the inflation of a representation of a cyclic group and would not be irreducible since $d=2$. On the other hand $\bar{H}$ cannot be isomorphic to $H_{4n}$ for any $n\geq 2$, since otherwise $\bar{H}$ (and hence $H$) would have a subgroup of index $2$ ($\langle\sigma\rangle$ has index $2$ in $H_{4n}$, with notation as in \S\ref{quaternion}), which is forbidden by Lemma \ref{2groups}.
\end{proof}

\subsubsection{}\label{btg}
We now recall the definition and some properties of the binary tetrahedral group $\tilde{A}_4$. We will then show that $\tilde{A}_4$ is the only group of order $24$ satisfying property (i) of Lemma \ref{properties} (we will see in \S \ref{tA4rep} that it also satisfies property (ii)). We refer the reader to \cite[Section 8.2]{RoseCFG} for any unproven assertion concerning $\tilde{A}_4$. This group is isomorphic to $SL_2(\mathbb{F}_3)$ and can be presented as
$$\tilde{A}_4=\langle\alpha, \beta\,|\, \alpha^3=\beta^3=(\alpha\beta)^2\rangle.$$
The center $Z(\tilde{A}_4)=\langle\alpha^3\rangle$ is the only subgroup of order $2$ of $\tilde{A}_4$ and $\tilde{A}_4/Z(\tilde{A}_4)$ is isomorphic to $A_4$, the alternating group on four elements. We thus have an exact sequence
\begin{equation}\label{fundext}
1\to Z(\tilde{A}_4)\to \tilde{A}_4\to A_4\to 1.
\end{equation}
Moreover, $\tilde{A}_4$ has no subgroup of index $2$: in particular, thanks to Lemma \ref{2groups}, it satisfies property (i) of Lemma \ref{properties} (and the above exact sequence is non-split).
The group $\tilde{A}_4$ has other properties which will be useful later: every subgroup of $\tilde{A}_4$ is cyclic, except for its Sylow $2$-subgroup, which is normal and isomorphic to the quaternion group $H_8$. For simplicity we will denote the Sylow $2$-subgroup of $\tilde{A}_4$ by $H_8$. In fact, $\tilde{A}_4$ is also isomorphic to $H_8\rtimes_{\eta} \Z/3\Z$ where $\eta:\Z/3\Z\to\mathrm{Aut}(H_8)\cong S_4$ is any nontrivial homomorphism and $S_4$ is the permutation group on four elements. 

\begin{remark}\label{repgroup} 
The extension (\ref{fundext}) is a \textit{representation group} for $A_4$, in the sense of Schur, as we now recall. Let $H$ be a finite group and let $M(H)$ denote its Schur multiplier (thus $M(H)=H^2(H,\mathbb{C}^{\times})\cong H_2(H,\mathbb{Z})$, see \cite[Definition 9.6, Chapter 2]{SuzukiGT1}). A representation group for $H$ (see \cite[Definition 9.10, Chapter 2]{SuzukiGT1}) is a central extension of $H$ by a group $Z$   
$$1\to Z\to \tilde{H} \to H\to 1$$
such that
\begin{enumerate}
\item[(RG1)] $\tilde{H}$ has no proper subgroup $H'$ such that $\tilde{H}=H'Z$;
\item[(RG2)] $\#M(H)=\#\big(Z\cap [\tilde{H},\tilde{H}]\big)$, where $[\tilde{H},\tilde{H}]$ denotes the commutator subgroup of $\tilde{H}$; 
\item[(RG3)] $\#\tilde{H}=\#H\cdot\#M(H)$.
\end{enumerate} 
One can show that if the above properties are satisfied, then $Z\cong M(H)$ (see \cite[(9.15), Chapter 2]{SuzukiGT1}). Schur showed that $M(A_4)$ has order $2$ (see \cite[(2.22), Chapter 3]{SuzukiGT1}) and it is easy to check that the extension (\ref{fundext}) is indeed a representation group for $A_4$. In fact, Schur also showed that the extension (\ref{fundext}) is the only representation group for $A_4$ up to isomorphism (see \cite[Exercise 5, Chapter 2, \S 9]{SuzukiGT1}). One often says briefly that $\tilde{A}_4$ is the representation group of $A_4$.
\end{remark}
 
\begin{proposition}\label{24}
A group of order $24$ which is generated by elements of odd order is isomorphic to $\tilde{A}_4$.
\end{proposition}
\begin{proof}
Let $\tilde H$ be a group of order $24$ which is generated by elements of odd order. We will first prove that the center $Z(\tilde H)$ of $\tilde H$ has order $2$ and $\tilde H$ fits into an exact sequence 
\begin{equation}\label{nses}
1\to Z(\tilde H)\to \tilde H\to A_4\to 1.
\end{equation}
The second step will consist in showing that the above sequence is a representation group for $A_4$, which implies in particular that $\tilde H$ is isomorphic to $\tilde{A}_4$ by Remark \ref{repgroup}. 

In order to prove the first step, we argue as in the proof of Lemma \ref{order20}. Observe that, by Sylow's theorem, $\tilde H$ acts transitively by conjugation on the set of its Sylow $3$-subgroups, which has cardinality, say, $a$. Then $a$ equals the index of the normalizer of any Sylow $3$-subgroup in $\tilde H$ and thus divides the index of any Sylow $3$-subgroup of $\tilde H$, which is $8$. We also know that $a$ is congruent to $1$ modulo $3$, by Sylow's theorem. Then $a=1$ or $4$ but the case $a=1$ is excluded since otherwise $\tilde H$ would have a normal Sylow $3$-subgroup of index a power of $2$, which is impossible by Lemma \ref{2groups}. Hence $a=4$.

Since conjugation permutes the Sylow $3$-subgroups of $\tilde H$, we get an homomorphism $\varphi:\tilde H\to S_4$ and we want to determine $\# \mathrm{ker}(\varphi)$. Observe that, by the definition of $\varphi$, $\mathrm{ker}(\varphi)$ is contained in the normalizer of any Sylow $3$-subgroup of $\tilde H$, which has index $a=4$ in $\tilde H$. Thus $\#\mathrm{ker}(\varphi)$ divides $6$. Since $\mathrm{ker}(\varphi)$ is normal, $\#\mathrm{ker}(\varphi)$ is not divisible by $3$, otherwise $\tilde H$ would have a normal subgroup of index a power of $2$, which is impossible by Lemma \ref{2groups}. The case $\#\mathrm{ker}(\varphi)=1$ is also excluded because otherwise $\tilde H\cong S_4$ and $S_4$ has a subgroup of index $2$ (namely $A_4$), again contradicting Lemma \ref{2groups}. Thus $\mathrm{ker}(\varphi)$ is a normal subgroup of order $2$ of $\tilde H$ and $\tilde H/\mathrm{ker}(\varphi)$ is isomorphic to a subgroup of index $2$ of $S_4$. It is well-known (and easy to check) that $A_4$ is the only subgroup of index $2$ of $S_4$, hence $\tilde H/\mathrm{ker}(\varphi)\cong A_4$. In other words $\tilde H$ fits into an exact sequence 
$$1\to \mathrm{ker}(\varphi)\to \tilde H\to A_4\to 1\enspace.$$
Of course $\mathrm{ker}(\varphi)$ is contained in $Z(\tilde H)$, being a normal subgroup of order $2$. Furthermore the image of $Z(\tilde H)$ in $A_4$ is trivial since $Z(A_4)$ is trivial. Hence $\mathrm{ker}(\varphi)=Z(\tilde H)$ and we get (\ref{nses}), completing the first step of the proof.

To prove the second step, we have to show that (\ref{nses}) satisfies the properties of a representation group for $A_4$ (see Remark \ref{repgroup}). Let $\tilde H'$ be a subgroup of $\tilde H$ such that $\tilde H'Z(\tilde H)=\tilde H$. Then, since $Z(\tilde H)$ is normal in $\tilde H$, we have
$$\#\tilde H'=\frac{\#\tilde H \cdot\#(\tilde H'\cap Z(\tilde H))}{\#Z(\tilde H)}=12\cdot\#(\tilde H'\cap Z(\tilde H))\enspace.$$ 
Thus $\#\tilde H'$ is either $12$ or $24$. Since the case $\#\tilde H'=12$ is excluded by Lemma \ref{2groups}, we deduce that $\tilde H$ has no proper subgroup $\tilde H'$ such that $\tilde H'Z(\tilde H)=\tilde H$. Thus $\tilde H$ satisfies property (RG1).

Observe that the surjection $\tilde H\to A_4$ induces a surjection $[\tilde H,\tilde H]\to [A_4,A_4]$ on commutator subgroups, whose kernel is $Z(\tilde H)\cap [\tilde H,\tilde H]$. In particular
$$4=\#[A_4,A_4]=\frac{\#[\tilde H,\tilde H]}{\#(Z(\tilde H)\cap[\tilde H,\tilde H])}\enspace.$$
This shows that $\#[\tilde H,\tilde H]$ is either $8$ or $4$, according to whether $Z(\tilde H)\cap[\tilde H,\tilde H]= Z(\tilde H)$ or not. We claim that $\#[\tilde H,\tilde H]$ cannot be $4$. For, if $\#[\tilde H,\tilde H]=4$, then $\tilde H$ has a quotient of order $6$ (note that $[\tilde H,\tilde H]$ is normal in $\tilde H$) and any group of order $6$ has a subgroup of index $2$. Thus, by the homomorphism theorem, $\tilde H$ has a subgroup of index $2$ which is a contradiction by Lemma \ref{2groups}. Thus $\#[\tilde H,\tilde H]=8$ and $Z(\tilde H)\cap[\tilde H,\tilde H]= Z(\tilde H)$, so that $\tilde H$ satisfies (RG2).  

Finally since $\#Z(\tilde H)=2=\#M(A_4)$, $\tilde H$ also satisfies (RG3) and hence (\ref{nses}) is a representation group for $A_4$. In particular by Schur's uniqueness result recalled in Remark \ref{repgroup}, $\tilde H\cong \tilde{A}_4$. 
\end{proof} 

\subsubsection{}\label{tA4rep}
We now want to verify that $\tilde{A}_4$ also satisfies property (ii) of Lemma \ref{properties}. Let us first recall the list and some properties of the irreducible complex characters of $\tilde{A}_4$. 

We start with the group $A_4$, which has four irreducible complex characters (see \cite[\S 5.7]{Serre-reprlin}). Three have degree $1$ and are inflated from characters of the maximal abelian quotient of $A_4$ (which is cyclic of oder $3$), while the remaining one has degree $3$ and is real. Inflating these characters to $\tilde{A}_4$, we get three characters $\chi_1$, $\chi_2$ and $\chi_3$ of degree $1$ (one of these, say $\chi_1$, is the trivial character, the other two are non-real) and a real character $\chi_4$ of degree $3$. Now $\tilde{A}_4$ has seven conjugacy classes (see \cite[Section 8.2]{RoseCFG}), so $\tilde{A}_4$ must have three more irreducible characters $\chi_5$, $\chi_6$ and $\chi_7$ (see \cite[Th\'eor\`eme 7]{Serre-reprlin}). A standard counting argument (see \cite[Corollary 2 to Proposition 5]{Serre-reprlin}) shows that these three characters all have degree $2$. Moreover, using \cite[Exercice 1, \S13.2]{Serre-reprlin} and the explicit description of the conjugacy classes of $\tilde{A}_4$ given in \cite[Section 8.2]{RoseCFG}, one can easily show that $\tilde{A}_4$ has precisely three irreducible real-valued characters. Therefore, since $\chi_1$ and $\chi_4$ are real while $\chi_2$ and $\chi_3$ are not,  one of $\chi_5$, $\chi_6$ and $\chi_7$ must be real-valued. Renumbering these characters if necessary, we may suppose that $\chi_5$ is real-valued and $\chi_6$ and $\chi_7$ are non-real. Since complex conjugation acts on the set of irreducible complex characters of $\tilde{A}_4$ of degree $2$, we must have $\chi_7=\overline{\chi_6}$.
 

The character $\chi_5$ is in fact symplectic, as we shall show in the next proposition. We first need a lemma, which will also be useful later. We are going to use some standard notation and results on complex characters (see \cite{Serre-reprlin}). For a group $H$, we denote by $(\--,\--)_{_H}$ the scalar product defined on the set of complex-valued function on $H$. Then any complex character $\chi$ of $H$ can be written uniquely as a linear combination with integral coefficients of the irreducible characters of $H$ and in fact 
\begin{equation}\label{chardec}
\chi=\sum_\rho(\chi,\rho)_{_H}\rho
\end{equation}
where the sum is taken over the set of irreducible complex character of $H$. In particular \begin{equation}\label{degree}
\deg \chi=\sum_\rho (\chi,\rho)_{_H}\deg \rho.
\end{equation}
and, if $\chi$ is the character of a representation of $H$, then $(\chi,\rho)_{_H}\geq 0$ for any irreducible character $\rho$. If $H'$ is a subgroup of $H$, $\chi'$ is a character of $H'$ and $\chi$ is a character of $H$, then we have
$$(\mathrm{Ind}_{H'}^H\chi',\chi)_H=(\chi',\mathrm{Res}_{H'}^H\chi)_{H'}$$
(Frobenius reciprocity). If $\chi_{_{H'}}$ (resp. $\chi_{_H}$) denotes the regular representation of $H'$ (resp. $H$), we have
\begin{equation}\label{indreg}
\chi_{_H}=\mathrm{Ind}_{H'}^H \chi_{_{H'}}
\end{equation} 
and in particular 
$$(\chi_{_H},\rho)_{_H}=(\chi_{_{H'}},\mathrm{Res}_{H'}^{H}\rho)_{H'}$$
for any character $\rho$ of $H$. 
Applying this to $H'=\langle\mathrm{id}\rangle$, we get
\begin{equation}\label{degreg}
(\chi_{_H},\rho)_{_H}=\deg \rho.
\end{equation}

We shall now describe the induction of an irreducible character of $H_8$ in terms of the $\chi_i$ (see Remark \ref{H8char} for notation). This will be used not only to show that $\chi_5$ is the only irreducible symplectic character of $\tilde{A}_4$ but also to compute root numbers (see Proposition \ref{WNQ}).

\begin{lemma}\label{indphi}
We have  
\begin{eqnarray*}
\mathrm{Ind}_{H_8}^{\tilde{A}_4}\psi_1&=&\chi_1+\chi_2+\chi_3\enspace,\\
\mathrm{Ind}_{H_8}^{\tilde{A}_4}\psi_i&=&\chi_4\qquad\textrm{for $i=2,3,4$}\enspace,\\
\mathrm{Ind}_{H_8}^{\tilde{A}_4}\phi&=&\chi_5+\chi_6+\overline{\chi}_6\enspace.
\end{eqnarray*}
\end{lemma}
\begin{proof}
Note that $\chi_1$, $\chi_2$ and $\chi_3$ are trivial on $H_8$ (they are inflated from the maximal abelian quotient of $\tilde{A}_4$ which is $\tilde{A}_4/H_8$). Hence if $i\in\{1,2,3\}$, then $\mathrm{Res}_{H_8}^{\tilde{A}_4}\chi_i=\psi_1$ and 
$$(\mathrm{Ind}_{H_8}^{\tilde{A}_4}\psi_1,\chi_i)_{\tilde{A}_4}=(\psi_1,\mathrm{Res}_{H_8}^{\tilde{A}_4}\chi_i)_{H_8}= 1\quad\textrm{for $i=1,2,3$}$$
by Frobenius reciprocity. Moreover $\deg\big(\mathrm{Ind}_{H_8}^{\tilde{A}_4}\psi_1\big)=[\tilde{A}_4:H]\deg\psi_1=3$ and $(\mathrm{Ind}_{H_8}^{\tilde{A}_4}\psi_1,\chi_i)_{\tilde{A}_4}\geq 0$ for any $1\leq i\leq 7$, since $\mathrm{Ind}_{H_8}^{\tilde{A}_4}\psi_1$ is the character of a representation of $\tilde{A}_4$. We deduce from (\ref{degree}) that $(\mathrm{Ind}_{H_8}^{\tilde{A}_4}\psi_1,\chi_i)_{\tilde{A}_4}=0$ if $4\leq i\leq 7$. Then we get the first equality of the lemma by (\ref{chardec}).

Next we show the second equality. We fix $i\in\{2,3,4\}$. It is sufficient to show that $\mathrm{Ind}_{H_8}^{\tilde{A}_4}\psi_i$ is irreducible, since it has degree $3$ and $\chi_4$ is the only irreducible character of $\tilde{A}_4$ of degree $3$. To prove that $\mathrm{Ind}_{H_8}^{\tilde{A}_4}\psi_i$ is irreducible, we shall use Mackey's irreducibility criterion (see \cite[Proposition 23]{Serre-reprlin}). Since $H_8$ is normal in $\tilde{A}_4$, we have to check that, for every $g\in\tilde{A}_4\setminus H_8$, $\psi_i$ is different from the representation $\psi_i^g:H_8\to \mathbb{C}^{\times}$ defined by $\psi_i^g(x)=\psi_i(g^{-1}xg)$ for every $x\in H_8$. Observe that the kernel of $\psi_i$ has order $4$ and therefore is not normal in $\tilde{A}_4$ (see \S\ref{btg}). We deduce that the normalizer of $\mathrm{ker}\psi_i$ in $\tilde{A}_4$ is $H_8$. Hence, for every $g\in\tilde{A}_4\setminus H_8$, there exists $y\in \mathrm{ker}\psi_i$ such that $g^{-1}yg\notin\ker\psi_i$ and in particular $\psi_i^g(y)=\psi_i(g^{-1}yg)\ne 1$. Thus $\psi_i$ is trivial on $y$, while $\psi_i^g$ is not. This shows that, for every $g\in\tilde{A}_4\setminus H_8$, $\psi_i^g$ and $\psi_i$ are different and hence $\mathrm{Ind}_{H_8}^{\tilde{A}_4}\psi_i$ is irreducible and equals $\chi_4$.     

To prove the last equality of the lemma, let $\chi_{_{H_8}}$ (resp. $\chi_{\tilde{A}_4}$) be the regular representation of $H_8$ (resp. $\tilde{A}_4$). Then, combining (\ref{chardec}) and (\ref{degreg}), we have
\begin{eqnarray*}
\chi_{_{H_8}}&=&\psi_1+\psi_2+\psi_3+\psi_4+2\phi,\\
\chi_{\tilde{A}_4}&=&\chi_1+\chi_2+\chi_3+3\chi_4+2\chi_5+2\chi_6+2\chi_7.
\end{eqnarray*}
Using the first two equalities of the statement of the present lemma and (\ref{indreg}), we deduce that
\begin{eqnarray*}
\chi_1+\chi_2+\chi_3+3\chi_4+2\mathrm{Ind}_{H_8}^{\tilde{A}_4}\phi
&=&\mathrm{Ind}_{H_8}^{\tilde{A}_4}\chi_{_{H_8}}\\
&=&\chi_{\tilde{A}_4}\\
&=&\chi_1+\chi_2+\chi_3+3\chi_4+2\chi_5+2\chi_6+2\chi_7
\end{eqnarray*}
and therefore $\mathrm{Ind}_{H_8}^{\tilde{A}_4}\phi=\chi_5+\chi_6+\chi_7$.
\end{proof}

\begin{proposition}\label{chi5}
The character $\chi_5$ is the only irreducible symplectic character of $\tilde{A}_4$.
\end{proposition}
\begin{proof}
Observe first that if $i\ne 5$, then $\chi_i$ is not symplectic  because either it has odd degree or it takes non real values. So we are left to prove that $\chi_5$ is symplectic. 
We will use the fact that an irreducible representation is symplectic if and only if its Frobenius-Schur indicator is $-1$ (see \cite[Proposition 38]{Serre-reprlin}). So we have to prove that 
$$\frac{1}{24}\sum_{g\in\tilde{A}_4}\chi_5(g^2)=-1\enspace.$$
Thanks to Lemma \ref{indphi} we have 
$$\frac{1}{24}\sum_{g\in\tilde{A}_4}\chi_5(g^2)=\frac{1}{24}\sum_{g\in\tilde{A}_4}\big(\mathrm{Ind}_{H_8}^{\tilde{A}_4}\phi\big)(g^2)-\frac{1}{24}\sum_{g\in\tilde{A}_4}\chi_6(g^2)-\frac{1}{24}\sum_{g\in\tilde{A}_4}\overline{\chi}_6(g^2)\enspace.$$
The terms involving $\chi_6$ and $\overline{\chi}_6$ are trivial since the Frobenius-Schur indicator of an irreducible character which takes non real values is trivial (see \cite[Proposition 38]{Serre-reprlin}). As for the term involving $\mathrm{Ind}_{H_8}^{\tilde{A}_4}\phi$, since $H_8$ is normal in $\tilde{A}_4$, the formula for the character of an induced representation (see \cite[Th\'eor\`eme 12]{Serre-reprlin}) gives $\big(\mathrm{Ind}_{H_8}^{\tilde{A}_4}\phi\big)(g^2)=0$ if $g^2\notin H_8$. Now observe that, if $g\in \tilde{A}_4$, then $g\in H_8$ if and only if $g^2\in H_8$ (since $\tilde{A}_4/H_8$ has order coprime with $2$) and, if $g\in H_8$, then
$$g^2=\left\{\begin{array}{cc}\mathrm{id}&\textrm{if $g=\mathrm{id}, z$}\\
z&\textrm{otherwise,}
\end{array}\right.$$
where $z$ is the only nontrivial square of $H_8$ (thus $z=\tau^2$ in the presentation of $H_8$ given in \S\ref{quaternion}).
We deduce that 
\begin{eqnarray*}
\sum_{g\in\tilde{A}_4}\big(\mathrm{Ind}_{H_8}^{\tilde{A}_4}\phi\big)(g^2)&=&\sum_{g\in H_8}\big(\mathrm{Ind}_{H_8}^{\tilde{A}_4}\phi\big)(g^2)\\
&=&2\big(\mathrm{Ind}_{H_8}^{\tilde{A}_4}\phi\big)(\mathrm{id})+6\big(\mathrm{Ind}_{H_8}^{\tilde{A}_4}\phi\big)(z)\enspace .
\end{eqnarray*} 
We have 
$$\big(\mathrm{Ind}_{H_8}^{\tilde{A}_4}\phi\big)(\mathrm{id})=\deg\big(\mathrm{Ind}_{H_8}^{\tilde{A}_4}\phi\big) =[\tilde{A}_4:H_8]\deg \phi=6\enspace.$$
The formula for the character of an induced representation, together with the fact that $H_8$ is normal in $\tilde{A}_4$ and $z\in Z(\tilde{A}_4)$, gives 
$$\big(\mathrm{Ind}_{H_8}^{\tilde{A}_4}\phi\big)(z)=\frac{1}{\#H_8}\sum_{g\in \tilde{A}_4}\phi(g^{-1}zg)=3\phi(z)=-6\enspace,$$
where the last equality follows from the explicit computation of the values of $\phi$ (see for instance \cite[Exercice 3, Section 12.2]{Serre-reprlin}). It follows that 
$$\sum_{g\in\tilde{A}_4}\big(\mathrm{Ind}_{H_8}^{\tilde{A}_4}\phi\big)(g^2)=-24$$
as desired.
\end{proof}

Thus $\tilde{A}_4$ satisfies the properties of Lemma \ref{properties} and has smallest possible order. In fact we also have that a tame $\tilde{A}_4$-Galois extension whose inverse different is a square of $\Q$ is automatically locally abelian, even locally cyclic.

\begin{theorem}
The group $\tilde{A}_4$ is the group of smallest order satisfying properties (i) and (ii) of Lemma \ref{properties}. Moreover, if $N/\Q$ is a tame $\tilde{A_4}$-Galois extension whose inverse different is a square, then $N/\mathbb{Q}$ is locally cyclic.
\end{theorem}
\begin{proof}
The first assertion follows by Propositions \ref{<24}, \ref{24} and \ref{chi5}. As for the last assertion, $\tilde{A}_4$ and $H_8$, the only noncyclic subgroups of $\tilde{A}_4$, cannot be the decomposition subgroup of an archimedean place of $N$, since the latter is of order dividing $2$. Suppose one of $\tilde{A}_4$ and $H_8$ is the decomposition subgroup of a finite place of $N$, then this place has to be ramified since decomposition subgroups of unramified places are cyclic. But finite primes have odd inertia degree in $N/\mathbb{Q}$, since $\C_{N/\Q}$ is a square, thus $H_8$ cannot be a decomposition subgroup. The same holds for $\tilde{A}_4$ itself: its quotient by the inertia subgroup would have to be cyclic, hence the commutator $[\tilde{A}_4,\tilde{A}_4]$  would have to be contained in the inertia subgroup. But $[\tilde{A}_4,\tilde{A}_4]=H_8$ (as we showed in the proof of Proposition \ref{24}) and therefore $\tilde{A}_4$ cannot be the decomposition group of a place of $N$. It follows that $N/\mathbb{Q}$ is locally cyclic. 
\end{proof}

\subsubsection{}
We now explicitly describe a tame $\tilde{A}_4$-Galois extension $N/\mathbb{Q}$ whose inverse different is a square. We use the results of Bachoc and Kwon \cite{BachocKwon}, who studied the embedding problem of $A_4$-extensions in $\tilde{A}_4$-extensions. 

We briefly recall how the $\tilde{A}_4$-Galois extension we are interested in is constructed, although this construction is not strictly necessary for us. We begin with the polynomial
$$X^4-2X^3-7X^2+3X+8$$
which is irreducible over $\Q$ and let $\gamma$ be any fixed root of it (in an algebraic closure of $\Q$). Then, up to conjugation, $K=\mathbb{Q}(\gamma)$ is the totally real field of degree four over $\Q$ of smallest discriminant having trivial class number and Galois closure with Galois group isomorphic to $A_4$ (see \cite[tables at pp. 395-396]{BPS}). An easy computation with PARI \cite{PARI2} reveals that the Galois closure $M/\mathbb{Q}$ of $K$ is explicitly given by the polynomial
$$
X^{12} - 23X^{10} + 125X^8 - 231X^6 + 125X^4 - 23X^2 + 1\enspace.
$$
A further computation gives that the discriminant of $M/\mathbb{Q}$ is $163^8$ (in particular $M/\Q$ is tame) and the four primes above $163$ in $M/\mathbb{Q}$ have ramification index $3$ (in particular $\C_{M/\mathbb{Q}}$ is a square). 
Let $k$ denote the only degree $3$ subextension of $M/\Q$, it follows that $k/\Q$ is totally ramified above $163$ and unramified elsewhere, and $M/k$ is unramified at every finite place.

Using again PARI, we get that the narrow class number of $K$ equals $2$. In other words the narrow Hilbert class field of $K$, namely its maximal abelian extension which is unramified at finite places, is of degree $2$ over $K$. Moreover $K/\Q$ has the same discriminant as $k/\Q$.
Therefore $K$ satisfies the hypotheses of \cite[Proposition 3.1(1)]{BachocKwon} and there exists a unique number field $\tilde{K}$ of degree $8$ over $\Q$ with the following properties:
\begin{itemize}
\item $K\subset \tilde{K}$ and $\tilde{K}/K$ is unramified outside the primes ramifying in $M/k$ (\text{i.e.} $\tilde{K}$ is a \textit{pure} embedding of $K$);
\item the Galois closure $N/\mathbb{Q}$ of $\tilde{K}$ has Galois group $\tilde{A}_4$.
\end{itemize}
It follows from $K\subset \tilde{K}$ and the definition of $M$ and $N$ that $M\subset N$ (see Figure \ref{ext-dia}).
From the above we get that $\tilde{K}/K$ is unramified at every finite place, thus the same holds for $N/M$ and for $N/k$. This shows that $N/\mathbb{Q}$ is tame and that its inverse different is a square (since the same holds for $k/\mathbb{Q}$).

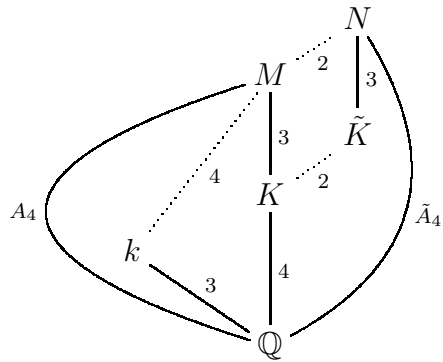
\begin{figure}[ht]
{
$$\xymatrix@=7pt@C15pt{
&&& N \\
&& M \ar@{.}@/0pc/[ur]_2 &\\
&&& \tilde{K} \ar@{-}@/0pc/[uu]_3 \\
&& K \ar@{-}@/0pc/[uu]_3 \ar@{.}@/0pc/[ur]_2 &\\
k \ar@{.}@/0pc/[uuurr]_4 &&&\\
&&&\\
&& \Q\ar@{-}@/0pc/[uuu]_4 \ar@{-}@/0pc/[uull]_3 \ar@/_3pc/@{-}[uuuuuur]_{\tilde{A}_4} \ar@/^7pc/@{-}[uuuuu]^{{A}_4} &
}$$
}
\caption{Extensions diagram}
\label{ext-dia}
\end{figure}

Since the class number of $K$ is trivial, $\tilde{K}/K$ is ramified at some archimedean place (in fact $\tilde{K}$ is the narrow Hilbert class field of $K$). In particular $\tilde{K}$ is not totally real and therefore $N$ is totally complex (being a non-totally real Galois extension of $\Q$). In other words the archimedean place of $\Q$ is ramified in $N$, so the extension $N/\Q$ does not satisfy the hypotheses of Proposition \ref{infnonram}.  

Bachoc and Kwon also compute a polynomial defining $\tilde{K}$ (see \cite[table of p. 9, first line]{BachocKwon}): 
$$X^8+14X^6+23X^4+9X^2+1.$$
With the help of PARI, we get that $N/\mathbb{Q}$ is explicitly given by the polynomial
$$\begin{array}{l}
X^{24} - 3X^{23} - 2X^{22} + 16X^{21} - 12X^{20} + 52X^{19} - 324X^{18} - 436X^{17} +\\ 3810X^{16} - 1638X^{15} - 8012X^{14} - 12988X^{13}+ 67224X^{12} - 76152X^{11} +\\ 41175X^{10} - 39587X^9 + 70068X^8 - 66440X^7 + 38488X^6 - 23248X^5\\ +16672X^4 - 6976X^3 + 2816X^2 - 1280X + 512.
\end{array}$$


\subsubsection{}
We now want to verify that $t_{\tilde{A}_4}W_{N/\Q}\in\Cl(\Z[\tilde{A}_4])$ is nontrivial. We first show that $W_{N/\Q}\in\mathrm{Hom}_{\Omega_\Q}(S_{\tilde{A}_4},\{\pm 1\})$ is not trivial. 

\begin{proposition}\label{WNQ} 
We have $W(N/\mathbb{Q},\chi_5)=-1$, \textit{i.e.} $W_{N/\Q}$ is not trivial.
\end{proposition}
\begin{proof}
By Lemma \ref{indphi} and the properties of root numbers recalled in \S \ref{rnp}, we have
\begin{equation}\label{rndec}
W(N/\mathbb{Q},\chi_5)=W(N/k,\phi)W(N/\mathbb{Q},\chi_6)^{-1}W(N/\mathbb{Q},\overline{\chi}_6)^{-1}\enspace.
\end{equation}

We first show that $W(N/k,\phi)=-1$. Since $k/\Q$ is cyclic of degree $3$, it is totally real with three real places and $\Gal(N/k)=H_8$, the unique Sylow $2$-subgroup of $\tilde{A}_4$. We will show that 
$$
W(N_{\mathcal{P}}/k_P),\phi_{_{D_\PP}})=\left\{\begin{array}{ll}-1& \textrm{if $P$ is a real place of $k$}\\1&\textrm{otherwise}\end{array}\right.
$$
which implies, by (\ref{locdec}), that $W(N/k,\phi)=-1$ (recall that $\phi_{D_\PP}$ denotes the restriction of $\phi$ at $\PP$).
If $P$ is a finite prime of $k$, then one shows the triviality of $W(N_{\mathcal{P}}/k_P,\phi_{_{D_\PP}})$ as in the proof of Proposition \ref{infnonram} (case (a)). Suppose that $P$ is a real place of $k$, then the decomposition subgroup $D_\PP$ of a place $\mathcal{P}$ of $N$ above $P$ is cyclic of order $2$, since $N$ is totally imaginary, and therefore $D_\PP=Z(\tilde{A}_4)=\langle\alpha^3\rangle$ since the center is the only subgroup of order $2$ of $\tilde{A}_4$. In particular, by kemma \ref{quatrescyclic}, we must have  $\phi_{_{D_\PP}}=\nu+\overline{\nu}$ where $\nu:D_\PP\to \{\pm 1\}$ is a character (thus in fact $\nu=\overline{\nu}$ and $\phi_{_{D_\PP}}=2\nu$). Actually, since $\phi(\alpha^3)=-2$ (as remarked in the proof of kemma \ref{indphi}), $\nu$ must be the sign character, \textit{i.e.} the only nontrivial character of $D_\PP$. We therefore obtain  
$$W(N_{\mathcal{P}}/k_P,\phi_{_{D_\PP}})=W(N_{\mathcal{P}}/k_P,\nu)W(N_{\mathcal{P}}/k_P,\overline{\nu})=\mathrm{det}_{\nu}(-1)=\nu\circ r_P(-1)=-1,$$
where as usual $r_P: k_P^\times=\mathbb{R}^\times\to D_\PP$ is the local reciprocity map at $P$ (which is nontrivial on $-1$ precisely because $P$ is ramified in $N/k$).   

As for the factor $W(N/\mathbb{Q},\chi_6)^{-1}W(N/\mathbb{Q},\overline{\chi}_6)^{-1}$ in (\ref{rndec}), using (\ref{locdec}) as above, we are reduced to the local setting (here local means corresponding to a place of $\mathbb{Q}$). Again, if $\PP$ is a place of $N$, either finite or archimedean,
$$W(N/\mathbb{Q},(\chi_6)_{_{D_\PP}})W(N/\mathbb{Q},(\overline{\chi}_6)_{_{D_\PP}})=\det{}_{(\chi_{_6})_{_{D_\PP}}}(-1)\enspace.$$
Observe that the determinant of $\chi_6$ is trivial on $H_8$ (which is the commutator subgroup of $\tilde{A}_4$). If $\PP$ is archimedean, $D_\PP$ is cyclic of order $2$ and therefore $D_\PP\subset H_8$, since $H_8$ is the Sylow $2$-subgroup of $\tilde{A}_4$. In particular $\det{}_{(\chi_{_6})_{_{D_\PP}}}=1$. If instead $\PP$ lies above a finite rational prime $p$, then, as in the proof of Proposition \ref{infnonram}, the reciprocity map $r_p:\Q_p^\times\to D_\PP$ is trivial on $-1$, since $r_p(-1)$ is of order dividing $2$ and belongs to the inertia subgroup $I_\PP$ which has odd order.  In particular $\det{}_{(\chi_{_6})_{_{D_\PP}}}(-1)=1$ also in this case.
\end{proof}

Concerning the map $t_{\tilde{A}_4}$, we have the following result.

\begin{lemma}\label{tiso}
The map
$$t_{\tilde{A}_4}:\mathrm{Hom}_{\Omega_\mathbb{Q}}(S_{\tilde{A}_4},\{\pm 1\})\to\Cl(\mathbb{Z}[\tilde{A}_4])$$ 
is an isomorphism between groups of order $2$.
\end{lemma}
\begin{proof}
Consider the following diagram
\begin{equation}\label{tres}
\begin{CD}
\mathrm{Hom}_{\Omega_\mathbb{Q}}(S_{\tilde{A}_4},\{\pm 1\})@>t_{\tilde{A}_4}>>\Cl(\mathbb{Z}[\tilde{A}_4])\\
@V\varphi VV @VVresV\\
\mathrm{Hom}_{\Omega_\mathbb{Q}}(S_{H_8},\{\pm 1\})@>t_{H_8}>>\Cl(\mathbb{Z}[H_8])\\
\end{CD}
\end{equation}
where $res$ is induced by restriction of scalars from $\Z[\tilde{A}_4]$ to $\Z[H_8]$. To define $\varphi$, observe that each of $\tilde{A}_4$ and $H_8$ has precisely one irreducible symplectic representation ($\chi_5$ and $\phi$, respectively). This means that $S_{\tilde{A}_4}$ and $S_{H_8}$ both have $\Z$-rank $1$ and $\Omega_{\Q}$ acts trivially on them. Of course $\Omega_{\Q}$ acts trivially on $\{\pm1\}$ too, so that $\mathrm{Hom}_{\Omega_\mathbb{Q}}(S_{\tilde{A}_4},\{\pm 1\})$ and $\mathrm{Hom}_{\Omega_\mathbb{Q}}(S_{H_8},\{\pm 1\})$ both have order $2$. Then we let $\varphi$ be the only isomorphism between $\mathrm{Hom}_{\Omega_\mathbb{Q}}(S_{\tilde{A}_4},\{\pm 1\})$ and $\mathrm{Hom}_{\Omega_\mathbb{Q}}(S_{H_8},\{\pm 1\})$. In particular
$$\varphi(f)(\phi)=f(\chi_5)$$  
for $f\in\mathrm{Hom}_{\Omega_\mathbb{Q}}(S_{\tilde{A}_4},\{\pm 1\})$.

We claim that the above diagram is commutative. First observe that, thanks to \cite[Theorem 12]{Frohlich-Alg_numb}, we have a commutative diagram 
$$
\begin{CD}
\mathrm{Hom}_{\Omega_\mathbb{Q}}(R_{\tilde{A}_4},J(L'))@>>>\Cl(\Z[\tilde{A}_4])\\
@Vres VV @VVresV\\
\mathrm{Hom}_{\Omega_\mathbb{Q}}(R_{H_8},J(L'))@>>>\Cl(\Z[H_8])\\
\end{CD}
$$
where $L'$ is a large enough number field, the horizontal arrows are the projections induced by the Hom-description and the map $res$ on the left satisfies $res(g)(\chi)=g(\mathrm{Ind}_{H_8}^{\tilde{A}_4}\chi)$ for any $g\in \mathrm{Hom}_{\Omega_\mathbb{Q}}(R_{\tilde{A}_4},J(L'))$ and any character $\chi$ of $H_8$.
Thus it is sufficient to prove that the diagram
$$
\begin{CD}
\mathrm{Hom}_{\Omega_\mathbb{Q}}(S_{\tilde{A}_4},\{\pm 1\})@>t'_{\tilde{A}_4}>>\mathrm{Hom}_{\Omega_\mathbb{Q}}(R_{\tilde{A}_4},J(L'))\\
@V\varphi VV @VVresV\\
\mathrm{Hom}_{\Omega_\mathbb{Q}}(S_{H_8},\{\pm 1\})@>t'_{H_8}>>\mathrm{Hom}_{\Omega_\mathbb{Q}}(R_{H_8},J(L'))\\
\end{CD}
$$ 
is commutative (the definition of $t'_{\tilde{A}_4}$ and $t'_{H_8}$ is recalled in \S\ref{sympltG}). Take $f\in\mathrm{Hom}_{\Omega_\mathbb{Q}}(S_{\tilde{A}_4},\{\pm 1\})$, then, on the one hand, by Lemma \ref{indphi} we have
$$res(t'_{\tilde{A}_4}(f))(\chi)=t'_{\tilde{A}_4}(f)(\mathrm{Ind}_{H_8}^{\tilde{A}_4}\chi)=\left\{\begin{array}{ll}t'_{\tilde{A}_4}(f)(\chi_5+\chi_6+\chi_7)&\textrm{if $\chi=\phi$}\\
t'_{\tilde{A}_4}(f)(\chi_1+\chi_2+\chi_3)&\textrm{if $\chi=\psi_1$}\\
t'_{\tilde{A}_4}(f)(\chi_4)&\textrm{otherwise,}\end{array}\right.$$ 
for any irreducible character $\chi$ of $H_8$. In particular, using the definition of $t'_{\tilde{A}_4}$ and Proposition \ref{chi5}, we get, for every place $\mathfrak{l}$ of $L'$,
$$res(t'_{\tilde{A}_4}(f))(\chi)_\mathfrak{l}=\left\{\begin{array}{ll}f(\chi_5)&\textrm{if $\mathfrak{l}$ is finite and $\chi=\phi$}\\
1&\textrm{otherwise.}\end{array}\right.$$
On the other hand we have
\begin{eqnarray*}
t'_{H_8}(\varphi(f))(\chi)_\mathfrak{l}&=&\left\{\begin{array}{ll}\varphi(f)(\chi)&\textrm{if $\mathfrak{l}$ is finite and $\chi=\phi$}\\
1&\textrm{otherwise}\end{array}\right.\\&=&\left\{\begin{array}{ll}f(\chi_5)&\textrm{if $\mathfrak{l}$ is finite and $\chi=\phi$}\\
1&\textrm{otherwise},\end{array}\right.
\end{eqnarray*}
and we have proved our claim.

Now the bottom horizontal arrow of (\ref{tres}) is an isomorphism, by a result of Fr\"ohlich (see \cite[I, Proposition 7.2]{Frohlich-Alg_numb}). Swan showed that the right-hand vertical arrow of (\ref{tres}) is also an isomorphism  (see \cite[Theorem 14.1]{SwanPMOBPG}). Thus $t_{\tilde{A}_4}$ is an isomorphism and all groups have order $2$.
\end{proof}

Combining the above lemma (in fact just the injectivity of $t_{\tilde{A}_4}$) with Proposition \ref{WNQ} and Corollary \ref{A=tW}, we get the main result of this section and prove Theorem \ref{surprise} of the Introduction.

\begin{theorem}
For the above $\tilde{A}_4$-extension $N/\mathbb{Q}$ one has $t_{\tilde{A}_4}W_{N/\Q}\ne 1$ in $\Cl(\Z[\tilde{A}_4])$. In particular the classes of $\A_{N/\mathbb{Q}}$ and $\OO_{N}$ are both equal to the nontrivial element in $\Cl(\Z[\tilde{A}_4])$. 
\end{theorem}

\begin{remark}
Recall that $N/k$ is unramified at every finite place and that $\mathrm{Gal}(N/k)=H_8$. Therefore $\A_{N/k}=\OO_N$ and in particular $(\A_{N/k})=(\OO_N)$ in $\Cl(\Z[H_8])$. Note that, in the proof of Proposition \ref{chi5}, we have shown that $W_{N/k}$ is nontrivial. Moreover, since $t_{H_8}$ is an isomorphism by \cite[I, Proposition 7.2]{Frohlich-Alg_numb}, we get that $t_{H_8}W_{N/k}$ is the nontrivial element of $\Cl(\Z[H_8])$. Thus $N/k$ gives another example of a tame (in fact unramified at finite places) Galois extension whose codifferent is a square and the square root of the codifferent has nontrivial class in the locally free class group. Anyway, the case of $N/\Q$ is perhaps more suggestive, since $\A_{N/\mathbb{Q}}\not=\OO_{N}$.
\end{remark}

%

\bibliography{bib}
\end{document}